\documentclass[12pt]{amsart}

\usepackage{amsmath}
\usepackage{amsfonts}
\usepackage{amssymb}
\usepackage{graphicx}
\usepackage{hyperref}
\usepackage{mathrsfs}
\usepackage{pb-diagram}
\usepackage{epstopdf}
\usepackage{amsmath,amsfonts,amsthm,enumerate,amscd,latexsym,curves}
\usepackage{bbm}
\usepackage[mathscr]{eucal}
\usepackage{epsfig,epsf,xypic,epic}
\usepackage{caption}
\usepackage{subcaption}
\usepackage{verbatim}

\newtheorem{thm}{Theorem}[section]
\newtheorem{lemma}[thm]{Lemma}
\newtheorem{corollary}[thm]{Corollary}
\newtheorem{prop}[thm]{Proposition}
\newtheorem{defn}[thm]{Definition}
\newtheorem{example}[thm]{Example}
\newtheorem{remark}[thm]{Remark}

\newtheorem{assumption}[thm]{Assumption}
\newtheorem{setting}[thm]{Setting}
\newtheorem{construction}[thm]{Construction}

\numberwithin{equation}{section}

\newcommand{\der}{\mathrm{d}}

\newcommand{\pairing}[2]{\left( #1\, , \, #2 \right)}

\newcommand{\nat}{\mathbb{N}}
\newcommand{\integer}{\mathbb{Z}}
\newcommand{\rat}{\mathbb{Q}}
\newcommand{\real}{\mathbb{R}}
\newcommand{\N}{\mathbb{N}}
\newcommand{\Z}{\mathbb{Z}}
\newcommand{\Q}{\mathbb{Q}}
\newcommand{\R}{\mathbb{R}}
\newcommand{\C}{\mathbb{C}}
\newcommand{\cpx}{\mathbb{C}}

\newcommand{\bb}{\boldsymbol{b}}

\newcommand{\bx}{\boldsymbol{x}}
\newcommand{\bX}{\mathcal{X}}

\newcommand{\unu}{\underline{\nu}}

\newcommand{\conste}{\mathbf{e}}

\newcommand{\WT}[1]{\widetilde{#1}}
\newcommand{\UL}[1]{\underline{#1}}
\newcommand{\proj}{\mathbb{P}}

\newcommand{\CM}{\mathcal{M}}

\newcommand{\bD}{\mathcal{D}}

\newcommand{\pt}{\mathrm{pt}}

\begin{document}

\title[Gross fibrations, SYZ, and open GW for toric CY orbifolds]{Gross fibrations, SYZ mirror symmetry, \\and open Gromov-Witten invariants\\ for toric Calabi-Yau orbifolds}
\author[Chan]{Kwokwai Chan}
\address{Department of Mathematics\\ The Chinese University of Hong Kong\\ Shatin\\ Hong Kong}
\email{kwchan@math.cuhk.edu.hk}
\author[Cho]{Cheol-Hyun Cho}
\address{Department of Mathematical Sciences, Research institute of Mathematics\\ Seoul National University\\ San 56-1, Shinrimdong\\ Gwanakgu \\Seoul 47907\\ Korea}
\email{chocheol@snu.ac.kr}
\author[Lau]{Siu-Cheong Lau}
\address{Department of Mathematics\\ Harvard University\\ One Oxford Street\\ Cambridge \\ MA 02138\\ USA}
\email{s.lau@math.harvard.edu}
\author[Tseng]{Hsian-Hua Tseng}
\address{Department of Mathematics\\ Ohio State University\\ 100 Math Tower, 231 West 18th Ave. \\ Columbus \\ OH 43210\\ USA}
\email{hhtseng@math.ohio-state.edu}


\begin{abstract}
For a toric Calabi-Yau (CY) orbifold $\bX$ whose underlying toric variety is semi-projective, we construct and study a non-toric Lagrangian torus fibration on $\bX$, which we call the Gross fibration. We apply the Strominger-Yau-Zaslow (SYZ) recipe to the Gross fibration of $\bX$ to construct its mirror with the instanton corrections coming from genus 0 open orbifold Gromov-Witten (GW) invariants, which are virtual counts of holomorphic orbi-disks in $\bX$ bounded by fibers of the Gross fibration.

We explicitly evaluate all these invariants by first proving an open/closed equality and then employing the toric mirror theorem for suitable toric (partial) compactifications of $\bX$. Our calculations are then applied to
\begin{itemize}
\item[(1)]
prove a conjecture of Gross-Siebert on a relation between genus 0 open orbifold GW invariants and mirror maps of $\bX$ -- this is called the open mirror theorem, which leads to an enumerative meaning of mirror maps, and
\item[(2)]
demonstrate how open (orbifold) GW invariants for toric CY orbifolds change under toric crepant resolutions -- an open analogue of Ruan's crepant resolution conjecture.
\end{itemize}
\end{abstract}

\maketitle


\section{Introduction}

\subsection{SYZ mirror construction}
In 1996, Strominger-Yau-Zaslow \cite{syz96} proposed an intrinsic and geometric way to understand mirror symmetry for Calabi-Yau (CY) manifolds via $T$-duality. Roughly speaking, the Strominger-Yau-Zaslow (SYZ) conjecture asserts that a mirror pair of CY manifolds $X$ and $\check{X}$ admit fiberwise dual special Lagrangian torus fibrations.

Mathematical approaches to SYZ mirror symmetry have since been extensively studied by many researchers including Kontsevich-Soibelman \cite{kontsevich00, kontsevich-soibelman04}, Leung-Yau-Zaslow \cite{LYZ}, Leung \cite{leung01}, Gross-Siebert \cite{Gross-Siebert03, Gross-Siebert06, Gross-Siebert10, gross07}, Auroux \cite{auroux07,auroux09}, Chan-Leung \cite{Chan-Leung, Chan-Leung2}, Chan-Lau-Leung \cite{CLL} and Abouzaid-Auroux-Katzarkov \cite{AAK12}.

A very important application of the SYZ conjecture is to provide a geometric construction of mirrors: Given a CY manifold $X$, a mirror $\check{X}$ can be obtained by finding a (special) Lagrangian torus fibration $X\to B$ and suitably modifying the complex structure of the total space of the fiberwise dual by instanton corrections. For toric CY manifolds, Gross \cite{gross_examples} (and independently Goldstein \cite{goldstein}) constructed such a special Lagrangian torus fibration which we call the {\em Gross fibration}. In \cite{CLL}, the SYZ construction was applied to the Gross fibration to produce an instanton-corrected mirror family of a toric CY manifold, following the Floer-theoretic approach pioneered by Auroux \cite{auroux07, auroux09}.

In this paper we consider the SYZ construction for toric CY {\em orbifolds}. A toric CY orbifold is a (necessarily non-compact) Gorenstein toric orbifold $\bX$ whose canonical line bundle $K_\bX$ is trivial. We also assume that the coarse moduli space of $\bX$ is a {\em semi-projective} toric variety, or equivalently, that $\bX$ is as in Setting \ref{setting:toricCY}.

Following \cite{gross_examples}, we define in Definition \ref{defn:gross_fibration} a special Lagrangian torus fibration
$\mu: \bX \to B$
which we again call the Gross fibration of $\bX$. As in the manifold case, the discriminant locus $\Gamma\subset B$ can be described explicitly. $\Gamma$ is a real codimension 2 subset contained in a hyperplane which we call the {\em wall} in the base $B$. The wall divides the smooth locus $B_0 = B \setminus \Gamma$ into two chambers $B_+$ and $B_-$. Over $B_0$, the fibration $\mu$ restricts to a torus bundle $\mu: \bX_0\to B_0$, and the dual torus bundle
$\check{\mu}: \check{\bX}_0\to B_0$
admits a natural complex structure, producing the so-called {\em semi-flat mirror} of $\bX$.

This does not give the genuine mirror for $\bX$ because the semi-flat complex structure {\em cannot} be extended further to any partial compactification of $\check{\bX}_0$, due to nontrivial monodromy of the affine structure around the discriminant locus $\Gamma$. According to the SYZ proposal, we should deform the semi-flat complex structure by instanton corrections so that it becomes extendable. More concretely, what we do is to modify the gluing between the complex charts over the chambers $B_+$ and $B_-$ by {\em wall-crossing formulas} for genus 0 open orbifold GW invariants of $\bX$ (cf. the manifold case \cite{auroux07, auroux09, CLL, AAK12}). The latter are virtual counts of holomorphic orbi-disks in $\bX$ with boundary lying on regular fibers of $\mu$. A suitable partial compactification then yields the following instanton-corrected mirror, or {\em SYZ mirror}, of $\bX$:

\begin{thm}[See Section \ref{subsec:mirror}]
Let $\bX$ be a toric CY orbifold as in Setting \ref{setting:toricCY} and equipped with the Gross fibration in Definition \ref{defn:gross_fibration}. Then the SYZ mirror of $\bX$ (with a hypersurface removed) is the family of non-compact CY manifolds
\begin{equation*}
\check{\bX}:=\{(u, v, z_1, \ldots, z_{n-1})\in \C^2\times (\C^\times)^{n-1} \mid uv = g(z_1, \ldots, z_{n-1})\},
\end{equation*}
where the defining equation $uv=g$ is given by
\begin{equation*}
uv = (1+\delta_0) + \sum_{j=1}^{n-1}(1+\delta_j)z_j + \sum_{j=n}^{m-1}(1+\delta_j)q_j z^{\bb_j}
+ \sum_{\nu\in \mathrm{Box}'(\Sigma)^{\mathrm{age}=1}} (\tau_\nu+\delta_\nu)q^{-D^{\vee}_\nu} z^\nu.
\end{equation*}
Here $1+\delta_j$ and $\tau_\nu+\delta_\nu$ are generating functions of orbi-disk invariants of $(\bX, F_r)$ (see Section \ref{sec:instanton_corrections} for the reasons why the generating functions are of these forms).
\end{thm}

\begin{remark}
\hfill
\begin{enumerate}
\item
The SYZ mirror of the toric CY orbifold $\bX$, without removing a hypersurface, is given by the {\em Landau-Ginzburg model} $(\check{\bX},W)$ where $W:\check{\bX} \to \C$ is the holomorphic function $W := u$; this is exactly like the manifold case as discussed in \cite[Section 4.6]{CLL} and \cite[Section 7]{AAK12}.

\item
Section \ref{sec:examples_mirror} contains several explicit examples. For instance, let $\kappa_j$ be explicitly given by \eqref{eqn:kappas}. Then the mirror of $\bX=[\cpx^2/\Z_m]$ is given by the equation
$uv = \prod_{j=0}^{m-1} ( z - \kappa_j)$.
\end{enumerate}
\end{remark}

To the best of our knowledge, this is the first time the SYZ construction is applied systematically to construct mirrors for {\em orbifolds}.

\subsection{Orbi-disk invariants}
To demonstrate that $\check{\bX}$ is indeed mirror to $\bX$, we would like to show that the family $\check{\bX}$ is written in {\em canonical coordinates}. This can be rephrased as the conjecture that the SYZ map, defined in terms of orbi-disk invariants, is inverse to the toric mirror map of $\bX$ (cf. \cite[Conjecture 0.2]{gross07}, \cite[Conjecture 1.1]{CLL} and \cite[Conjecture 2]{CLT11}). To prove this, knowledge about the orbi-disk invariants is absolutely crucial.

One major advance of this paper is the {\em complete} calculation of these orbi-disk invariants, or genus 0 open orbifold GW invariants, for moment-map Lagrangian torus fibers in toric CY orbifolds. Our calculation is based on the following {\em open/closed equality}:

\begin{thm}[See Theorem \ref{thm:open_closed_equality} and Equation \eqref{eqn:open_closed_equality}]
Let $\bX$ be a toric CY orbifold as in Setting \ref{setting:toricCY} and equipped with a toric K\"ahler structure. Let $L \subset \bX$ be a Lagrangian torus fiber of the moment map of $\bX$, and let $\beta\in \pi_2(\bX, L)$ be a holomorphic (orbi-)disk class of Chern-Weil (CW) Maslov index 2. Let $\bar{\bX}$ be the toric partial compactification of $\bX$ constructed in Construction \ref{construction:compactification} which depends on $\beta$. Then we have the following equality between genus 0 open orbifold GW invariants of $(\bX,L)$ and closed orbifold GW invariants of $\bar{\bX}$:
\begin{equation}\label{eqn:open_closed_equality_intro}
n_{1,l,\beta}^\bX([\mathrm{pt}]_L; \mathbf{1}_{\nu_1},\ldots,\mathbf{1}_{\nu_l})=\langle[\mathrm{pt}], \mathbf{1}_{\bar{\nu}_1},\ldots,\mathbf{1}_{\bar{\nu}_l} \rangle_{0,1+l, \bar{\beta}}^{\bar{\bX}}.
\end{equation}
\end{thm}
The proof is by showing that the relevant moduli space of stable (orbi-)disks in $\bX$ is isomorphic to the relevant moduli space of stable maps to $\bar{\bX}$ {\em as Kuranishi spaces}. The key geometric ingredients underlying the proof are that the toric compactification $\bar{\bX}$ is constructed so that (orbi-)disks in $\bX$ can be ``capped off'' in $\bar{\bX}$ to obtain (orbi-)spheres, and that the deformations and obstructions of the two moduli problems can naturally be identified.

The closed orbifold GW invariants of $\bar{\bX}$ in \eqref{eqn:open_closed_equality_intro} are encoded in the $J$-function of $\bar{\bX}$. Evaluating these invariants via the toric mirror theorem requires extra care since $\bar{\bX}$ may be noncompact. Fortunately, $\bar{\bX}$ is {\em semi-Fano} (see Definition \ref{defn:sF}) and semi-projective, so the {\em equivariant} toric mirror theorem of \cite{CCIT_toricDM} still applies to give an explicit formula for the equivariant $J$-function of $\bar{\bX}$. Extracting the relevant equivariant closed orbifold GW invariants from this formula and taking non-equivariant limits then yield explicit formulas for the genus 0 open orbifold GW invariants of $\bX$ and hence the generating functions which appear in the defining equation of the SYZ mirror $\check{\bX}$:

\begin{thm}[See Theorems \ref{thm:sm_disk_gen_function} and \ref{thm:orbi_disk_gen_function}]\label{thm:disk_calc_intro}
Let $\bX$ be a toric CY orbifold as in Setting \ref{setting:toricCY}. Let $F_r$ be a Lagrangian torus fiber of the Gross fibration of $\bX$ lying above a point $r$ in the chamber $B_+ \subset B_0$. Let the functions $A^\bX_i(y)$'s be given explicitly in \eqref{eqn:toric_mirror_map_revised1}.
\begin{enumerate}
\item
Let $1+\delta_i$ be the generating function of genus 0 open orbifold GW invariants of $\bX$ in classes $\beta_i(r)+\alpha$, with $\alpha\in H_2^\textrm{eff}(\bX)$ satisfying $c_1(\bX)\cdot \alpha=0$ and $\beta_i(r)\in \pi_2(\bX, F_r)$ the basic smooth disk class corresponding to the primitive generator $\bb_i$ of a ray in $\Sigma$. Then
$$1+\delta_i=\exp\left(-A^\bX_i(y)\right),$$
after inverting the toric mirror map \eqref{eqn:toric_mirror_map_X}.

\item
Let $\tau_\nu+\delta_\nu$ be the generating function of genus 0 open orbifold GW invariants of $\bX$ in classes $\beta_\nu(r)+\alpha$, with $\alpha\in H_2^\textrm{eff}(\bX)$ satisfying $c_1(\bX)\cdot\alpha=0$ and $\beta_\nu(r)\in \pi_2(\bX, F_r)$ the basic orbi-disk class corresponding to a Box element $\nu$ of age $1$. Then
$$\tau_\nu+\delta_\nu=y^{D_\nu^\vee}\exp\left(-\sum_{i\notin I_\nu}c_{\nu i}A^\bX_i(y)\right),$$
after inverting the toric mirror map \eqref{eqn:toric_mirror_map_X}.
\end{enumerate}
\end{thm}

These generalize results in \cite{CLT11} to all semi-projective toric CY orbifolds, including the toric CY 3-fold $\bX = K_{\mathbb{F}_2}$ which cannot be handled by \cite{CLT11} (see Example (4) in Section \ref{sec:examples_mirror}).

\subsection{Applications}
Our explicit calculations in Theorem \ref{thm:disk_calc_intro} has two major applications.

\subsubsection{Open mirror theorems}
The first application, as we mentioned above, is to show that the SYZ mirror family $\check{\bX}$ is written in canonical coordinates. This concerns the comparison of several mirror maps for a toric CY orbifold $\bX$. More precisely, the SYZ construction yields what we call the {\em SYZ map} $\mathcal{F}^\mathrm{SYZ}$, defined in terms of genus 0 open orbifold GW invariants (see the precise definition in \eqref{eqn:SYZ_map}). In closed GW theory, the toric mirror theorem of \cite{CCIT_toricDM} involves a combinatorially defined {\em toric mirror map} $\mathcal{F}^\mathrm{mirror}$ (see Section 2.4 and \eqref{eqn:toric_mirror_map_X}). We prove the following {\em open mirror theorem}:
\begin{thm}[Open mirror theorem for toric CY orbifolds - Version 1]\label{thm:open_mirror}
For a toric CY orbifold $\bX$ as in Setting \ref{setting:toricCY}, the SYZ map is inverse to the toric mirror map, i.e. we have
\begin{equation*}
\mathcal{F}^{\mathrm{SYZ}} = \left(\mathcal{F}^\mathrm{mirror}\right)^{-1}
\end{equation*}
near the large volume limit $(q, \tau)=0$ of $\bX$. In particular, this holds for a semi-projective toric CY manifold.
\end{thm}

We remark that an open mirror theorem was proved for compact semi-Fano toric manifolds in \cite{CLLT12} and some examples of compact semi-Fano toric orbifolds in \cite{CCLT12}. On the other hand, open mirror theorems for 3-dimensional toric CY geometries relative to Aganagic-Vafa type Lagrangian branes were proved in various degrees of generality in \cite{graber-zaslow01, brini-cavalieri11, fang-liu, fang-liu-tseng}.

By combining the above open mirror theorem together with the analysis of relations between {\em period integrals} and the {\em GKZ hypergeometric system} associated to $\bX$ done in \cite{CLT11}, we obtain another version of the open mirror theorem, linking the SYZ map to period integrals:
\begin{thm}[Open mirror theorem for toric CY orbifolds - Version 2]\label{thm:period_mirror}
For a toric CY orbifold $\bX$ as in Setting \ref{setting:toricCY}, there exists a collection $\{\Gamma_1,\ldots,\Gamma_r\} \subset H_n(\check{\bX};\C)$ of linearly independent cycles such that
\begin{equation*}
\begin{split}
q_a & = \exp\left(-\int_{\Gamma_a}\check{\Omega}_{\mathcal{F}^\mathrm{SYZ}(q,\tau)}\right), \quad a = 1,\ldots,r',\\
\tau_{\bb_j} & = \int_{\Gamma_{j-m+r'+1}}\check{\Omega}_{\mathcal{F}^\mathrm{SYZ}(q,\tau)}, \quad j = m,\ldots,m'-1,
\end{split}
\end{equation*}
where $q_a$'s and $\tau_{\bb_j}$'s are the K\"ahler and orbifold parameters in the extended complexified K\"ahler moduli space of $\bX$.
\end{thm}

We deduce the following relation between disk invariants and period integrals in the manifold case:
\begin{corollary}[Open mirror theorem for toric CY manifolds]\label{cor:period_mirror_manifold}
For a semi-projective toric CY manifold $\bX$, there exists a collection $\{\Gamma_1,\ldots,\Gamma_r\} \subset H_n(\check{\bX};\C)$ of linearly independent cycles such that
\begin{equation*}
q_a = \exp\left(-\int_{\Gamma_a}\check{\Omega}_{\mathcal{F}^\mathrm{SYZ}(q,\tau)}\right), \quad a = 1,\ldots,r,
\end{equation*}
where $q_a$'s are the K\"ahler parameters in the complexified K\"ahler moduli space of $\bX$, and $\mathcal{F}^\mathrm{SYZ}(q)$ is the SYZ map in Definition~\ref{defn:SYZ_map}, now defined in terms of the generating functions $1+\delta_i$ of genus 0 open GW invariants $n^\bX_{1,l,\beta_i+\alpha}([\mathrm{pt}]_L)$.
\end{corollary}
Our results provide an enumerative meaning to period integrals, as conjectured by Gross and Siebert in \cite[Conjecture 0.2 and Remark 5.1]{gross07}. One difference between our results and their conjecture is that we use holomorphic disks while they considered {\em tropical} disks.  In the case of toric Calabi-Yau manifolds, our symplectic construction was proved in \cite{Lau14} to be equivalent to the Gross-Siebert tropical construction by using our explicit formula for open Gromov-Witten invariants given in Theorem \ref{thm:disk_calc_intro}.  On the other hand, their conjecture is much more general and expected to hold even when $\bX$ is a {\em compact} CY manifold. See \cite[Conjecture 1.1]{CLL} (also \cite[Conjecture 2]{CLT11}) for a more precise formulation of the Gross-Siebert conjecture in the case of toric CY manifolds.

Corollary \ref{cor:period_mirror_manifold} proves a weaker form of \cite[Conjecture 1.1]{CLL}, which concerns periods over {\em integral} cycles in $\check{\bX}$ (while here the cycles $\Gamma_1,\ldots,\Gamma_r$ are allowed to have complex coefficients), for {\em all} semi-projective toric CY manifolds. The case when $\bX$ is the total space of the canonical line bundle of a toric Fano manifold was previous proved in \cite{CLT11}.\footnote{As explained in \cite[Section 5.2]{CLT11}, to prove the original stronger form of the conjecture, we need {\em integral} cycles whose periods have specific logarithmic terms. Such cycles have been constructed by Doran and Kerr in \cite[Section 5.3 and Theorem 5.1]{Doran-Kerr11} when $\bX$ is the total space of the canonical line bundle $K_Y$ over a toric del Pezzo surface $Y$. Doran suggested to us that it should not be difficult to extend their construction to general toric CY varieties. Hence the stronger form of the conjecture should follow from Corollary \ref{cor:period_mirror_manifold} and their construction; cf. the discussion in \cite[Section 4]{Doran-Kerr13}. In the recent paper \cite{Ruddat-Siebert14}, Ruddat and Siebert gave yet another construction of such integral cycles by tropical methods. Though they worked only in the compact CY case, Ruddat pointed out that the method can be generalized to handle the toric CY case as well.}

\subsubsection{Open crepant resolution conjecture}
The second main application concerns how genus 0 open (orbifold) GW invariants change under crepant birational maps. String theoretic considerations suggest that GW theory should remain unchanged as the target space changes under a crepant birational map. This is known as the {\em crepant resolution conjecture} and has been intensively studied in closed GW theory; see e.g. \cite{Ruan06, Bryan-Graber09, CIT09, Coates09, Coates-Ruan} and references therein.

In \cite{CCLT12}, a conjecture on how generating functions of genus 0 open GW invariants behave under crepant resolutions was formulated and studied for compact Gorenstein toric orbifolds. In this paper, we apply our calculations to prove an analogous result for toric CY orbifolds (see Section \ref{sec:openCRC} for details):

\begin{thm}[See Theorem \ref{thm:openCR_toric_CY}]
Let $\bX$ be a toric CY orbifold as in Setting \ref{setting:toricCY}, and let $\bX'$ be a toric orbifold which is a toric crepant partial resolution of $\bX$ (such an $\bX'$ will automatically be as in Setting \ref{setting:toricCY}). Then we have
\begin{equation*}
\mathcal{F}^{\mathrm{SYZ}}_{\mathcal{X}} = \mathcal{F}^{\mathrm{SYZ}}_{\bX'},
\end{equation*}
after analytic continuation and a change of variables.
\end{thm}

Open versions of the crepant resolution conjecture for Aganagic-Vafa type Lagrangian branes in 3-dimensional toric CY orbifolds have been considered in recent works of Brini, Cavalieri and Ross \cite{Cavalieri-Ross11, BCR13-2}, and of Ke and Zhou \cite{Ke-Zhou}.

\subsection{Organization}
The rest of the paper is organized as follows. Section \ref{sec:toric_orbifolds} contains a review on the basic materials about toric orbifolds that we need. The (equivariant) mirror theorem for toric orbifolds is discussed in Section \ref{sec:toric_mirror_theorem}. In Section \ref{sec:review_disk_inv} we give a summary on the theory of genus 0 open orbifold GW invariants for toric orbifolds. In Section \ref{sec:orb_Gross_fib} we define and study the Gross fibration of a toric CY orbifold. In Section \ref{sec:SYZ_mirror} we construct the instanton-corrected mirror of a toric CY orbifold by applying the SYZ recipe to the Gross fibration of a suitable toric modification. The genus 0 open orbifold GW invariants which are relevant to the SYZ mirror construction are computed in Section \ref{sec:computation_open_GW} via an open/closed equality and an equivariant toric mirror theorem applied to various toric (partial) compactifications. In Section \ref{sec:open_mirror_thms} we apply our calculation to deduce the open mirror theorems which relate various mirror maps associated to a toric CY orbifold. Our calculation is also applied in Section \ref{sec:openCRC} to prove a relationship between genus 0 open orbifold GW invariants of a toric CY orbifold and those of its toric crepant (partial) resolutions. 
Appendix \ref{app:analy_conti} contains the technical discussions on the analytic continuations of mirror maps.

\subsection{Acknowledgment}
We are grateful to N. C. Leung for continuous encouragement and related collaborations. We thank L. Borisov, S. Hosono, Y. Konishi and S. Minabe for very enlightening and useful discussions on GKZ systems and period integrals, and C. Doran and H. Ruddat for explaining their constructions of integral cycles. H.-H. Tseng also thanks T. Coates, A. Corti, and H. Iritani for related collaborations and discussions. We thank the referees for very helpful corrections, comments and suggestions.

The research of K. Chan was substantially supported by a grant from the Research Grants Council of the Hong Kong Special Administrative Region, China (Project No. CUHK404412).
C.-H. Cho was supported in part by the National Research Foundation of Korea (NRF) grant funded by the Korea Government (MEST) (No. 2013042157 and No. 2012R1A1A2003117).
S.-C. Lau was supported by Harvard University.
H.-H. Tseng was supported in part by Simons Foundation Collaboration Grant.

\section{Preliminaries on toric orbifolds}\label{sec:toric_orbifolds}

In this section we review the construction and basic properties of toric orbifolds. We also describe the closed mirror theorem for toric orbifolds due to \cite{CCIT_toricDM}. See \cite{BCS, Jiang08} for more details on toric orbifolds, and see \cite{iritani09, fang-liu-tseng, CCIT_toricDM} for mirror theorems for toric orbifolds.

\subsection{Construction}\label{sec:toric_orb_defn}
A {\em toric orbifold}, as introduced in \cite{BCS}, is associated to a set of combinatorial data called a {\em stacky fan}: $(\Sigma, \bb_0,\ldots, \bb_{m-1}),$
where $\Sigma$ is a simplicial fan contained in the $\mathbb{R}$-vector space $N_\R := N\otimes_\Z\R$ associated to a rank $n$ lattice $N$, and
$\{\bb_i \mid 0\leq i\leq m-1\}$ are integral generators of 1-dimensional cones (rays) in $\Sigma$. We call $\bb_i$ the {\em stacky vectors}. Denote by $|\Sigma|\subset N_\R$ the support of $\Sigma$.

Let $\bb_{m}, \ldots, \bb_{m'-1}\in N\cap |\Sigma|$ be additional vectors such that the set $\{\bb_i\}_{i=0}^{m-1}\cup \{\bb_j\}_{j=m}^{m'-1}$ generates $N$ over $\integer$. Following \cite{Jiang08}, the data
$(\Sigma, \{\bb_i\}_{i=0}^{m-1}\cup \{\bb_j\}_{j=m}^{m'-1})$
is called an {\em extended stacky fan}, and $\{\bb_j\}_{j=m}^{m'-1}$ are called {\em extra vectors}. The flexibility of choosing extra vectors is important in the toric mirror theorem, see Section \ref{sec:toric_mirror_theorem}.

We describe the construction of toric orbifolds from extended stacky fans. The {\em fan map},
$$\phi: \widetilde{N}:=\bigoplus_{i=0}^{m'-1}\Z e_i\to N, \quad \phi(e_i):= \bb_i\textrm{ for $i=0,\ldots,m'-1$},$$
which is a surjective group homomorphism, gives an exact sequence (the ``fan sequence'')
\begin{equation}\label{eqn:fan_seq}
0\longrightarrow \mathbb{L}:=\mathrm{Ker}(\phi)\overset{\psi}{\longrightarrow}\widetilde{N}\overset{\phi}{\longrightarrow} N\longrightarrow 0.
\end{equation}
Note that $\mathbb{L}\simeq \Z^{m'-n}$. Tensoring with $\mathbb{C}^\times$ gives the following exact sequence:
\begin{equation}\label{kexact4}
0 \longrightarrow G:=\mathbb{L}\otimes_\Z \C^\times \longrightarrow \widetilde{N}\otimes_\Z \C^\times \simeq (\C^\times)^{m'} \stackrel{\phi_{\C^\times}}{\longrightarrow} \mathbb{T}:= N\otimes_\Z \C^\times \to 0.
\end{equation}

Consider the set of {\em ``anti-cones''},
\begin{equation}\label{defn:anticone}
\mathcal{A}:=\left\{I\subset \{0, 1, \ldots, m'-1\} \mid \sum_{i\notin I}\R_{\geq 0} \bb_i \text{ is a cone in } \Sigma\right\}.
\end{equation}
For $I\in \mathcal{A}$, let $\C^I\subset \C^{m'}$ be the subvariety defined by the ideal in $\C[Z_0, \ldots, Z_{m'-1}]$ generated by $\{Z_i \mid i\in I\}$. Put
$$U_\mathcal{A}:=\C^{m'}\setminus \bigcup_{I \notin \mathcal{A}}\C^I.$$
The algebraic torus $G$ acts on $\C^{m'}$ via the map $G\to (\C^\times)^{m'}$ in \eqref{kexact4}. Since $N$ is torsion-free, the induced $G$-action on $U_\mathcal{A}$  is effective and has finite isotropy groups. The global quotient
\begin{equation*}
\mathcal{X}_\Sigma := [U_\mathcal{A}/G]
\end{equation*}
is called the {\em toric orbifold} associated to $(\Sigma, \{\bb_i\}_{i=0}^{m-1}\cup \{\bb_j\}_{j=m}^{m'-1})$.  By construction, the standard $(\C^\times)^{m'}$-action on $U_\mathcal{A}$ induces a $\mathbb{T}$-action on $\bX_\Sigma$.


\begin{defn}\label{defn:semi_proj}
Let $X_\Sigma$ be the toric variety which is the coarse moduli space of a toric orbifold $\bX_\Sigma$. We say that $X_\Sigma$ is {\em semi-projective} if $X_\Sigma$ admits a $\mathbb{T}$-fixed point, and the natural map $X_\Sigma\to \text{Spec }H^0(X_\Sigma, \mathcal{O}_{X_\Sigma})$ is projective.
\end{defn}

Toric orbifolds appearing in this paper all have semi-projective coarse moduli spaces. We refer to \cite[Section 7.2]{CLS_toricbook} for more detailed discussions on semi-projective toric varieties.


\subsection{Twisted sectors} \label{sect:TwSec}
For a $d$-dimensional cone $\sigma\in \Sigma$ generated by $\bb_\sigma=(\bb_{i_1}, \ldots, \bb_{i_d})$, put
$$\mathrm{Box}_{\bb_\sigma} :=\left\{\nu \in N \mid \nu=\sum_{k=1}^d t_k \bb_{i_k},\ t_k \in [0,1)\cap\rat\right\}.$$
Let $N_{\bb_\sigma}\subset N$ be the submodule generated by $\{ \bb_{i_1}, \ldots, \bb_{i_d}\}$. Then $\mathrm{Box}_{\bb_\sigma}$ is in bijection with the finite group $G_{\bb_{\sigma}} = N/N_{\bb_{\sigma}}$. It is easy to see that if $\tau \prec \sigma$, then $\mathrm{Box}_{\bb_\tau}\subset \mathrm{Box}_{\bb_\sigma}$. Define
\begin{equation*}
\mathrm{Box}_{\bb_\sigma}^{\circ} := \mathrm{Box}_{\bb_\sigma} - \bigcup_{\tau \precneqq \sigma} \mathrm{Box}_{\bb_\tau},\quad \mathrm{Box}(\Sigma) := \bigcup_{\sigma \in \Sigma^{(n)}} \mathrm{Box}_{\bb_\sigma} = \bigsqcup_{\sigma \in \Sigma} \mathrm{Box}_{\bb_\sigma}^{\circ}
\end{equation*}
where $\Sigma^{(n)}$ is the set of $n$-dimensional cones in $\Sigma$. We set $\mathrm{Box}'(\Sigma)=\mathrm{Box}(\Sigma)\setminus\{0\}$.

By \cite{BCS}, $\mathrm{Box}'(\Sigma)$ is in bijection with the {\em twisted sectors}, i.e. non-trivial connected components of the inertia orbifold of $\bX_\Sigma$. For $\nu\in\mathrm{Box}(\Sigma)$, denote by $\bX_\nu$ the corresponding twisted sector of $\bX$. Note that $\bX_0=\bX$ as orbifolds. See Figure \ref{fig:fan for C^2 mod Z_m} for an example of $\mathrm{Box}'(\Sigma)$.

The {\em Chen-Ruan orbifold cohomology} $H^*_\mathrm{CR}(\mathcal{X};\rat)$ of a toric obifold $\bX$, defined in \cite{CR04}, is
$$H^d_\mathrm{CR}(\mathcal{X};\rat)=\bigoplus_{\nu\in\mathrm{Box}}H^{d-2\mathrm{age}(\nu)}(\mathcal{X}_\nu;\rat),$$
where $\mathrm{age}(\nu)$ is the {\em degree shifting number} or {\em age} of the twisted sector $\mathcal{X}_\nu$ and the cohomology groups on the right hand side are singular cohomology groups. If we write $\nu=\sum_{k=1}^d t_k \bb_{i_k} \in \mathrm{Box}(\Sigma)$ where $\{\bb_{i_1},\ldots,\bb_{i_d}\}$ generates a cone in $\Sigma$, then $\mathrm{age}(\nu) = \sum_{k=1}^d t_k \in \rat_{\geq0}.$

The $\mathbb{T}$-action on $\bX$ induces $\mathbb{T}$-actions on twisted sectors. This allows one to define the {\em $\mathbb{T}$-equivariant Chen-Ruan orbifold cohomology} $H^*_{\mathrm{CR}, \mathbb{T}}(\mathcal{X};\rat)$ as
$$H^d_{\mathrm{CR}, \mathbb{T}}(\mathcal{X};\rat)=\bigoplus_{\nu\in\mathrm{Box}}H^{d-2\mathrm{age}(\nu)}_{\mathbb{T}}(\mathcal{X}_\nu;\rat),$$
where $H^*_\mathbb{T}(-)$ denotes $\mathbb{T}$-equivariant cohomology. The trivial $\mathbb{T}$-bundle over a point $\text{pt}$ defines a map $\text{pt}\to B\mathbb{T}$, inducing a map $H^*_\mathbb{T}(\text{pt}, \mathbb{Q})=H^*(B\mathbb{T}, \mathbb{Q})\to H^*(\text{pt})$. Let $Y$ be a space with a $\mathbb{T}$-action. By construction the $\mathbb{T}$-equvariant cohomology of $Y$ admits a map $H^*_\mathbb{T}(\text{pt})\to H^*_\mathbb{T}(Y, \mathbb{Q})$. This defines a natural map
$$H^*_\mathbb{T}(Y, \mathbb{Q})\to H^*_\mathbb{T}(Y, \mathbb{Q})\otimes_{H^*_\mathbb{T}(\text{pt})} H^*(\text{pt})\simeq H^*(Y, \mathbb{Q}).$$
For a class $C\in H^*_\mathbb{T}(Y, \mathbb{Q})$, its image under this map, which is a class in $H^*(Y, \mathbb{Q})$, is called the {\em non-equivariant limit} of $C$. In Section \ref{sec:computation_open_GW}, we will need to consider non-equivariant limits of certain classes in $H^*_{\mathrm{CR}, \mathbb{T}}(\mathcal{X};\rat)$.

\subsection{Toric divisors, K\"ahler cones, and Mori cones}\label{sec:div_cone_etc}

Let $\bX$ be a toric orbifold defined by an extended stacky fan $(\Sigma, \{\bb_i\}_{i=0}^{m-1}\cup \{\bb_j\}_{j=m}^{m'-1})$. Let $\mathcal{A}$ be the set of anticones given in \eqref{defn:anticone}. Applying $\mathrm{Hom}_\Z(-,\Z)$ to the fan sequence \eqref{eqn:fan_seq} gives the following exact sequence:
\begin{equation*}
0\longrightarrow M \overset{\phi^\vee}{\longrightarrow} \widetilde{M} \overset{\psi^\vee}{\longrightarrow} \mathbb{L}^\vee \longrightarrow 0,
\end{equation*}
called the ``divisor sequence''. Here $M := N^\vee = \mathrm{Hom}(N, \Z)$, $\widetilde{M} := \widetilde{N}^\vee = \mathrm{Hom}(\widetilde{N},\Z)$ and $\mathbb{L}^\vee = \mathrm{Hom}(\mathbb{L},\Z)$ are dual lattices. The map $\psi^\vee: \widetilde{M}\to \mathbb{L}^\vee$ is surjective since $N$ is torsion-free.

By construction, line bundles on $\bX$ correspond to $G$-equivariant line bundles on $\mathcal{U}_\mathcal{A}$. Because of (\ref{kexact4}), $\mathbb{T}$-equivariant line bundles on $\bX$ correspond to $(\mathbb{C}^\times)^{m'}$-equivariant line bundles on $\mathcal{U}_\mathcal{A}$. Because $\cup_{I\notin \mathcal{A}} \mathbb{C}^I\subset \mathbb{C}^{m'}$ is of codimension at least $2$, we have the following descriptions of the Picard groups: $$Pic(\bX)\simeq Hom(G, \mathbb{C}^\times)\simeq \mathbb{L}^\vee, Pic_\mathbb{T}(\bX)\simeq Hom((\mathbb{C}^\times)^{m'}, \mathbb{C}^\times)\simeq \widetilde{N}^\vee=\widetilde{M}.$$
Moreover, the natural map $Pic_\mathbb{T}(\bX)\to Pic(\bX)$ is identified with  $\psi^\vee: \widetilde{M}\to \mathbb{L}^\vee$.

Let $\{e_i^\vee | i=0, 1, \ldots, m'-1 \}\subset \widetilde{M}$ be the basis dual to $\{e_i | i=0, 1, \ldots, m'-1\}\subset \widetilde{N}$. For $i=0, 1, \ldots, m'-1$, we denote by $D_i^\mathbb{T}$ the $\mathbb{T}$-equivariant line bundle on $\bX$ corresponding to $e_i^\vee$ under the identification $Pic_\mathbb{T}(\bX)\simeq \widetilde{M}$. Also put
$$D_i:=\psi^\vee(e_i^\vee)\in \mathbb{L}^\vee.$$
The collection $\{D_i \mid 0\leq i\leq m-1\}$ are toric prime divisors corresponding to the generators $\{\bb_i \mid 0\leq i\leq m-1\}$ of rays in $\Sigma$, and $\{D^\mathbb{T}_i \mid 0\leq i\leq m-1\}$ are their $\mathbb{T}$-equivariant lifts. There is a natural commutative diagram and isomorphisms
\begin{equation*}
\xymatrix{
\widetilde{M}\otimes \mathbb{Q}\ar[d]\ar[r]^{\psi^\vee\otimes \mathbb{Q}} & \mathbb{L}^\vee\otimes \mathbb{Q}\ar[d]\\
\left(\widetilde{M}\otimes\mathbb{Q}\right) \Big/ \left(\sum_{j=m}^{m'-1}\mathbb{Q}D^\mathbb{T}_j\right)\simeq H_\mathbb{T}^2(\bX, \mathbb{Q})\ar[r] & H^2(\bX, \mathbb{Q}) \simeq \left(\mathbb{L}^\vee\otimes\mathbb{Q}\right) \Big/ \left(\sum_{j=m}^{m'-1}\mathbb{Q}D_j\right).
}
\end{equation*}





As explained in \cite[Section 3.1.2]{iritani09}, there is a canonical splitting of the quotient map $\mathbb{L}^\vee\otimes \mathbb{Q}\to H^2(\bX; \mathbb{Q})$, which we now describe. For $m\leq j \leq m'-1$, $\bb_j$ is contained in a cone in $\Sigma$. Let $I_j\in \mathcal{A}$ be the anticone of the cone containing $\bb_j$. Then we can write
$\bb_j=\sum_{i\notin I_j} c_{ji}\bb_i$ for $c_{ji}\in\rat_{\geq 0}$.

By the fan sequence \eqref{eqn:fan_seq} tensored with $\mathbb{Q}$, there exists a unique $D_j^\vee\in \mathbb{L}\otimes \mathbb{Q}$ such that
\begin{equation}\label{eqn:dual_of_D_j}
\langle D_i, D_j^\vee\rangle
=\left\{
\begin{array}{lll}
1 & \textrm{ if } i=j,\\
-c_{ji} &  \textrm{ if } i\notin I_j,\\
0 & \textrm{ if } i\in I_j\setminus \{j\}.
\end{array}\right.
\end{equation}
Here and henceforth $\langle -,-\rangle$ denotes the natural pairing between $\mathbb{L}^\vee$ and $\mathbb{L}$ (or relevant extensions of scalars). This defines a decomposition
\begin{equation}\label{eqn:splitting_L_dual}
\mathbb{L}^\vee\otimes \mathbb{Q} =\mathrm{Ker}\left(\left(D_{m}^\vee, \ldots, D_{m'-1}^\vee \right): \mathbb{L}^\vee\otimes \mathbb{Q}\to \mathbb{Q}^{m'-m}\right)\oplus \bigoplus_{j=m}^{m'-1}\mathbb{Q}D_j.
\end{equation}
Moreover, the term $\mathrm{Ker}\left(\left(D_{m}^\vee, \ldots, D_{m'-1}^\vee \right): \mathbb{L}^\vee\otimes \mathbb{Q}\to \mathbb{Q}^{m'-m}\right)$ is naturally identified with $H^2(\bX; \mathbb{Q})$ via the quotient map $\mathbb{L}^\vee\otimes\mathbb{Q}\to H^2(\bX; \mathbb{Q})$, which allows us to regard $H^2(\bX; \mathbb{Q})$ as a subspace of $\mathbb{L}^\vee\otimes\mathbb{Q}$.

The {\em extended K\"ahler cone} of $\bX$ is defined to be
$\widetilde{C}_\bX := \bigcap_{I\in \mathcal{A}}\left(\sum_{i\in I}\mathbb{R}_{>0}D_i \right)\subset \mathbb{L}^\vee\otimes\mathbb{R}.$
The genuine K\"ahler cone $C_\bX$ is the image of $\widetilde{C}_\bX$ under the quotient map $\mathbb{L}^\vee\otimes \mathbb{R}\to H^2(\bX; \R)$. The splitting \eqref{eqn:splitting_L_dual} of $\mathbb{L}^\vee\otimes\mathbb{Q}$ induces a splitting of the extended K\"ahler cone in $\mathbb{L}^\vee\otimes \R$:
$\widetilde{C}_\bX = C_\bX+\sum_{j=m}^{m'-1} \mathbb{R}_{>0}D_j.$

Recall that the rank of $\mathbb{L}^\vee$ is $r:=m'-n$ while the rank of $H_2(\bX; \mathbb{Z})$ is given by $r':=r-(m'-m)=m-n$. We choose an integral basis
$\{p_1, \ldots, p_r\}\subset \mathbb{L}^\vee$
such that $p_a$ is in the closure of $\widetilde{C}_\bX$ for all $a$ and $p_{r'+1}, \ldots, p_r \in \sum_{i=m}^{m'-1}\mathbb{R}_{\geq 0}D_i$. Then the images $\{\bar{p}_1,\ldots,\bar{p}_{r'}\}$ of $\{p_1,\ldots,p_{r'}\}$ under the quotient map $\mathbb{L}^\vee\otimes \Q\to H^2(\bX; \Q)$ gives a nef basis for $H^2(\bX;\Q)$ and $\bar{p}_a = 0$ for $r'+1\leq a\leq r$.

Choose $\{p_1^\mathbb{T}, \ldots, p_r^\mathbb{T}\}\subset \widetilde{M}\otimes \mathbb{Q}$ such that $\psi^\vee(p_a^\mathbb{T})=p_a$ for all $a$, and $\bar{p}_a^{\mathbb{T}}=0$ for $a=r'+1,...,r$. Here, for $p\in \widetilde{M}\otimes \mathbb{Q}$, denote by $\bar{p}\in H_\mathbb{T}^2(\bX, \mathbb{Q})$ the image of $p$ under the natural map $\widetilde{M}\otimes \mathbb{Q}\to H_\mathbb{T}^2(\bX, \mathbb{Q})$.  By construction, for $a=1,...,r'$, $\bar{p}_a$ is the non-equivariant limit of $\bar{p}_a^\mathbb{T}$.

Define a matrix $(Q_{ia})$ by
$D_i = \sum_{a=1}^r Q_{ia} p_a,\quad Q_{ia}\in\mathbb{Z}.$
Denote by $\bar{D}_i$ the image of $D_i$ under $\mathbb{L}^\vee\otimes \Q\to H^2(\bX; \Q)$. Then for $i=0,\ldots,m-1$, the class $\bar{D}_i$ of the toric prime divisor $D_i$ and its equivariant lift $\bar{D}^\mathbb{T}_i$ are given by
$$\bar{D}_i = \sum_{a=1}^{r'} Q_{ia}\bar{p}_a, \quad \bar{D}^\mathbb{T}_i = \sum_{a=1}^{r'} Q_{ia}\bar{p}^\mathbb{T}_a+\lambda_i, \text{ where }\lambda_i\in H^2(B\mathbb{T}; \mathbb{Q}),$$
and for $i=m,\ldots, m'-1$, $\bar{D}_i=0$ in $H^2(\bX; \R)$, $\bar{D}^\mathbb{T}_i=0$.

Let ${\bf 1}\in H^0(\bX, \rat)$ be the fundamental class. For $\nu\in \text{Box}$ with $\text{age}(\nu)=1$, let ${\bf 1}_\nu\in H^0(\bX_\nu, \rat)$ be the fundamental class. It is then straightforward to see that $$H^0_{\mathrm{CR}, \mathbb{T}}(\bX, K_\mathbb{T})=K_\mathbb{T}{\bf 1}, \quad H^2_{\mathrm{CR}, \mathbb{T}}(\bX, K_\mathbb{T})=\bigoplus_{a=1}^{r'} K_\mathbb{T} \bar{p}_a^{\mathbb{T}}\oplus \bigoplus_{\nu\in \text{Box}, \text{age}(\nu)=1} K_\mathbb{T} {\bf 1}_\nu,$$
where $K_\mathbb{T}$ is the field of fractions of $H_\mathbb{T}^*(\text{pt}, \mathbb{Q})$, and $H^*_\mathbb{T}(-, K_\mathbb{T}):=H^*_\mathbb{T}(-, \mathbb{Q})\otimes_{H_\mathbb{T}^*(\text{pt}, \mathbb{Q})}K_\mathbb{T}$.

The dual basis of $\{p_1, \ldots, p_r\}\subset \mathbb{L}^\vee$ is given by $\{\gamma_1, \ldots, \gamma_r\} \subset \mathbb{L}$ where
$\gamma_a = \sum_{i=0}^{m'-1} Q_{ia}e_i \in \widetilde{N}.$
Then $\{\gamma_1,\ldots,\gamma_{r'}\}$ provides a basis of $H_2^\textrm{eff}(\bX;\Q)$. In particular, we have $Q_{ia} = 0$ when $m\leq i\leq m'-1$ and $1\leq a\leq r'$.

Set
\begin{align*}
\mathbb{K} & :=\{d\in\mathbb{L}\otimes\rat \mid \{j\in\{0, 1,\ldots,m'-1\}\mid\langle D_j,d\rangle\in\Z\}\in\mathcal{A}\},\\
\mathbb{K}_\mathrm{eff} & :=\{d\in\mathbb{L}\otimes\rat \mid \{j\in\{0,1,\ldots,m'-1\}\mid\langle D_j,d\rangle\in\Z_{\geq0}\}\in\mathcal{A}\},
\end{align*}
Roughly speaking $\mathbb{K}_\mathrm{eff}$ is the set of effective curve classes. In particular, the intersection $\mathbb{K}_\mathrm{eff} \cap H_2(\bX;\R)$ consists of classes of stable maps $\proj(1,m) \to \bX$ for some $m \in \Z_{\geq0}$. See e.g. \cite[Section 3.1]{iritani09} for more details.

For a real number $\lambda\in\R$, let $\lceil \lambda \rceil$, $\lfloor \lambda \rfloor$ and $\{\lambda\}$ denote the ceiling, floor and fractional part of $\lambda$ respectively. Now for $d\in\mathbb{K}$, define
\begin{equation}\label{eqn:reduction_func}
\nu(d):=\sum_{i=0}^{m'-1}\lceil\langle D_i,d\rangle\rceil\bb_i\in N,
\end{equation}
and let $I_d := \{j\in\{0, 1,\ldots,m'-1\}\mid\langle D_j,d\rangle\in\Z\} \in \mathcal{A}$. Then since we can rewrite
$$\nu(d) = \sum_{i=0}^{m'-1}(\{-\langle D_i,d\rangle\}+\langle D_i,d\rangle)\bb_i
= \sum_{i=0}^{m'-1} \{-\langle D_i,d\rangle\}\bb_i
= \sum_{i\notin I_d} \{-\langle D_i,d\rangle\}\bb_i,$$
we have $\nu(d)\in\mathrm{Box}$, and hence $\nu(d)$, if nonzero, corresponds to a twisted sector $\bX_{\nu(d)}$ of $\bX$.

\subsection{The $I$-function}
We now define the following combinatorial object.
\begin{defn}\label{defn:equiv_I_function}
The {\em $\mathbb{T}$-equivariant $I$-function} of a toric orbifold $\bX$ is an $H^*_{\mathrm{CR}, \mathbb{T}}(\bX)$-valued power series defined by
\begin{align*}
I_{\bX, \mathbb{T}}(y,z) = \conste^{\sum_{a=1}^r\bar{p}^\mathbb{T}_a\log y_a/z}\left(\sum_{d\in\mathbb{K}_\mathrm{eff}}y^d
\prod_{i=0}^{m'-1} \frac{\prod_{k=\lceil\langle D_i,d\rangle\rceil}^\infty(\bar{D}^\mathbb{T}_i+(\langle D_i,d\rangle-k)z)}
{\prod_{k=0}^\infty(\bar{D}^\mathbb{T}_i+(\langle D_i,d\rangle-k)z)}\mathbf{1}_{\nu(d)}\right),
\end{align*}
where $y^d=y_1^{\langle p_1,d\rangle}\cdots y_{r}^{\langle p_{r},d\rangle}$ and $\mathbf{1}_{\nu(d)}\in H^0(\bX_{\nu(d)})\subset H^{2\mathrm{age}(\nu(d))}_\mathrm{CR}(\bX)$ is the fundamental class of the twisted sector $\bX_{\nu(d)}$. The {\em $I$-function} of $\bX$ is an $H^*_\mathrm{CR}(\bX)$-valued power series $I_\bX(y,z)$ defined by the above equation with $\bar{p}^\mathbb{T}_a$ (resp. $\bar{D}^\mathbb{T}_i$) replaced by $\bar{p}_a$ (resp. $\bar{D}_i$). Clearly the non-equivariant limit of $I_{\bX, \mathbb{T}}$ is $I_\bX$.
\end{defn}


\begin{defn} \label{defn:sF}
A toric orbifold $\bX$ is said to be {\em semi-Fano} if $\hat{\rho}(\bX) := \sum_{i=0}^{m'-1}D_i$ is contained in the closure of the extended K\"ahler cone $\widetilde{C}_\bX$ in $\mathbb{L}^\vee\otimes \R$.
\end{defn}
We remark that this condition {\em depends} on the choice of the extra vectors $\bb_m,\ldots,\bb_{m'-1}$. It holds if and only if the first Chern class $c_1(\bX) \in H^2(\bX; \rat)$ of $\bX$ is contained in the closure of the K\"ahler cone $C_\bX$ (i.e. the anticanonical divisor $-K_\bX$ is nef) and $\textrm{age}(\bb_j):=\sum_{i\notin I_j} c_{ji}\leq 1$ for $m\leq j\leq m'-1$, because $\hat{\rho}(\bX) = c_1(\bX) + \sum_{j=m}^{m'-1} (1-\textrm{age}(\bb_j)) D_j$ by \cite[Lemma 3.3]{iritani09}. In particular, when $\bX$ is a toric manifold, the condition is equivalent to requiring the anticanonical divisor $-K_\bX$ to be nef.

The orbifolds that we consider in this paper satisfy the following assumption.
\begin{assumption}\label{assumption}
The set $\{\bb_0,\ldots,\bb_{m-1}\} \cup \{\nu\in\mathrm{Box}(\Sigma) \mid \textrm{age}(\nu)\leq1\}$ generates the lattice $N$ over $\mathbb{Z}$.
\end{assumption}
Under this assumption, we choose the extra vectors $\bb_m,\ldots,\bb_{m'-1}\in \{\nu\in\mathrm{Box}(\Sigma) \mid \textrm{age}(\nu)\leq1\}$ so that $\{\bb_0,\ldots,\bb_{m'-1}\}$ generates $N$ over $\integer$. Then the fan sequence \eqref{eqn:fan_seq} determines the elements $D_0,\ldots,D_{m'-1}$ and $\hat{\rho}(\bX) = D_0+\cdots+D_{m'-1}$ holds (see \cite[Remark 3.4]{iritani09}). Furthermore, we can then identify $\mathbb{L}^\vee\otimes \cpx$ with the subspace $H^2(\bX)\oplus\bigoplus_{j=m}^{m'-1} H^0(\bX_{\bb_j}) \subset H^{\leq2}_\mathrm{CR}(\bX).$

If $\bX$ is semi-Fano, then its $I$-function is a convergent power series in $y_1,\ldots,y_r$ by \cite[Lemma 4.2]{iritani09}. Moreover, it can be expanded as $I_\bX(y,z) = 1+\frac{\tau(y)}{z}+O(z^{-2}),$
where $\tau(y)$ is a (multi-valued) function with values in $H^{\leq2}_\mathrm{CR}(\bX)$ which expands as
$$\tau(y) = \sum_{a=1}^{r'} \bar{p}_a \log y_a + \sum_{j=m}^{m'-1} y^{D^\vee_j}\mathbf{1}_{\bb_j} + \textrm{higher order terms}.$$
We call $q(y)=\exp \tau(y)$ the {\em toric mirror map}, and it defines a local embedding near $y=0$ (it is a local embedding if we further assume that $\{\bb_m,\ldots,\bb_{m'-1}\} = \{\nu\in\mathrm{Box}(\Sigma) \mid \textrm{age}(\nu)\leq1\}$); see \cite[Section 4.1]{iritani09} for more details. Similar discussion is valid for equivariant $I$-functions.

\subsection{Equivariant GW invariants}\label{sec:equiv_GW_inv}
In this subsection we discuss the construction of equivariant GW invariants. We refer to \cite{CR01} and \cite{AGV08} for the basics of GW theory of orbifolds, and to e.g. \cite{givental_imrn96} and \cite{liu_HIM} for generalities on equivariant GW theory.

The $\mathbb{T}$-action on $\bX$ induces $\mathbb{T}$-actions on moduli spaces of stable maps to $\bX$. It is well-known that in this situation we can define {\em $\mathbb{T}$-equivariant GW invariants} of $\bX$ as integrals against $\mathbb{T}$-equivariant virtual fundamental classes of these moduli spaces.

Let $\mathcal{M}^{cl}_n(\bX, d)$ be the moduli space of $n$-pointed genus $0$ orbifold stable maps to $\bX$ of degree $d\in H_2(\bX; \mathbb{Q})$. For $i=1,...,n$, we have an evaluation map $ev_i: \mathcal{M}^{cl}_n(\bX, d)\to I\bX$, and a complex line bundle $L_i\to \mathcal{M}^{cl}_n(\bX, d)$ whose fibers are cotangent lines at the $i$-th marked point of the coarse domain curves.
Suppose $\mathcal{M}^{cl}_n(\bX, d)$ is compact. Then there is a virtual fundamental class $[\mathcal{M}^{cl}_n(\bX, d)]_{virt}\in H_*(\mathcal{M}^{cl}_n(\bX, d), \mathbb{Q})$. For cohomology classes $\phi_1,...,\phi_n\in H_{\mathrm{CR}}^*(\bX, \mathbb{Q})$ and integers $k_1,...,k_n\geq 0$, genus $0$ closed orbifold GW invariants of $\bX$ are defined as
\begin{equation}\label{defn:GW_inv_X}
\left\langle \phi_1\psi_1^{k_1},...,\phi_n\psi_n^{k_n} \right\rangle^\bX_{0, n, d}:=\int_{[\mathcal{M}^{cl}_n(\bX, d)]_{virt}}\prod_{i=1}^n (ev_i^*\phi_i\cup \psi_i^{k_i})\in \mathbb{Q},
\end{equation}
where $\psi_i:=c_1(L_i)\in H^2(\mathcal{M}^{cl}_n(\bX, d), \mathbb{Q})$.

The $\mathbb{T}$-action on $\bX$ induces a $\mathbb{T}$-action on $\mathcal{M}^{cl}_n(\bX, d)$. When $\mathcal{M}^{cl}_n(\bX, d)$ is compact, there is a $\mathbb{T}$-equivariant virtual fundamental class $[\mathcal{M}^{cl}_n(\bX, d)]_{virt, \mathbb{T}}\in H_{*, \mathbb{T}}(\mathcal{M}^{cl}_n(\bX, d), \mathbb{Q})$. For cohomology classes $\phi_{1, \mathbb{T}},...,\phi_{n, \mathbb{T}}\in H_{\mathrm{CR}, \mathbb{T}}^*(\bX, \mathbb{Q})$ and integers $k_1,...,k_n\geq 0$, $\mathbb{T}$-equivariant genus $0$ closed orbifold GW invariants of $\bX$ are defined as
\begin{equation}\label{defn:equiv_GW_inv_X}
\left\langle \phi_{1, \mathbb{T}}\psi_1^{k_1},...,\phi_{n, \mathbb{T}}\psi_{n}^{k_n} \right\rangle^{\bX, \mathbb{T}}_{0, n, d}:=\int_{[\mathcal{M}^{cl}_n(\bX, d)]_{virt, \mathbb{T}}}\prod_{i=1}^n (ev_i^*\phi_{i, \mathbb{T}}\cup \psi_i^{k_i})\in H^*_\mathbb{T}(\text{pt}, \mathbb{Q}),
\end{equation}
where $\psi_i:=c^\mathbb{T}_1(L_i)\in H_\mathbb{T}^2(\mathcal{M}^{cl}_n(\bX, d), \mathbb{Q})$ are $\mathbb{T}$-equivariant first Chern classes.

Suppose again that $\mathcal{M}^{cl}_n(\bX, d)$ is compact. Also suppose that $\phi_1,...,\phi_n\in H_{\mathrm{CR}}^*(\bX, \mathbb{Q})$ are non-equivariant limits of $\phi_{1, \mathbb{T}},...,\phi_{n, \mathbb{T}}\in H_{\mathrm{CR}, \mathbb{T}}^*(\bX, \mathbb{Q})$. Then by construction of virtual fundamental classes, the non-equivariant limit of $\left\langle \phi_{1, \mathbb{T}}\psi_1^{k_1},...,\phi_{n, \mathbb{T}}\psi_{n}^{k_n} \right\rangle^{\bX, \mathbb{T}}_{0, n, d}$, i.e. its image under the natural map $H^*_\mathbb{T}(\text{pt}, \mathbb{Q})\to H^*(\text{pt})=\mathbb{Q}$, is equal to $\left\langle \phi_1\psi_1^{k_1},...,\phi_n\psi_n^{k_n} \right\rangle^\bX_{0, n, d}$.

If $\mathcal{M}^{cl}_n(\bX, d)$ is noncompact but the locus $\mathcal{M}^{cl}_n(\bX, d)^\mathbb{T}\subset \mathcal{M}^{cl}_n(\bX, d)$ of $\mathbb{T}$-fixed points is compact, then the $\mathbb{T}$-equivariant invariant $\left\langle \phi_{1, \mathbb{T}}\psi_1^{k_1},...,\phi_{n, \mathbb{T}}\psi_{n}^{k_n} \right\rangle^{\bX, \mathbb{T}}_{0, n, d}$ can still be defined by (\ref{defn:equiv_GW_inv_X}), with the integration $\int_{[\mathcal{M}^{cl}_n(\bX, d)]_{virt, \mathbb{T}}}$ defined by the virtual localization formula \cite{graber_pand}:
$$\int_{[\mathcal{M}^{cl}_n(\bX, d)]_{virt, \mathbb{T}}}(-):=\sum_{F\subset \mathcal{M}^{cl}_n(\bX, d)^\mathbb{T}}\int_{[F]_{virt}}\frac{\iota_F^*(-)}{e_\mathbb{T}(N^{virt}_{F})}\in K_\mathbb{T},$$
where $F$ runs through all connected components of $\mathcal{M}^{cl}_n(\bX, d)^\mathbb{T}$, $\iota_F: F\to \mathcal{M}^{cl}_n(\bX, d)^\mathbb{T}$ is the inclusion, $[F]_{virt}$ is the natural virtual fundamental class on $F$, and $e_\mathbb{T}(N^{virt}_F)$ is the $\mathbb{T}$-equivariant Euler class of the virtual normal bundle $N^{virt}_F$ of $F\subset \mathcal{M}^{cl}_n(\bX, d)$. It follows easily from the virtual localization formula that if both  $\mathcal{M}^{cl}_n(\bX, d)^\mathbb{T}$ and $\mathcal{M}^{cl}_n(\bX, d)$ are compact, the two definitions of $\mathbb{T}$-equivariant invariants agree.

\begin{remark}
If $\bX$ is projective, then $\mathcal{M}^{cl}_n(\bX, d)$ is compact. If $\bX$ is not projective but semi-projective, then it is easy to show that the locus $\mathcal{M}^{cl}_n(\bX, d)^\mathbb{T}\subset \mathcal{M}^{cl}_n(\bX, d)$ of $\mathbb{T}$-fixed points is compact. In this case, $\mathbb{T}$-equvariant GW invariants are still defined.
\end{remark}

\subsection{Toric mirror theorem}\label{sec:toric_mirror_theorem}
We give a review of the mirror theorem for toric orbifolds proven in \cite{CCIT_toricDM} in the case of semi-Fano toric orbifolds. Our exposition follows \cite{iritani09} and \cite{fang-liu-tseng}.

Let $\bX$ be a toric orbifold as in Section \ref{sec:toric_orb_defn}.

\begin{defn}\label{defn:equiv_J_function}
The {\em $\mathbb{T}$-equivariant (small) $J$-function} of a toric orbifold $\bX$ is an $H^*_{\mathrm{CR}, \mathbb{T}}(\bX)$-valued power series defined by
\begin{align*}
J_\bX(q,z) = \conste^{\tau_{0,2}/z}(1+\sum_\alpha\sum_{\substack{(d,l)\neq(0,0)\\ d\in H_2^\mathrm{eff}(\bX)}}\frac{q^d}{l!}\left\langle 1,\tau_\mathrm{tw},\ldots,\tau_\mathrm{tw},\frac{\phi_\alpha}{z-\psi}\right\rangle^{\bX, \mathbb{T}}_{0,l+2,d} \phi^\alpha),
\end{align*}
where $\tau_{0,2}=\sum_{a=1}^{r'} \bar{p}^\mathbb{T}_a\log q_a\in H_\mathbb{T}^2(\bX)$, $\tau_\mathrm{tw}=\sum_{j=m}^{m'-1}\tau_{\bb_j}\mathbf{1}_{\bb_j}\in\bigoplus_{j=m}^{m'-1} H_\mathbb{T}^0(\bX_{\bb_j})$, $q^d=\conste^{\langle \tau_{0,2},d\rangle}=q_1^{\langle\bar{p}_1,d\rangle}\cdots q_{r'}^{\langle\bar{p}_{r'},d\rangle}$, $\{\phi_\alpha\}$, $\{\phi^\alpha\}$ are dual basis of $H^*_{\mathrm{CR}, \mathbb{T}}(\bX)$. The {\em (small) $J$-function} of $\bX$ is an $H^*_\mathrm{CR}(\bX)$-valued power series $J_\bX(q,z)$ defined by the above equation, with $\bar{p}^\mathbb{T}_a$ replaced by $\bar{p}_a$ and $\{\phi_\alpha\}$, $\{\phi^\alpha\}$ replaced by dual basis of $H^*_{\mathrm{CR}}(\bX)$. The non-equivariant limit of $J_{\bX, \mathbb{T}}$ is $J_\bX$.
\end{defn}


Roughly speaking, the (equivariant) mirror theorem for the toric orbifold $\bX$ states that the (equivariant) $J$-function coincides with the (equivariant) $I$-function via the mirror map.

\begin{thm}[Equivariant mirror theorem for toric orbifolds \cite{CCIT_toricDM}; see also \cite{fang-liu-tseng}, Conjecture 4.1]\label{equiv_closed_mirror_thm}
Let $\bX$ be a semi-projective semi-Fano toric K\"ahler orbifold. Then
\begin{equation*}
e^{q_0(y){\bf 1}/z}J_{\bX, \mathbb{T}}(q,z)=I_{\bX, \mathbb{T}}(y(q,\tau),z),
\end{equation*}
where $y = y(q,\tau)$ is the inverse of the toric mirror map $q=q(y)$, $\tau=\tau(y)$ determined by the expansion of the equivariant $I$-function:
$$I_{\bX, \mathbb{T}}(y,z) = 1+\frac{q_0(y){\bf 1}+\tau(y)}{z}+O(z^{-2}),\quad \tau(y)\in H^2_{\mathrm{CR}, \mathbb{T}}(\bX).$$
\end{thm}

Taking non-equivariant limits gives (note that the non-equivariant limit of $q_0(y)$ is $0$):

\begin{thm}[Closed mirror theorem for toric orbifolds \cite{CCIT_toricDM}; see also \cite{iritani09}, Conjecture 4.3]\label{closed_mirror_thm}
Let $\bX$ be a compact semi-Fano toric K\"ahler orbifold. Then
\begin{equation*}
J_\bX(q,z)=I_\bX(y(q,\tau),z),
\end{equation*}
where $y = y(q,\tau)$ is the inverse of the toric mirror map $q=q(y)$, $\tau=\tau(y)$.
\end{thm}


\section{Orbi-disk invariants}\label{sec:review_disk_inv}

We briefly review the construction of genus 0 open orbifold GW invariants of toric orbifolds \cite{CP}.

Let $(\bX,\omega)$ be a toric K\"ahler orbifold of complex dimension $n$, equipped with the standard toric complex structure $J_0$ and a toric K\"ahler structure $\omega$. Suppose that $\bX$ is associated to the stacky fan $(\Sigma,\bb)$, where $\bb=(\bb_0,\ldots,\bb_{m-1})$ and $\bb_i=c_iv_i$. As before, $D_i$ ($i=0,\ldots,m-1$) denotes the toric prime divisor associated to $\bb_i$.
Let $L \subset \bX$ be a Lagrangian torus fiber of the moment map $\mu_0:\bX \to M_\R := M\otimes_\Z \R$, and consider a relative homotopy class $\beta \in \pi_2(\bX,L) = H_2(\bX,L;\integer)$.

\subsection{Holomorphic orbi-disks and their moduli spaces}\label{sec:orbidisk_and_moduli}
A {\em holomorphic orbi-disk} in $\bX$ with boundary in $L$ is a continuous map
$w:(\bD,\partial\bD) \to (\bX,L)$
such that the following conditions are satisfied:
\begin{enumerate}
\item
$(\bD, z_1^+,\ldots,z_l^+)$ is an orbi-disk with interior orbifold marked points $z_1^+,\ldots,z_l^+$. Namely $\bD$ is analytically the disk $D^2\subset\C$, together with orbifold structure at each marked point $z_j^+$ for $j=1,\ldots,l$. For each $j$, the orbifold structure at $z_j^+$ is given by a disk neighborhood of $z_j^+$ which is uniformized by a branched covering map $br:z \to z^{m_j}$ for some\footnote{If $m_j=1$,  $z_j^+$ is a smooth interior marked point.} $m_j\in \integer_{>0}$.

\item
For any $z_0 \in \bD$, there is a disk neighborhood of $z_0$ with a branched covering map $br:z \to z^m$, and there is a local chart $(V_{w(z_0)},G_{w(z_0)},\pi_{w(z_0)})$ of $\mathcal{X}$ at $w(z_0)$ and a local holomorphic lifting $\WT{w}_{z_0}$ of $w$ satisfying
$w \circ br = \pi_{w(z_0)} \circ \WT{w}_{z_0}.$

\item
The map $w$ is {\em good} (in the sense of Chen-Ruan \cite{CR01}) and {\em representable}. In particular, for each marked point $z_j^+$, the associated  homomorphism
\begin{equation}\label{eq:locgphi}
h_p:\Z_{m_j}\to G_{w(z_j^+)}
\end{equation}
between local groups which makes $\WT{w}_{z_j^+}$ equivariant, is injective.
\end{enumerate}

Denote by $\nu_j\in\mathrm{Box}(\Sigma)$ the image of the generator $1\in \Z_{m_j}$ under $h_j$ and let $\bX_{\nu_j}$ be the twisted sector of $\bX$ corresponding to $\nu_j$. Such a map $w$ is said to be of {\em type} $\bx := (\bX_{\nu_1},\ldots, \bX_{\nu_l})$.

There are two notions of Maslov index for an orbi-disk.
The {\em desingularized Maslov index} $\mu^{de}$ is defined by desingularizing the interior singularities (following Chen-Ruan \cite{CR01}) of the pull-back bundle $w^*T\bX$ in \cite[Section 3]{CP}.  The {\em Chern-Weil (CW) Maslov index} is defined in \cite{CS} as the integral
of the curvature of a unitary connection on $w^*T\bX$ which preserves the Lagrangian boundary condition.  We will mainly use the CW Maslov index in this paper. The following lemma, which generalizes results in \cite{cho06, auroux07, CP}, can be used to compute the Maslov index of disks.
\begin{lemma}\label{lem:maslov_ind_comp}
Let $(\bX, \omega, J)$ be a K\"ahler orbifold of complex dimension $n$, equipped with a non-zero meromorphic $n$-form $\Omega$ on $\bX$ which has at worst simple poles. Let $D\subset \bX$ be the pole divisor of $\Omega$. Suppose also that the generic points of $D$ are smooth. Then for a special Lagrangian submanifold $L \subset \bX$, the CW Maslov index of a class $\beta \in \pi_2(\bX,L)$ is given by
\begin{equation}\label{eqn:index_formula}
\mu_{CW}(\beta) = 2\beta \cdot D.
\end{equation}
\end{lemma}
\begin{proof}
Suppose $\beta$ is a homotopy class of a smooth disk. Given a smooth disk representative $u:D^2 \to \bX$ of $\beta$, note that the pull-back of the canonical line bundle $u^*(K_\bX)$ is an honest vector bundle over $D^2$, and hence, the proof in \cite{auroux07} applies to this case. Also since the CW Maslov index is topological, we can write any class $\beta$ which is represented by an orbi-disk as a (fractional) linear combination of homotopy classes of smooth disks. Hence \eqref{eqn:index_formula} for an orbi-disk class $\beta$ also follows.
\end{proof}

Orbi-disks in a symplectic toric orbifold have been classified \cite[Theorem 6.2]{CP}.
Among them, the following {\em basic disks} corresponding to the stacky vectors and twisted sectors play an important role.
\begin{thm}[\cite{CP}, Corollaries 6.3 and 6.4]
Let $\bX$ and $L$ be as in the beginning of this section.
\begin{enumerate}
\item The smooth holomorphic disks of Maslov index two (modulo $T^n$-action and automorphisms of the domain) are in a one-to-one correspondence with the stacky vectors $\{\bb_0,\ldots,\bb_{m-1}\}$, whose homotopy classes are denoted as $\beta_{0},\cdots,\beta_{m-1}$.
\item The holomorphic orbi-disks with one interior orbifold marked point and desingularized Maslov index zero (modulo $T^n$-action and automorphisms of the domain) are in a one-to-one correspondence with the twisted sectors $\nu \in \mathrm{Box}'(\Sigma)$ of the toric orbifold $\bX$, whose homotopy classes are denoted as $\beta_\nu$.
\end{enumerate}
\end{thm}

\begin{lemma}[\cite{CP}, Lemma 9.1]\label{lem:basic_disk_classes}
For $\bX$ and $L$ as above, the relative homotopy group $\pi_2(\bX,L)$ is generated by the classes $\beta_i$ for $i=0,\ldots,m-1$ together with $\beta_\nu$ for $\nu\in\mathrm{Box}'(\Sigma)$.
\end{lemma}

We call these generators of $\pi_2(\bX,L)$ the {\em basic disk classes}; they are the analogue of Maslov index two disk classes in toric manifolds. Basic disk classes were used in \cite{CP} to define the {\em leading order bulk orbi-potential}, and it can be used to determine the Floer homology of torus fibers with suitable bulk deformations.

Let
$\CM^{main}_{k+1,l}(L,\beta,\bx)$
be the moduli space of good representable stable maps from bordered orbifold Riemann surfaces of genus zero with $k+1$ boundary marked points $z_0,z_1\ldots,z_k$ and $l$ interior (orbifold) marked points $z_1^+,\ldots,z_l^+$ in the homotopy class $\beta$ of type $\bx = (\mathcal{X}_{\nu_1},\ldots, \mathcal{X}_{\nu_l})$. Here, the superscript ``$main$" indicates that we have chosen a connected component on which the boundary marked points respect the cyclic order of $S^1=\partial D^2$.
By \cite{CP}, $\CM^{main}_{k+1,l}(L,\beta,\bx)$ has a Kuranishi structure of real virtual dimension
\begin{equation}\label{eq:dimes}
n + \mu_{CW}(\beta) + k + 1 + 2l -3 - 2\sum_{j=1}^l\mathrm{age}(\nu_j).
\end{equation}

By \cite[Proposition 9.4]{CP}, if $\CM^{main}_{1,1}(L,\beta)$ is non-empty and if $\partial\beta$ is not in the sublattice generated by $\bb_0, \ldots, \bb_{m-1}$, then there exist $\nu\in \mathrm{Box}'(\Sigma)$, $k_i\in\N$ ($i=0,\ldots,m-1$) and $\alpha\in H_2^\textrm{eff}(\bX)$ such that
$\beta = \beta_\nu + \sum_{i=0}^{m-1} k_i \beta_i + \alpha,$
where $\alpha$ is realized by a union of holomorphic (orbi-)spheres. The CW Maslov index of $\beta$ written in this way is given by
$\mu_{CW}(\beta) = 2\textrm{age}(\nu) + 2 \sum_{i=0}^{m-1} k_i + 2 c_1(\bX)\cdot\alpha.$

\subsection{The invariants}\label{subsec:disk_inv}
Let $\bX_{\nu_1}, \ldots, \bX_{\nu_l}$ be twisted sectors of the toric orbifold $\bX$. Consider the moduli space $\CM^{main}_{1,l}(L,\beta,\bx)$ of good representable stable maps from bordered orbifold Riemann surfaces of genus zero with one boundary marked point and $l$ interior orbifold marked points of type $\bx = (\bX_{\nu_1},\ldots, \bX_{\nu_l})$ representing the class $\beta\in\pi_2(\bX, L)$. By \cite{CP},  $\CM^{main}_{1,l}(L,\beta,\bx)$ carries a virtual fundamental chain, which vanishes unless the following equality holds:
\begin{equation}\label{critmaslov}
\mu_{CW}(\beta) =  2 + \sum_{j=1}^l (2\mathrm{age}(\nu_j)-2).
\end{equation}


\begin{defn}
An orbifold $\bX$ is called {\em Gorenstein} if its canonical divisor $K_\bX$ is Cartier.
\end{defn}

For a Gorenstein orbifold, the age of every twisted sector is a non-negative integer. Now we assume that the toric orbifold $\bX$ is semi-Fano (see Definition \ref{defn:sF}) and Gorenstein. Then a basic orbi-disk class $\beta_\nu$ has Maslov index $2 \mathrm{age}(\nu) \geq 2$, and hence every non-constant stable disk class has at least Maslov index $2$.

We further restrict to the case where all the interior orbifold marked points are mapped to age-one twisted sectors, i.e. the type $\bx$ consists of twisted sectors with $\textrm{age}=1$. This will be enough for our purpose of constructing the mirror over $H^2_{\mathrm{CR}}(\bX)$. In this case, the virtual fundamental chain $[\CM^{main}_{1,l}(L,\beta,\bx)]^\mathrm{vir}$ is non-zero only when $\mu_{CW}(\beta) =  2$, and in fact we get a virtual fundamental {\em cycle} because $\beta$ attains the minimal Maslov index and thus disk bubbling does not occur. Therefore the following definition of {\em genus 0 open orbifold GW invariants} (also termed {\em orbi-disk invariants}) is independent of the choice of perturbations of the Kuranishi structures (in the general case one may restrict to torus-equivariant perturbations to make sense of the following definition following Fukaya-Oh-Ohta-Ono \cite{FOOO1, FOOO2, FOOO10b}):

\begin{defn}[Orbi-disk invariants]\label{defn:orbi-disk_inv}
Let $\beta\in\pi_2(\bX,L)$ be a relative homotopy class with Maslov index given by \eqref{critmaslov}. Suppose that the moduli space $\CM_{1,l}(L,\beta,\bx)$ is compact. Then we define $n_{1,l,\beta}^\bX([\mathrm{pt}]_{L};\mathbf{1}_{\nu_1},\ldots,\mathbf{1}_{\nu_l})\in\rat$ to be the push-forward
$$n_{1,l,\beta}^\bX([\mathrm{pt}]_{L};\mathbf{1}_{\nu_1},\ldots,\mathbf{1}_{\nu_l}):=
ev_{0*}\left([\CM_{1,l}(L,\beta,\bx)]^\mathrm{vir}\right)\in H_n(L;\rat)\cong\rat,$$
where $ev_0:\CM^{main}_{1,l}(L,\beta,\bx)\to L$ is evaluation at the boundary marked point, $[\mathrm{pt}]_{L} \in H^n(L;\rat)$ is the point class of the Lagrangian torus fiber $L$, and $\mathbf{1}_{\nu_j} \in H^0(\bX_{\nu_j};\rat)\subset H^{2\mathrm{age}(\nu_j)}_{\mathrm{CR}}(\bX;\rat)$ is the fundamental class of the twisted sector $\bX_{\nu_j}$.
\end{defn}

\begin{remark}
For the cases we need in this paper, compactness of the disk moduli space $\CM_{1,l}(L,\beta,\bx)$ will be proved in Proposition \ref{prop:cpt_disk_moduli} and Corollary \ref{cor:cpt_disk_moduli}.
\end{remark}

\begin{remark}
The Kuranishi structures in this paper are the same as those defined in \cite{FOOO1, FOOO2} (we refer the readers to \cite[Appendix]{FOOO_book} and \cite{FOOO12} for the detailed construction, and also to \cite{McDuff-Wehrheim12} (and its forthcoming sequels) for a different approach). But the moduli spaces considered here are in fact much simpler than those in \cite{FOOO1, FOOO2} (and \cite{FOOO_book}) because we only need to consider stable disks with just one disk component which is minimal, and hence disk bubbling does not occur. Also, we consider only disk counting invariants, but not
the whole $A_\infty$ structure; this reduces the problem to studying moduli spaces of virtual dimensions 0 or 1, which simplifies several issues involved.
\end{remark}

For a basic (orbi-)disk with at most one interior orbifold marked point, the corresponding moduli space $\CM_{1,0}(L,\beta_i)$ (or $\CM_{1,1}(L,\beta_\nu,\nu)$ when $\beta_\nu$ is a basic orbi-disk class) is regular and can be identified with $L$.  Thus the associated invariants are evaluated as follows \cite{CP}:
\begin{enumerate}
\item
For $\nu \in \mathrm{Box}'$, we have $n_{1,1,\beta_\nu}^\bX([\mathrm{pt}]_{L};\mathbf{1}_\nu) = 1.$
\item
For $i\in \{0,\ldots,m-1\}$, we have $n_{1,0,\beta_i}^\bX([\mathrm{pt}]_{L}) = 1.$
\end{enumerate}
When there are more interior orbifold marked points or when the disk class is not basic, the corresponding moduli space is in general non-regular and virtual theory is involved in the definition, making the invariant much more difficult to compute.  One primary aim of this paper is to compute all these invariants for toric CY orbifolds.

\section{Gross fibrations for toric CY orbifolds}\label{sec:orb_Gross_fib}
 The first ingredient needed for the SYZ construction is a Lagrangian torus fibration. For a toric CY manifold, such fibrations were constructed by Gross \cite{gross_examples} and Goldstein \cite{goldstein} independently.  In this section we generalize their constructions to toric CY orbifolds; cf. the manifold case as discussed in \cite[Sections 4.1-4.5]{CLL}. 

\subsection{Toric CY orbifolds} \label{sec_torCYorb}

\begin{defn}
A Gorenstein toric orbifold $\bX$ is called {\em Calabi-Yau (CY)} if there exists a dual vector $\UL{\nu} \in M=N^\vee=\mathrm{Hom}(N, \Z)$ such that $(\UL{\nu}, \bb_i) =1$ for all stacky vectors $\bb_i$.
\end{defn}

Let $\bX$ be a toric CY orbifold associated to a stacky fan $(\Sigma,\bb_0,\ldots,\bb_{m-1})$. Since $\bb_i=c_i v_i$ for some primitive vector $v_i\in N$ and $(\UL{\nu}, v_i) \in \Z$, we have $c_i=1$ for all $i=0,\ldots,m-1$. Therefore toric CY orbifolds are always simplicial.


For the purpose of this paper, we will always assume that the coarse moduli space of the toric CY orbifold $\bX$ is {\em semi-projective} (Definition \ref{defn:semi_proj}).

\begin{setting}[Partial resolutions of toric Gorenstein canonical singularities]\label{setting:toricCY}
Let $\sigma\subset N_\R$ be a strongly convex rational polyhedral Gorenstein canonical cone with primitive generators $\{\tilde{\bb}_i\}\subset N$.
Here, {\em strongly convex} means that the cone $\sigma$ is convex in $N_\R$ and does not contain any whole straight line;
while {\em Gorenstein canonical} means that there exists $\underline{\nu} \in M$ such that
$\pairing{\underline{\nu}}{\tilde{\bb}_i} = 1$ for all $i$, and $\pairing{\underline{\nu}}{v} \geq 1$ for all $v \in \sigma \cap (N \setminus \{0\})$. We denote by $\mathcal{P} \subset N_\R$
the convex hull of $\{\tilde{\bb}_i\} \subset N$ in the hyperplane $\{v\in N_\R \mid (\underline{\nu},v)=1 \} \subset N_\R$. $\mathcal{P}$ is an $(n-1)$-dimensional lattice polytope.

Let $\Sigma\subset N_\R$ be a simplicial refinement of $\sigma$ obtained by taking the cones over a triangulation of $\mathcal{P}$ (where all vertices of the triangulation belong to $\mathcal{P} \cap N$). Then $\Sigma$ together with the collection
$$\{\bb_i \mid i=0, \ldots, m-1\}\subset N$$
of primitive generators of rays in $\Sigma$ is a stacky fan. The associated toric orbifold $\bX=\bX_\Sigma$ is Gorenstein and CY.

By relabeling the $\bb_i$'s if necessary, we assume that $\{\bb_0, \ldots, \bb_{n-1}\}$ generates a top-dimensional cone in $\Sigma$ and hence forms a rational basis of $N_\rat := N\otimes_\Z\rat$.
\end{setting}

\begin{prop} \label{prop:sproj}
The coarse moduli space of a toric CY orbifold $\bX$ is semi-projective if and only if $\bX$ satisfies Setting \ref{setting:toricCY}.
\end{prop}
\begin{proof}
If $\bX$ satisfies Setting \ref{setting:toricCY}, it is clear that its fan has full-dimensional convex support. Moreover, $\bX$ can be constructed by using its moment map polytope, so its coarse moduli space is semi-projective.

Conversely, suppose that the coarse moduli space of $\bX$ is semi-projective. Since $\bX$ is Gorenstein, there exists $\underline{\nu} \in M$ such that $\pairing{\underline{\nu}}{\bb_i} = 1$ for all primitive generators $\bb_i$ of rays in $\Sigma$.  Then the convex hull of $\bb_i$'s in the hyperplane $\{\pairing{\underline{\nu}}{\cdot} = 1\} \subset N_\R$ defines a lattice polytope $\mathcal{P}$, and the support of the fan is equal to the cone $\sigma$ over this lattice polytope by convexity of the fan. Obviously, the cone $\sigma$ is strongly convex and Gorenstein. Also the fan of $\bX$ is obtained by a triangulation of the lattice polytope $\mathcal{P}$.
\end{proof}

For the rest of this paper, $\bX$ will be a toric CY orbifold as in Setting \ref{setting:toricCY}. This implies that Assumption \ref{assumption} is satisfied: If $\mathcal{P}$ has no interior lattice points, then clearly $\{0\}\cup (\mathcal{P} \cap N)$ generate the lattice $N$.  Otherwise we can inductively find a minimal simplex contained in $\mathcal{P}$ which does not contain any interior lattice points, and it follows that $\{0\}\cup (\mathcal{P} \cap N)$ generate the lattice $N$.

Without loss of generality we may assume that $\underline{\nu} = (0,1) \in M \simeq \Z^{n-1}\oplus \Z$ so that $\mathcal{P}$ is contained in the hyperplane $\{v\in N_\R \mid \pairing{(0,1)}{v} = 1 \}$.  We enumerate
$$\mathrm{Box}'(\Sigma)^{\mathrm{age}=1} := \{\nu \in \textrm{Box}'(\Sigma) \mid \textrm{age}(\nu) = 1\} = \{\bb_m,\ldots,\bb_{m'-1}\}$$
and choose $\bb_m,\ldots,\bb_{m'-1}$ to be the extra vectors so that
$$\mathcal{P} \cap N = \{\bb_0,\ldots,\bb_{m-1},\bb_m,\ldots,\bb_{m'-1}\}.$$

\subsection{The Gross fibration}
In this section we construct a special Lagrangian torus fibration on a toric CY orbifold $\bX$. This is a fairly straightforward generalization of the constructions of Gross \cite{gross_examples} and Goldstein \cite{goldstein} to the orbifold setting.

To begin with, notice that the vector $\UL{\nu}\in M$ corresponds to a holomorphic function on $\bX$ which we denote by $w: \bX\to \mathbb{C}$. The following two lemmas are easy generalizations of the corresponding statements for toric CY manifolds \cite{CLL}, so we omit their proofs.

\begin{lemma}[cf. \cite{CLL}, Proposition 4.2]\label{lemma:w}
The function  $w$ on $\bX$ corresponding to $\underline{\nu} \in M$ is holomorphic, and its zero divisor $(w)$ is precisely the anticanonical divisor $-K_{\bX} = \sum_{i=0}^{m-1} D_i$.
\end{lemma}

\begin{lemma}[cf. \cite{CLL}, Proposition 4.3]
For the dual basis $\{u_0,\ldots, u_{n-1}\} \subset M_\Q:= M\otimes_\Z\Q$ of the basis $\{\bb_0, \ldots, \bb_{n-1}\}$, denote
by $\zeta_j$ the corresponding meromorphic function to $u_j$.
Then
$
d\zeta_0 \wedge \cdots \wedge d\zeta_{n-1} $
extends to a nowhere-zero holomorphic $n$-form $\Omega$ on $\bX$.
\end{lemma}

Next, we equip $\bX$ with a toric K\"ahler structure $\omega$ and consider the associated moment map $\mu_0:\bX \to P$, where $P$ is the moment polytope defined by a system of inequalities:
$$(\bb_i,\cdot) \geq c_i,\quad i=0,\ldots, m-1.$$
Consider the subtorus $T^{\perp \UL{\nu}}:=N_\R^{\perp \UL{\nu}}/N^{\perp \UL{\nu}}\subset N_\mathbb{R}/N$. The moment map of the $T^{\perp \UL{\nu}}$ action is given by composing $\mu_0$ with the natural quotient map:
$$[\mu_0]:\bX \overset{\mu_0}{\longrightarrow} M_\R\to M_\R/\R\langle\UL{\nu}\rangle.$$
The following is a generalization of the Gross fibration for toric CY manifolds \cite{goldstein, gross_examples}, which gives a Lagrangian torus fibration (SYZ fibration).
\begin{defn}\label{defn:gross_fibration}
Fix $K_2>0$. A {\em Gross fibration} of $\bX$ is defined to be
\begin{equation*}
\begin{split}
\mu:\bX \to ( M_\R/\R\langle\UL{\nu}\rangle)  \times  \R_{\geq -K_2^2}, \quad x \mapsto ([\mu_0(x)] , |w(x) - K_2|^2-K_2^2).
\end{split}
\end{equation*}
We denote by ${B}:=(M_\R/\R\langle\UL{\nu}\rangle) \times\R_{\geq -K_2^2}$ the base of the Gross fibration $\mu$.
\end{defn}

Since the holomorphic function $w$ vanishes on the toric prime divisors $D_i \subset \bX$, the images of $D_i \subset \bX$ under the map $\mu$ have second coordinate zero. Moreover, the hypersurface defined by $w(x)=K_2$ maps to the boundary of the image of $\mu$.

\begin{prop}
With respect to the holomorphic volume form $\Omega/(w-K_2)$ defined on $\mu^{-1}(B^{int})$ and the toric K\"ahler form $\omega$, the map $\mu$ is a special Lagrangian torus fibration.
\end{prop}
This proposition can be proved in exactly the same way as in the manifold case (cf. \cite[Theorem 2.4]{gross_examples} or \cite[Proposition 4.7]{CLL}).  It follows from the construction of symplectic reduction: The function $w$ descends to the symplectic reduction $\bX // T^{\perp \UL{\nu}} \to \cpx$; since the circles centered at $K_2$ are special Lagrangian with respect to the volume form $\der\log (w-K_2)$, it follows that their preimages are also special Lagrangian in $\bX$ with respect to the holomorphic volume form $\Omega/(w-K_2)$.

\subsubsection{Discriminant locus and local trivialization}
For each $\emptyset\neq I\subset \{0,\ldots,m-1\}$ such that $\{\bb_i \mid i\in I\}$ generates a cone in $\Sigma$, we define
\begin{equation}\label{eqn:face_of_P}
T_I:=\{\xi\in P \mid (\bb_i,\xi) = c_i,\ i\in I\}\subset \partial P.
\end{equation}
$T_I$ is a codimension-$(|I|-1)$ face of $\partial P$. Let $[T_I]:=[\mu_0](T_I)$.

\begin{prop}
The discriminant locus of the Gross fibration $\mu$ is given by
$$\Gamma:=\{r\in B \mid r \textrm{ is a critical value of }\mu\}=\partial B \cup \left(\left(\bigcup_{|I|=2}[T_I] \right)\times \{0\}\right).$$
\end{prop}
\begin{proof}
This is similar to the proof of \cite[Proposition 4.9]{CLL} in the manifold case: A fiber degenerates when the $T^{\perp \UL{\nu}}$-orbit degenerates or $|w - K_2| = 0$. A $T^{\perp \UL{\nu}}$-orbit degenerates if and only if $w = 0$ and $[\mu_0] \in \left(\bigcup_{|I|=2}[T_I] \right)$; $|w - K_2| = 0$ implies that the base point is located in $\partial B$.
\end{proof}

Put $B_0:=B\setminus \Gamma$. By the arguments in \cite[Section 2.1]{CLL}, the restriction $\mu: \bX_0:=\mu^{-1}(B_0)\to B_0$ is a torus bundle. For facets $T_0,\ldots,T_{m-1}$ of $P$, consider the following open subsets of $B_0$:
\begin{equation} \label{U_i}
U_i := B_0 \setminus \bigcup_{k\neq i} ([T_k] \times \{0\}).
\end{equation}
The torus bundle $\mu$ over each $U_i$ can be explicitly trivialized. Without loss of generality we describe this explicit trivialization over $U_0$.

\begin{defn}\label{defn:choice_basis}
We choose $\underline{v}_1,\ldots,\underline{v}_{n-1}\in N$ such that
\begin{enumerate}
\item
$\{\bb_0\}\cup\{\underline{v}_1,\ldots,\underline{v}_{n-1}\}$ is an integral basis of $N$;
\item
$(\underline{v}_i, \underline{\nu})=0$ for $1\leq i\leq n-1$.
\end{enumerate}
Let $\{\nu_0, \ldots, \nu_{n-1}\}\subset M$ be the dual basis of $\{\bb_0\}\cup\{\underline{v}_1,\ldots,\underline{v}_{n-1}\}$.
\end{defn}

\begin{defn}\label{defn:trivialization}
Denote
$$T^{\perp \bb_0}:=\frac{N_\R/\R\langle \bb_0\rangle}{N/\Z\langle \bb_0\rangle}.$$
Then, over $U_0$, we have a trivialization
$$\mu^{-1}(U_0) \cong U_0 \times T^{\perp \bb_0} \times (\R/2\pi\Z).$$
Here the first map is given by $\mu$, the last map is given by $\textrm{arg}(w-K_2)$, and the second map is given by the argument over $2\pi$ of the meromorphic functions corresponding to
$\nu_1, \ldots, \nu_{n-1}$.
\end{defn}

\subsubsection{Generators of homotopy groups} \label{sec:gen_htpy}
Fix $r_0:=(q_1,q_2)\in U_0$ with $q_2>0$. Consider the fiber $F_{r_0}:=\mu^{-1}(q_1, q_2)$. By the trivialization in Definition \ref{defn:trivialization}, we have $F_{r_0}\simeq T^{\perp \bb_0}\times (\R/2\pi\Z)$. Hence $\pi_1(F_{r_0})\simeq N/\Z\langle \bb_0\rangle\times \Z$ has the following basis (over $\Q$)
$$\{\lambda_i \mid 0\leq i \leq n-1\},\quad \lambda_0=(0,1), \lambda_i=([\underline{v}_i],0) \text{ for } 1\leq i\leq n-1.$$
Recall that for a regular fiber $L$ of the moment map $\bX\to P$, the basic disk classes form a natural basis of $\pi_2(\bX, L)$ (Lemma \ref{lem:basic_disk_classes}). 
Then the explicit Lagrangian isotopy between $F_{r_0}$ and $L$:
\begin{equation}\label{eqn:Lag_isotopy_X}
L_t:=\{x\in \bX \mid [\mu_0(x)]=q_1,\ |w(x)-t|^2=K_2^2+q_2\}, \quad t\in [0, K_2]
\end{equation}
allows us to identify $\pi_2(\bX, F_{r_0})$ with $\pi_2(\bX,L)$ and view the basic disk classes in $\pi_2(\bX, L)$ as classes in $\pi_2(\bX, F_{r_0})$. By abuse of notation, we still denote these classes by $\beta_0, \ldots, \beta_{m-1}$ and $\{\beta_\nu \mid \nu\in \mathrm{Box}'(\Sigma)\}$.
For a general $r\in U_0$, a basis for $\pi_2(\bX, F_{r})$ may be obtained by identifying $F_{r}$ with $F_{r_0}$ using the trivialization in Definition \ref{defn:trivialization}.
\begin{lemma} \label{boundary}
For a fiber $F_r$ of $\mu$ where $r \in U_0$, the boundary of the disk classes are 
\begin{equation*}
\begin{split}
&\partial \beta_j=\lambda_0+\sum_{i=1}^{n-1}(\nu_i, \bb_j)\lambda_i, \quad 0\leq j\leq m-1\\
&\partial \beta_\nu= \lambda_0+\sum_{i=1}^{n-1}(\nu_i, \nu)\lambda_i, \quad \nu=\sum_{i=1}^{n-1}(\nu_i,\nu) \underline{v}_i\in \mathrm{Box}'(\Sigma).
\end{split}
\end{equation*}
\end{lemma}
\begin{proof}
Similar to the proof of \cite[Proposition 4.12]{CLL}.
\end{proof}


\subsubsection{Wall-crossing of orbi-disk invariants}\label{sec:wall_crossing_X}
Like the manifold case, the behavior of disk invariants with boundary conditions on a fiber $F_r$ depends on the location of $r$. In this section we study this behavior for the Gross fibration $\mu:\bX\to B$ of a toric CY orbifold.

Let $\beta\in \pi_2(\bX, F_r)$ be a class represented by a stable disk. Then $\beta=\sum_i u_i +\alpha$ where $u_i$'s are disk classes and $\alpha$ is the class of a rational curve, so that $\mu_{CW}(\beta)=\sum_i \mu_{CW}(u_i)+2c_1(\bX) \cdot \alpha$. Since $\bX$ is CY, $c_1(\bX)\cdot\alpha=0$. The fiber $F_r\subset \bX$ is special Lagrangian with respect to the meromorphic form $\Omega/(w-K_2)$. Since the pole divisor of $\Omega/(w-K_2)$ is $\tilde{D}_0:=\{w(x)=K_2\}\subset \bX$, Lemma \ref{lem:maslov_ind_comp} implies that $\mu_{CW}(u_i)=2 u_i\cdot \tilde{D}_0 \geq 0$.  Thus we have
\begin{lemma}
If a class $\beta\in \pi_2(\bX, F_r)$ is represented by a stable disk, then $\mu_{CW}(\beta)\geq 0$.
\end{lemma}

The following result describes when the minimal Maslov index $0$ can be achieved.

\begin{lemma}
Let $r=(q_1,q_2)\in B_0$. Consider the fiber $F_r$.
\begin{enumerate}
\item
$F_r$ bounds a non-constant stable disk of CW Maslov index $0$ if and only if $q_2=0$.
\item
If $q_2\neq 0$, then $F_r$ has minimal CW Maslov index at least $2$, i.e. $F_r$ does not bound any non-constant stable disks with CW Maslov index less than $2$.
\end{enumerate}
\end{lemma}
\begin{proof}
By the observation that for a holomorphic orbi-disk $u:\mathcal{D} \to \bX$, the composition $w\circ u: \mathcal{D}\to \C$ is a holomorphic function on every local chart of $\mathcal{D}$ and is invariant under the action of the local groups and by the maximum principle, this lemma can be proved as in the manifold case; cf. \cite[Lemma 4.27 and Corollary 4.28]{CLL}.
\end{proof}

By definition, the {\em wall} of a Lagrangian fibration $\mu:\bX\to B$ is the locus $H\subset B_0$ of all $r\in B_0$ such that the Lagrangian fiber $F_r$ bounds a non-constant stable disk of CW Maslov index $0$. The above lemma shows that
$$H=M_\R/\R\langle \underline{\nu}\rangle\times \{0\}.$$
The complement $B_0\setminus H$ is the union of two connected components
\begin{equation*}
B_+:= M_\R/\R\langle \underline{\nu}\rangle\times (0, +\infty),\quad B_-:= M_\R/\R\langle \underline{\nu}\rangle\times (-K_2^2, 0).
\end{equation*}
For $r\in B_0\setminus H$, orbi-disk invariants with arbitrary numbers of age-one insertions are well-defined for relative homotopy classes with CW Maslov index $2$. We need to consider the two possibilities, namely $r\in B_+$ and $r\in B_-$.

{\noindent \em Case 1: $r\in B_+$.}
Let $r=(q_1,q_2)\in B_+$, namely $q_2>0$. Then \eqref{eqn:Lag_isotopy_X} gives a Lagrangian isotopy between the fiber $F_r$ and a regular Lagrangian torus fiber $L$. Furthermore, since $q_2>0$, for each $t\in [0, K_2]$, $w$ is never $0$ on $L_t$. It follows that the Lagrangians $L_t$ in the isotopy do not bound non-constant disks of CW Maslov index $0$. Hence for $r\in B_+$, the orbi-disk invariants of $(\bX, F_r)$ with arbitrary numbers of age-one insertions and CW Maslov index $2$ coincide with those of $(\bX, L)$, which are reviewed in Section \ref{subsec:disk_inv}.

{\noindent \em Case 2: $r\in B_-$.}
In this case we have the following
\begin{prop} \label{prop:B_-}
Let $r=(q_1,q_2)\in B_-$, namely $q_2<0$. Let $\beta\in \pi_2(\bX, F_r)$. Suppose $\mathbf{1}_{\nu_1},\ldots,\mathbf{1}_{\nu_l}\in H_{\mathrm{CR}}^*(\bX)$ are fundamental classes of twisted sectors $\bX_{\nu_1},\ldots,\bX_{\nu_l}$ such that $\mathrm{age}(\nu_1)=\cdots=\mathrm{age}(\nu_l)=1$. Then we have
$$n_{1,l,\beta}^{\bX}([\mathrm{pt}]_{F_r}; \mathbf{1}_{\nu_1},\ldots,\mathbf{1}_{\nu_l})
=\left\{
\begin{array}{ll}
1 & \textrm{ if } \beta=\beta_0\textrm{ and }l = 0\\
0 &  \textrm{ otherwise }.
\end{array}\right.$$
\end{prop}
\begin{proof}
By dimension reason, we may assume that $\mu_{CW}(\beta)=2$.

Let $u: (\mathcal{D}, \partial \mathcal{D})\to (\bX, F_r)$ be a non-constant holomorphic orbi-disk. Then the composition $(w-K_2)\circ u$ descends to a holomorphic function $\overline{(w-K_2)\circ u}: |\mathcal{D}|\to \C$ on the smooth disk $|\mathcal{D}|$ underlying $\mathcal{D}$. Since $r\in B_-$, $|w-K_2|$ is constant on $F_r$ with value less than $K_2$. Since $u(\partial |\mathcal{D}|)=u(\partial \mathcal{D})\subset F_r$, we have $|\overline{(w-K_2)\circ u}|<K_2$ on $\partial |\mathcal{D}|$. By the maximum principle, $|\overline{(w-K_2)\circ u}|<K_2$ on the whole $|\mathcal{D}|$. Hence the image of $u$ is contained in $S_-:=\mu^{-1}(\{(q_1,q_2)\in B \mid q_2<0\})$. Also observe that $u(\mathcal{D})$ must intersect $\tilde{D}_0:=\{w(x)=K_2\}\subset \bX$. Since the hypersurface $w(x)=K_2$ does not contain orbifold points, we have $u(\mathcal{D})\cdot \tilde{D}_0\in \Z_{> 0}$. Lemma \ref{lem:maslov_ind_comp} implies that $u$ is of CW Maslov index at least $2$.

Let $h: \mathcal{C}\to \bX$ be a non-constant holomorphic map from an orbifold sphere $\mathcal{C}$. Then $h(\mathcal{C})\cap S_-=\emptyset$. To see this, we consider $w\circ h$, which descends to a holomorphic function $\overline{w\circ h}$ on the $\proj^1$ underlying $\mathcal{C}$. Since $\overline{w\circ h}$ must be a constant function, the image $h(\mathcal{C})$ is contained in a level set $w^{-1}(c)$ for some $c\in \C$. For $c\neq 0$, we have $w^{-1}(c)\simeq (\C^\times)^{n-1}$ which does not support non-constant holomorphic spheres, so $c=0$. Now we conclude by noting that $w^{-1}(0)\cap S_-=\emptyset$.

Now let $v\in \CM^{main}_{1,l}(F_r,\beta,(\bX_{\nu_1},\ldots,\bX_{\nu_l}))$ be a stable orbi-disk of CW Maslov index $2$. As explained above, each orbi-disk component contributes at least $2$ to the CW Maslov index. Hence $v$ can have only one orbi-disk component. Also, a non-constant holomorphic orbi-sphere in $\bX$ cannot meet an orbi-disk, so $v$ does not have any orbi-sphere components. This shows that for any $\beta\in \pi_2(\bX, F_r)$ of Maslov index $2$, the moduli space $\CM^{main}_{1,l}(F_r,\beta,(\bX_{\nu_1}, \ldots, \bX_{\nu_l}))$ parametrizes only orbi-disks, and all these orbi-disks are contained in $S_-$ and do not meet the toric divisors $D_1, \ldots, D_{m-1}$. Since each orbifold point on the orbi-disk of type $\nu\in \textrm{Box}'(\Sigma)$ contributes $2\mathrm{age}(\nu)$ to the CW Maslov index $\mu_{CW}(\beta)$, and we assumed $\mathrm{age}(\nu)=1$ and $\mu_{CW}(\beta)=2$, there are no orbifold points on the disk.

Recall that relative homotopy classes $\beta_\nu$ can be written as (fractional) linear combinations of $\beta_0, \ldots, \beta_{m-1}$ with non-negative coefficients. Thus, the class $\beta$ of any orbi-disk can be written as a linear combination of $\beta_0, \ldots, \beta_{m-1}$ with non-negative coefficients. Hence, from the fact that intersection numbers of $\beta$ with the divisors $D_1, \ldots, D_{m-1}$ are zero, we may conclude that $\beta=k\beta_0$ for some $k \geq 0$, and $\mu(\beta)=2$ implies that $k=1$ and $\beta=\beta_0$. Holomorphic smooth disks representing the class $\beta_0$ are confined in an affine toric chart. The argument analogous to that in the proof of \cite[Proposition 4.32]{CLL} then shows that the invariant is $1$ in this case. This concludes the proof.
\end{proof}

\subsection{Examples} \label{sect:eg fib}

\begin{enumerate}
\item
$\bX = [\cpx^2 / \Z_m]$.  This is the $A_{m-1}$ surface singularity.  The stacky fan is a cone generated by $(0,1)$ and $(m,1)$ in $N = \Z^2$ (Figure \ref{fig:fan for C^2 mod Z_m}).  By subdividing the cone by the rays generated by $(k,1)$ for $k = 1, \ldots, m-1$, one obtains a resolution of the singularity.  The age-one twisted sectors of $\bX$ are in a one-to-one correspondence with the lattice points $(k,1) \in \mathrm{Box}'$ for $k = 1, \ldots, n-1$.  The Gross fibration of this orbifold is depicted in Figure \ref{fig:fib C^2 mod Z_m}.\\

\begin{figure}[htp]
  \centering
   \begin{subfigure}[b]{0.3\textwidth}
   	\centering
    \includegraphics[width=\textwidth]{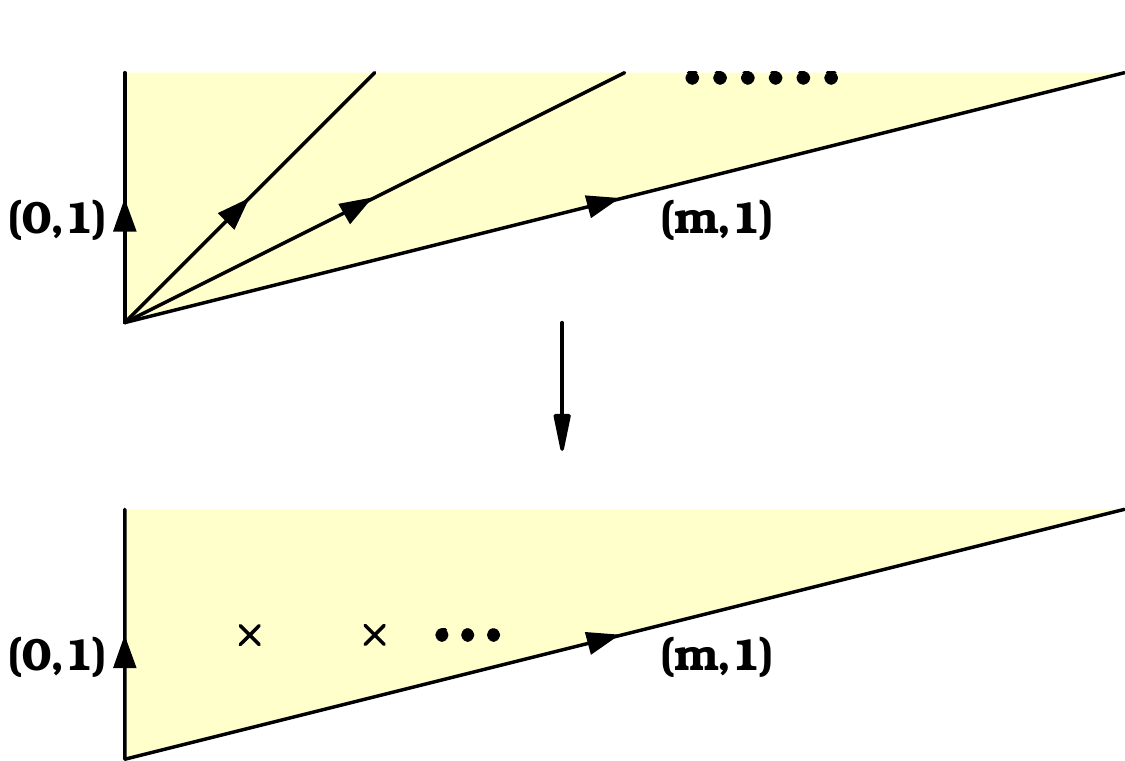}
    \caption{Fans for $[\cpx^2 / \Z_m]$ and its resolution. The crosses represent twisted sectors.}
    \label{fig:fan for C^2 mod Z_m}
   \end{subfigure}
   \hspace{10pt}
   \begin{subfigure}[b]{0.3\textwidth}
   	\centering
    \includegraphics[width=\textwidth]{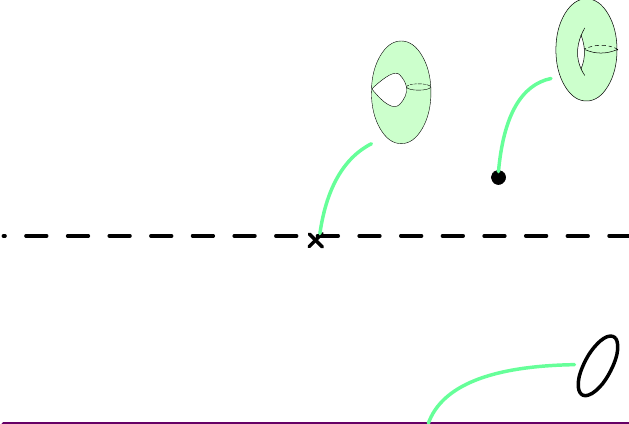}
    \caption{Gross fibration on $[\cpx^2 / \Z_m]$.  The dotted line is the wall and the cross is the discriminant locus.}
    \label{fig:fib C^2 mod Z_m}
   \end{subfigure}
   \caption{$[\cpx^2 / \Z_m]$.}
\end{figure}

\item
$\bX = [\cpx^3 / \Z_{2g+1}]$ for $g \in \N$.  Let $N$ be the lattice $\Z^3 + \Z\left\langle \frac{(1,1,2g-1)}{2g+1}\right\rangle$.
The stacky fan is a cone generated by $(1,0,0), (0,1,0), (0,0,1) \in N$, which is a cone over the convex hull of these 3 vectors in the hyperplane $\{(a,b,c) \in N_\R: a + b + c = 1\}$.  Using the triangulation of the polygon by the lattice points $(k,k,2g+1-2k) / (2g+1)$ (Figure \ref{fig:fan for C^3 mod Z_g}), one obtains a resolution of the orbifold singularity, which is the mirror manifold of a Riemann surface of genus $g$ (see \cite{KKOY09, Efimov}).\footnote{The mirror of a Riemann surface of genus $g$ is a Landau-Ginzburg model, which is a holomorphic function defined on the manifold described here \cite{KKOY09, Efimov}.}  The age-one twisted sectors of $\bX$ are in a one-to-one correspondence with the lattice points $(k,k,2g+1-2k) / (2g+1)\in \mathrm{Box}'$ for $k = 1, \ldots, g$.  See Figure \ref{fig:fib C^3 mod Z_g} for the Gross fibration this orbifold.\\

\begin{figure}[htp]
  \centering
   \begin{subfigure}[b]{0.3\textwidth}
   	\centering
    \includegraphics[width=\textwidth]{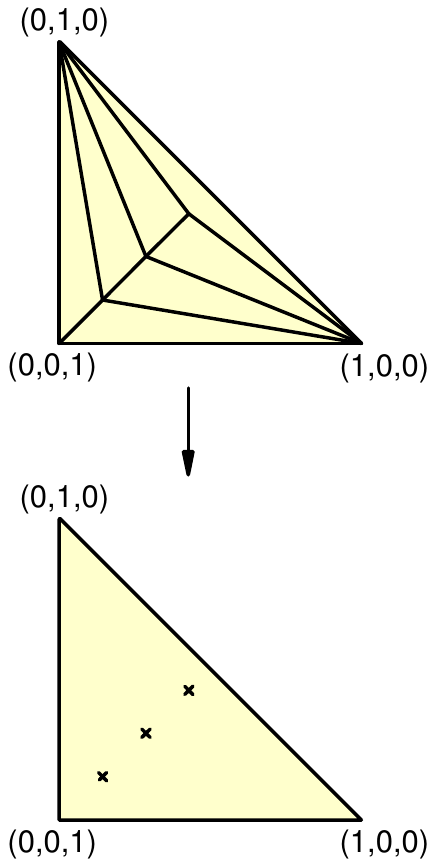}
    \caption{Cones over the polytopes give the fans for $[\cpx^3 / \Z_{2g+1}]$ and its resolution; the crosses represent twisted sectors. This figure is for $g = 3$.}
    \label{fig:fan for C^3 mod Z_g}
   \end{subfigure}
   \hspace{10pt}
   \begin{subfigure}[b]{0.3\textwidth}
   	\centering
    \includegraphics[width=\textwidth]{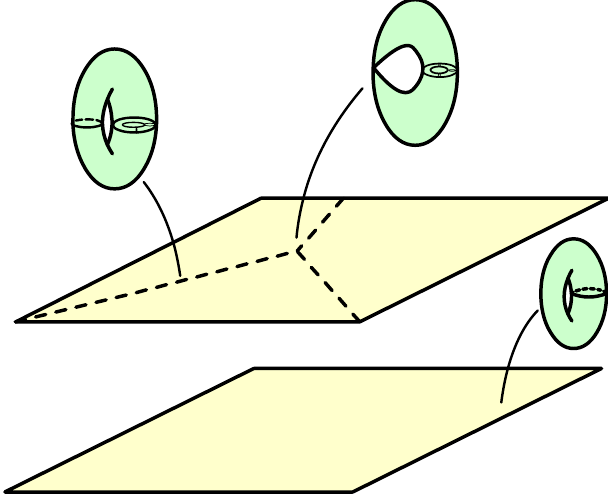}
    \caption{Gross fibration on $[\cpx^3 / \Z_{2g+1}]$ with base an upper-half-space.  The plane in the middle is the wall; the dotted line and the plane at the bottom are the discriminant loci, with singular fibers as shown.}
    \label{fig:fib C^3 mod Z_g}
   \end{subfigure}
   \caption{$[\cpx^3 / \Z_{2g+1}]$.}
\end{figure}

\item
$\bX = [\cpx^n / \Z_n]$ for $n \in \Z$.  
The stacky fan is a cone generated by $(e_1,1), \ldots, (e_{n-1},1), (-e_1 - \cdots - e_{n-1},1) \in N = \Z^{n-1} \times \Z$, where $\{e_i\}$ is the standard basis of $\Z^{n-1}$.  One obtains a resolution of the orbifold singularity by subdividing the cone by the ray generated by $(0,1) \in N$, and the resulting manifold is the total space of canonical line bundle over $\proj^{n-1}$.  There is only one age-one twisted sector, namely the lattice point $(0,1) \in \mathrm{Box}'$.
\end{enumerate}

\section{SYZ mirror construction}\label{sec:SYZ_mirror}

In this section we carry out the SYZ mirror construction for toric CY orbifolds; cf. the manifold case as discussed in \cite[Sections 4.6]{CLL}. The procedure may be summarized as follows. Let $\bX$ be a toric CY orbifold as in Setting \ref{setting:toricCY}, and $\mu:\bX\to B$ be the Gross fibration in Definition \ref{defn:gross_fibration}.
\begin{enumerate}
\item[Step 1.]
Consider the torus bundle $\mu: \bX_0\to B_0$. Take the dual torus bundle $\check{\mu}: \check{\bX_0}\to B_0$. The total space $\check{\bX_0}$ together with its canonical complex structure is called the {\em semi-flat mirror} of $\bX$. 

\item[Step 2.]
Construct instanton corrections to the semi-flat complex coordinates by taking family Fourier transforms of generating functions of genus 0 open orbifold GW invariants. 

\item[Step 3.]
Take the spectrum of the coordinate ring generated by the instanton-corrected complex coordinates to be the mirror.
\end{enumerate}

More precisely, we need a toric partial compactification of $\bX$ and the corresponding Lagrangian fibration to obtain sufficiently many complex coordinates of the mirror. This was explained for the manifold case in \cite[Section 4.3]{CLL}, and will not be repeated here.

Realizing the SYZ construction by using symplectic geometry and open GW invariants was pioneered by Auroux in \cite{auroux07, auroux09}.  The above procedure was proposed in \cite{CLL} and applied to all toric CY manifolds; see also Abouzaid-Auroux-Katzarkov \cite{AAK12}. We carry out this construction for toric CY orbifolds in the remainder of this section.

\subsection{The semi-flat mirror}
We construct the semi-flat mirror of $\bX$ as follows. Consider the torus bundle $\mu: \bX_0:=(\mu)^{-1}(B_0)\to B_0$. Let $\check{\bX_0}$ be the space of pairs $(F_r,\nabla)$, where $F_r:=\mu^{-1}(r), r\in B_0$ and $\nabla$ is a flat $U(1)$-connection on the trivial complex line bundle over $F_r$ up to gauge. There is a natural projection map $\check{\mu}: \check{\bX_0}\to B_0$. We write $\check{F}_r:=\check{\mu}^{-1}(r)$ for $r\in B_0$. According to \cite[Proposition 2.5]{CLL}, $\check{\mu}: \check{\bX_0}\to B_0$ is a torus bundle.

Recall that $B_0$ has an open cover $\{U_i\}$ defined by \eqref{U_i}. We focus on the open set $U=U_0$ and describe the semi-flat complex coordinates on the chart $\check{\mu}^{-1}(U)$. Fix a base point $r_0\in U$. For $r\in U$, consider the class $\lambda_i\in \pi_1(F_r)$ as in Section \ref{sec:gen_htpy}. Define cylinder classes
$[h_i(r)]\in \pi_2((\mu)^{-1}(U), F_{r_0}, F_r)$
as follows. Recall the trivialization in Definition \ref{defn:trivialization}: $(\mu)^{-1}(U) \cong U \times T^{\perp \bb_0} \times (\R/2\pi\Z)$.
Pick a path $\gamma: [0,1]\to U$ with $\gamma(0)=r_0$ and $\gamma(1)=r$. Define
$$h_j: [0,1]\times \R/\Z\to U\times T^{\perp \bb_0} \times (\R/2\pi\Z), \quad h_j(R, \Theta):=\left(\gamma(R), \frac{\Theta}{2\pi}[\underline{v}_j], 0 \right)$$
for $j=1, \ldots, n-1$, and
$$h_0: [0,1]\times \R/\Z\to U\times T^{\perp \bb_0} \times (\R/2\pi\Z), \quad h_0(R, \Theta):= (\gamma(R),0,2\pi\Theta).$$
The classes $[h_i(r)]$ are independent of the choice of $\gamma$. Now the semi-flat complex coordinates of $(\mu)^{-1}(U)$ are given by $z_0, z_1, \ldots, z_{n-1}$ where $z_i(F_r,\nabla):=\exp(\rho_i+2\pi\sqrt{-1}\check{\theta}_i)$.
The semi-flat holomorphic volume form is the nowhere vanishing form $dz_1\wedge dz_2\wedge\cdots\wedge dz_{n-1}\wedge dz_0$ on $(\mu)^{-1}(U)$.
Semi-flat complex coordinates on the other charts $\check{\mu}^{-1}(U_j)$ can be similarly described.

\subsection{Instanton corrections}\label{sec:instanton_corrections}
The semi-flat complex coordinate $z_0$ receives instanton corrections by taking a family version of Fourier transformations of generating functions of genus 0 open orbifold GW invariants which count orbi-disks with CW Maslov index 2. The result is a complex-valued function
$\tilde{z}_0: (\check{\mu})^{-1}(B_0\setminus H)\to \mathbb{C}$
such that for $(F_r,\nabla)\in (\check{\mu})^{-1}(B_0\setminus H)$, the value of $\tilde{z}_0$ is given by
\begin{equation}\label{eqn:corrected_cpx_ccord}
\tilde{z}_0 = \sum_{\beta\in \pi_2(\bX,F_r)} \sum_{l\geq0}\frac{1}{l!} n^{\mathcal{\bX}}_{1,l,\beta}([\mathrm{pt}]_{F_r};\tau,\ldots,\tau) \exp\left(-\int_\beta \omega\right)\mathrm{Hol}_\nabla(\partial\beta),
\end{equation}
where $\tau \in H^2_{\mathrm{CR}}(\bX) \subset H^*_{\mathrm{CR}}(\bX)$ and $\mu_{CW}(\beta)=2$.


When $r\in B_-$, Proposition \ref{prop:B_-} shows that the only non-vanishing genus 0 open GW invariant is $n_{1,0,\beta}=1$ for $\beta=\beta_0$.  Therefore \eqref{eqn:corrected_cpx_ccord} has only one term: $\tilde{z}_0 = \exp\left(-\int_{\beta_0(r_0)} \omega\right)z_0$.
To simplify notations, we put $C_0:= \exp\left(-\int_{\beta_0(r_0)} \omega\right)$.

When $r\in B_+$, there are non-trivial open GW invariants and \eqref{eqn:corrected_cpx_ccord} reads
$$\tilde{z}_0 = z_0\sum_{j=0}^{m-1}C_j(1+\delta_j)\prod_{i=1}^{n-1}z_i^{(\nu_i,\bb_j)}+z_0\sum_{\nu\in \mathrm{Box}'(\Sigma)^{\mathrm{age}=1}}C_\nu (\tau_\nu+\delta_\nu)\prod_{i=1}^{n-1}z_i^{(\nu_i,\nu)},$$
where $C_j := \exp\left(-\int_{\beta_j(r_0)}\omega\right)$ for $0\leq j\leq m-1$ and $C_\nu := \exp\left(-\int_{\beta_\nu(r_0)}\omega\right)$ for $\nu\in \mathrm{Box}'(\Sigma)^{\mathrm{age}=1}$,
and
\begin{equation}\label{eqn:generating_functions_open_GW}
\begin{split}
& 1+\delta_j := \sum_\alpha \sum_{l\geq 0}\sum_{\nu_1,\ldots,\nu_l\in \mathrm{Box}'(\Sigma)^{\mathrm{age}=1}}
\frac{\prod_{i=1}^l\tau_{\nu_i}}{l!} n_{1, l, \beta_j(r)+\alpha}([\mathrm{pt}]_{L}; \prod_{i=1}^l\mathbf{1}_{\nu_i})
\exp\left(-\int_{\alpha}\omega\right),\\
& (0\leq j\leq m-1),\\
& \tau_\nu+\delta_\nu := \sum_\alpha \sum_{l\geq 0}\sum_{\nu_1,\ldots,\nu_l\in \mathrm{Box}'(\Sigma)^{\mathrm{age}=1}}
\frac{\prod_{i=1}^l \tau_{\nu_i}}{l!} n_{1, l, \beta_\nu(r)+\alpha}([\mathrm{pt}]_{L}; \prod_{i=1}^l\mathbf{1}_{\nu_i})
\exp\left(-\int_{\alpha}\omega\right),\\
& (\nu\in \mathrm{Box}'(\Sigma)^{\mathrm{age}=1})
\end{split}
\end{equation}
are generating functions of genus 0 open orbifold GW invariants. Here we use the relation
\begin{equation*}
-\beta_j(r)=-\beta_j(r_0)-[h_0(r)]-\sum_{i=1}^{n-1}(\nu_i,\bb_j)[h_i(r)];
\end{equation*}
also, the generating functions can be written as in the left-hand-sides of \eqref{eqn:generating_functions_open_GW} because
$n_{1, 0, \beta_j(r)}([\mathrm{pt}]_{L}) = n_{1, 1, \beta_\nu(r)}([\mathrm{pt}]_{L}; \mathbf{1}_\nu) = 1$
for any $j$ and $\nu$. Notice that $n_{1, l, \beta_\nu(r)+\alpha}([\mathrm{pt}]_{L}; \prod_{i=1}^l\mathbf{1}_{\nu_i})$ is nonzero only when $l\geq 1$, so the generating function $\tau_\nu + \delta_\nu$ has no constant term.

The above discussion may be summarized as follows. For $0\leq j\leq m-1$ and $\nu\in \mathrm{Box}'(\Sigma)^{\mathrm{age}=1}$ we put $z^{\bb_j}:= \prod_{i=1}^{n-1}z_i^{(\nu_i,\bb_j)}$ and $z^\nu:=\prod_{i=1}^{n-1}z_i^{(\nu_i,\nu)}$.
\begin{prop}
We have
\begin{equation*}
\tilde{z}_0 = \left\{
\begin{array}{ll}
z_0\sum_{j=0}^{m-1}C_j(1+\delta_j)z^{\bb_j}+z_0\sum_{\nu\in \mathrm{Box}'(\Sigma)^{\mathrm{age}=1}}C_\nu (\tau_\nu+\delta_\nu)z^\nu & \text{for $r\in B_+$},\\
C_0z_0 & \text{for $r\in B_-$}.
\end{array}
\right.
\end{equation*}
\end{prop}

\subsection{The mirror}\label{subsec:mirror}
Let $\mathbb{C}[[q,\tau]]$ be the ring of formal power series in the variables
$$\{q_1,\ldots,q_{r}\} \cup \{\tau_\nu \mid \nu\in \mathrm{Box}'(\Sigma)^{\mathrm{age}=1}\},$$
which are parameters in the complexified extended K\"ahler moduli space of $\bX$ (see Section \ref{sec:kahler_moduli} for precise definitions of these parameters), with coefficients in $\mathbb{C}$. Consider $R_+=R_-:=\mathbb{C}[[q,\tau]][z_0^\pm, \ldots, z_{n-1}^\pm]$. Let $R_0$ be the localization of $\mathbb{C}[[q,\tau]][z_0^\pm, \ldots, z_{n-1}^\pm]$ at
\begin{equation*}
g:= \sum_{j=0}^{m-1}C_j(1+\delta_j)z^{\bb_j}+\sum_{\nu\in \mathrm{Box}'(\Sigma)^{\mathrm{age}=1}}C_\nu (\tau_\nu+\delta_\nu)z^\nu.
\end{equation*}
Let $[Id]: R_-\to R_0$ be the localization map. Also define $R_+\to R_0$ by $z_k\mapsto [z_k]$ for $k=1,\ldots,n-1$ and $z_0\mapsto [g^{-1}z_0]$.

Using these two maps, we put
$R:=R_-\times_{R_0}R_+.$
We identify $\tilde{z}_0$ with $u:=(C_0z_0, z_0g)\in R$.  Setting $v:=(C_0^{-1}z_0^{-1}g, z_0^{-1})\in R$, we have
$$R\simeq \mathbb{C}[[q,\tau]][u^\pm, v^\pm, z_1^\pm,\ldots,z_{n-1}^\pm]/\langle uv-g\rangle.$$
Taking the relative spectrum $\mathrm{Spec}(R)$, we obtain
$$ \check{\bX} = \{(u, v, z_1, \ldots, z_{n-1}) \in \C^2 \times (\C^\times)^{n-1}: uv = g(z_1,\ldots,z_{n-1}) \}.$$
This gives the SYZ mirror of {\em the complement of the hypersurface $\{w(x) = K_2\}$ in $\bX$}. The SYZ mirror of the toric CY orbifold $\bX$ itself is given by the {\em Landau-Ginzburg model} $(\check{\bX},W)$, where $W:\check{\bX} \to \C$ is the Fourier transformation of the generating function orbi-disk invariants for classes with CW Maslov index 2, which is simply the holomorphic function defined by $W := u$; see \cite[Section 4.6]{CLL} and \cite[Section 7]{AAK12} for related discussions in the manifold case.


It is not difficult to see that
\begin{prop}\label{prop:coor_change}
There exists a coordinate change such that under the new coordinates, the defining equation $uv=g$ of $\check{\bX}$ can be written as
\begin{equation*}
uv = (1+\delta_0) + \sum_{j=1}^{n-1}(1+\delta_j)z_j + \sum_{j=n}^{m-1}(1+\delta_j)q_j z^{\bb_j}
+ \sum_{\nu\in \mathrm{Box}'(\Sigma)^{\mathrm{age}=1}} (\tau_\nu+\delta_\nu)q^{-D^{\vee}_\nu} z^\nu,
\end{equation*}
where $q_j := q^{\xi_j}$ and $\xi_j \in H_2(\bX;\Q)$ is the class defined by $\bb_j = \sum_{i=0}^{n-1} a_{ji}\bb_i$ for $j = n,\ldots,m-1$, and $q^{-D^{\vee}_\nu} := \prod_{a=1}^{r} q_a^{-\langle p_a,D^\vee_\nu\rangle}$ for $\nu \in \mathrm{Box}'(\Sigma)^{\mathrm{age}=1}$.
\end{prop}

\begin{remark}[Convergence]
A priori the K\"ahler parameters $q_a$'s and the variables $\tau_\nu$'s keeping track of stacky insertions in the generating functions \eqref{eqn:generating_functions_open_GW} are only formal. However, in our case, the generating functions can be shown to be convergent; see Corollary \ref{cor:convergence} below.
\end{remark}

\subsection{Examples} \label{sect:eg_SYZ}
\noindent{(1)} $\bX = [\cpx^2 / \Z_m]$.  The stacky fan and Gross fibration are shown in Figure \ref{fig:fan for C^2 mod Z_m} and \ref{fig:fib C^2 mod Z_m} respectively.  It has $m-1$ twisted sectors of age $1$ which are in bijection with the vectors $\nu_i = (i,1)$ for $i = 1, \ldots, m-1$.  Each twisted sector $\nu_i$ has a corresponding basic orbi-disk class $\beta_{\nu_i}$. The SYZ mirror constructed in this section is
$uv = 1 + z^m + \sum_{j=1}^{m-1} (\tau_j + \delta_{\nu_j}(\tau)) z^j,$
where
\begin{equation}\label{eqn:gen_ex1}
\tau_j + \delta_{\nu_j}(\tau) = \sum_{k_1,\ldots,k_{m-1} \geq 0} \frac{\tau_1^{k_1}\ldots \tau_{m-1}^{k_{m-1}}}{(k_1 + \ldots + k_{m-1})!} n_{1,l,\beta_{\nu_j}} ([\pt]_L; (\mathbf{1}_{\nu_1})^{k_1} \times \ldots \times (\mathbf{1}_{\nu_{m-1}})^{k_{m-1}}),
\end{equation}
$l = k_1 + \ldots + k_g$ and $\tau = \sum_{i=1}^{m-1} \tau_i \mathbf{1}_{\nu_i} \in H^2_{\mathrm{CR}}(\bX)$. All K\"ahler parameters $\tau_i$ are contributed from twisted sectors in this case, and the non-triviality of orbi-disk invariants is also due to the presence of twisted sectors.

The $A_{m-1}$ singularity $X = \cpx^2 / \Z_m$ has a resolution $\tilde{X}$ whose fan and Gross fibration are shown in Figure \ref{fig:fan for C^2 mod Z_m} and \ref{fig:fib C^2 mod Z_m}.  It has $m-1$ irreducible $(-2)$ curves $l_i$'s which have Chern number $0$, and they are in bijection with the primitive generators $(i,1)$, $i=1,\ldots,m-1$. The SYZ mirror of the resolution $\tilde{X}$ is
\begin{equation}
 uv = 1 + z^m + \sum_{j=1}^{m-1} (1 + \delta_{j} (q)) z^j, \quad 1 + \delta_{j} (q) = \sum_{k_1,\ldots,k_{m-1} \geq 0} n_{1,0,\beta_j + \alpha_{k}} q^{\alpha_{k}},
 \end{equation}
and $\alpha_{k} = \sum_{i=1}^{m-1} k_i l_i$ in the above expression.  The K\"ahler parameters $q^{l_i}$'s are given by $\exp (-\int_{l_i} \omega)$, and the non-triviality of disk invariants is due to the presence of rational curves of Chern number zero. The SYZ construction for toric CY surfaces $\tilde{X}$ has been studied in \cite{LLW_surfaces}, where $\delta_{j}$ has been computed explicitly.

\noindent{(2)} $\bX = [\cpx^3 / \Z_{2g+1}]$ for $g \in \N$.  See Figure \ref{fig:fan for C^3 mod Z_g} and \ref{fig:fib C^3 mod Z_g} for the fan and Gross fibration.  It has $g$ twisted sectors of age one which are in one-to-one correspondence with the vectors $\nu_i = (i,i,2g+1-2i)/(2g+1) \in N$ for $i = 1, \ldots, g$. Let $z_1$ be the affine complex coordinate corresponding to the vector $(1,0,-1) \in N$, $z_2$ to $(1,1,-2)/(2g+1)$ and $u$ to $(0,0,1)$.  Then the SYZ mirror of $\bX$ is
\begin{equation}\label{eq:mirz3}
 uv = 1 + z_1 + z_1^{-1}z_2^{2g+1} + \sum_{j=1}^g (\tau_j + \delta_{\nu_j} (\tau)) z_2^j,
 \end{equation}
where
$$\tau_j + \delta_{\nu_j}(\tau) = \sum_{k_1,\ldots,k_{g} \geq 0} \frac{\tau_1^{k_1}\ldots \tau_{g}^{k_{g}}}{(k_1 + \ldots + k_{g})!}
n_{1,l,\beta_{\nu_j}}([\pt]_L; (\mathbf{1}_{\nu_1})^{k_1} \times \ldots \times (\mathbf{1}_{\nu_{g}})^{k_{g}}),$$
$l = k_1 + \ldots + k_g$ and $\tau = \sum_{i=1}^{g} \tau_i \mathbf{1}_{\nu_i} \in H^2_{\mathrm{CR}}(\bX)$.

The orbifold $X = \cpx^3 / \Z_{2g+1}$ has a toric resolution $\tilde{X}$.  Figure \ref{fig:mirror_genus_g} shows the codimension-two skeleta of its moment map polytope, which is also the discriminant locus of the Gross fibration.  Its Mori cone is generated by $C_1,\ldots,C_g$ as shown in Figure \ref{fig:mirror_genus_g}.  The SYZ mirror of the resolution $\tilde{X}$ is
$$ uv = 1 + z_1 + q^{\sum_{i=1}^g (2i-1) C_i} z_1^{-1}z_2^{2g+1}  + \sum_{j=1}^g (1 + \delta_{j} (q)) q^{\sum_{i=0}^{j-2} (j-1-i)C_{g-i}} z_2^j, $$
where
$$ 1 + \delta_{j} (q) = \sum_{k_1,\ldots,k_{g} \geq 0} n_{1,0,\beta_j + \alpha_{k}} q^{\alpha_{k}},$$
$\alpha_{k} = \sum_{i=1}^{g} k_i C_i$, and $\beta_j$ is the basic disk class corresponding to the toric divisor $D_j$.

\begin{figure}[htp]
\centering
\includegraphics{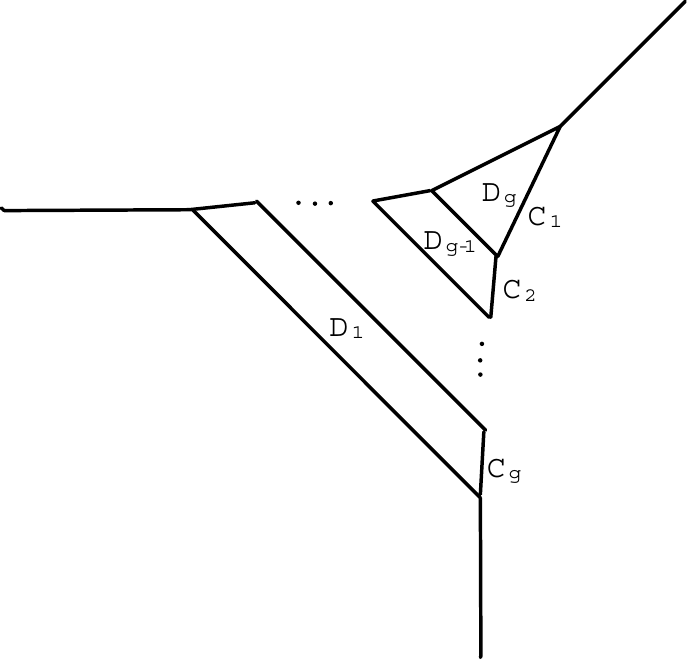}
\caption{A toric resolution of $\cpx^3 / \Z_{2g+1}$.  The diagram shows the $1$-strata of its moment map polytope.  $C_i$'s are labelling the holomorphic spheres which are mapped to the corresponding edges by the moment map.  $D_i$'s are labelling the toric divisors which are mapped to the corresponding facets.}\label{fig:mirror_genus_g}
\end{figure}

\noindent{(3)} $\bX = [\cpx^n / \Z_n]$ for $n \in \Z$.  Its fan has been described in Section \ref{sect:eg fib}.  It has a twisted sector of age one which corresponds to $\nu = (0,1) \in \Z^{n-1} \times \Z$.  Its SYZ mirror is
$$ uv = (\tau + \delta_\nu(\tau)) + z_1 + \ldots + z_{n-1} +  z_1^{-1}\ldots z_{n-1}^{-1}, \quad \tau + \delta_\nu(\tau) = \sum_{k \geq 1} \frac{\tau^{k}}{k!} n_{1,k,\beta_\nu} ([\pt]_L; (\mathbf{1}_\nu)^{k}).$$
When $n=3$, this mirror is the same as the one given in \eqref{eq:mirz3} as we can make change of variables $v^{new} = \frac{v}{z_2}, z_1^{new}= \frac{z_1}{z_2}, z_2^{new}= \frac{1}{z_2}$ in \eqref{eq:mirz3} to obtain the above  equation.

The total space of the canonical line bundle $K_{\proj^{n-1}}$ of  the projective space $\proj^{n-1}$ gives a crepant resolution, whose SYZ mirror is
$$ uv = (1 + \delta) + z_1 + \ldots + z_{n-1} +  q z_1^{-1}\ldots z_{n-1}^{-1}, \quad 1 + \delta = \sum_{k \geq 0} q^k n_{1,k,\beta_0 + kl},$$
where $l$ is the line class in $K_{\proj^{n-1}}$ and its corresponding K\"ahler parameter is $q$.
When $n=3$, this serves as one of the first nontrivial examples for the SYZ construction for toric CY 3-folds in \cite{CLL}.

Note that in the above examples, the mirror of $\bX$ and its crepant resolution almost have the same expressions, except that they have different coefficients.  This motivates the Open Crepant Resolution Theorem \ref{thm:openCR_toric_CY} which gives a precise relation between their mirrors.

\section{Computation of orbi-disk invariants}\label{sec:computation_open_GW}
In this section we compute the orbi-disk invariants of a toric CY orbifold relative to a Lagrangian torus fiber of the moment map.

Let $\bX$ be a toric CY orbifold as in Setting \ref{setting:toricCY}. Let $L\subset \bX$ be a Lagrangian torus fiber of the moment map. Let $\beta\in \pi_2(\bX, L)$ be such that $\mu_{CW}(\beta)=2$. Let $\bx=(\bX_{\nu_1},\ldots,\bX_{\nu_l})$ be a collection of twisted sectors of $\bX$ such that $\nu_i\in \mathrm{Box}'$ satisfies $\mathrm{age}(\nu_i)=1$ for all $i$. Suppose that the moduli space $\mathcal{M}_{1,l}^{main}(L, \beta,\bx)$ is non-empty. We would like to compute the corresponding genus 0 open orbifold GW invariant $n_{1,l,\beta}^\bX([\mathrm{pt}]_L; \mathbf{1}_{\nu_1},\ldots,\mathbf{1}_{\nu_l})$ (Definition \ref{defn:orbi-disk_inv}).

The approach we take here is to construct a suitable toric partial compactification $\bar{\bX}$ of $\bX$ for each $\beta\in \pi_2(\bX, L)$ with $\mu_{CW}(\beta)=2$, prove that the above invariants are equal to certain genus 0 {\em closed} orbifold GW invariants of $\bar{\bX}$, and evaluate them by toric mirror theorems, generalizing the approach in \cite{CLT11}. The proof of such an {\em open/closed equality}, which is geometric in nature, is by comparing moduli spaces of stable (orbi-)disks to $\bX$ with moduli spaces of stable orbi-maps to $\bar{\bX}$, as Kuranishi spaces. The key geometric idea, namely, ``capping off'' the disk component to form a genus $0$ closed Riemann surface, was first employed in \cite{Chan10, LLW10} and later in \cite{LLW_surfaces} (for toric CY surfaces) and \cite{chan-lau, CLLT12} (for compact semi-Fano toric manifolds). It was also applied in \cite{CCLT12} to calculate orbi-disk invariants for certain compact toric orbifolds.

\subsection{Toric (partial) compactifications}
We begin with the construction of the toric partial compactification $\bar{\bX}$. According to our discussion in Section~\ref{sec:orbidisk_and_moduli}, the class $\beta\in \pi_2(\bX,L)$ must be of the form
$\beta=\beta'+\alpha,$
where $\beta'\in\pi_2(\bX,L)$ is a basic (orbi) disk class with CW Maslov index 2 and $\alpha\in H_2^\textrm{eff}(\bX)$ is an effective curve class such that $c_1(\bX) \cdot \alpha=0$.
We have $\partial \beta' = \bb_{i_0} \in N$ for some $i_0\in\{0,1,\ldots,m'-1\}$.


\begin{construction}\label{construction:compactification}
Let
$\bb_\infty:=-\bb_{i_0} \in N.$
Let $\bar{\Sigma}\subset N_\mathbb{R}$ be the smallest complete simplicial fan that contains $\Sigma$ and the ray $\mathbb{R}_{\geq 0}\bb_\infty \subset N_\R$. More concretely, the fan $\bar{\Sigma}$ consists of cones in $\Sigma$ together with additional cones, each is spanned by the ray $\mathbb{R}_{\geq 0}\bb_\infty$ together with a cone over a face of the polytope $\mathcal{P}$ (recall the definition of $\mathcal{P}$ in Setting \ref{setting:toricCY}). The data
$(\bar{\Sigma}, \{\bb_i\}_{i=0}^{m-1}\cup\{\bb_\infty\})$
gives a stacky fan. Consider the associated toric orbifold: $\bar{\bX}:=\bX_{\bar{\Sigma}}.$
 We choose the extra vectors to be the same as that for $\bX$, namely, $\{\bb_m,\ldots,\bb_{m'-1}\} \subset N$.
\end{construction}

\begin{remark}\label{rmk:beta_depend}
We emphasize that, although not reflected in the notation, the compactification $\bar{\bX}$ depends on the class $\beta\in \pi_2(\bX,L)$.
\end{remark}

Since $\Sigma$ satisfies the Assumption \ref{assumption}, so does the stacky fan $\bar{\Sigma}$. The fan $\bar{\Sigma}$ has more primitive generators than $\Sigma$. We also have $\bX\subset \bar{\bX}$ and the toric prime divisor $D_\infty:=\bar{\bX}\setminus \bX$ corresponding to $\bb_\infty$.
The inclusion $\bX\subset \bar{\bX}$ divides the toric prime divisors in $\bar{\bX}$ into two kinds: the set of generators $\{\bb_i\}_{i=0}^{m-1}$ is a disjoint union $\{\bb_i\}=I\coprod J$, where for $\bb_i\in I$ the corresponding toric prime divisor $D_i\subset \bar{\bX}$ is contained entirely in $\bX$ (these correspond to the compact toric prime divisors in $\bX$), and for $\bb_j\in J$ the corresponding toric prime divisor $D_j\subset \bar{\bX}$ has non-empty intersection with $D_\infty$ (these correspond to the non-compact toric prime divisors in $\bX$).

Let $\beta_\infty\in\pi_2(\bar{\bX},L)$ be the basic disk class corresponding to $\bb_\infty$. As $\partial (\beta'+\beta_\infty)=\bb_{i_0}+\bb_\infty=0\in N$, the class $\bar{\beta}':=\beta'+\beta_\infty$ belongs to $H_2(\bar{\bX};\Q)$ (see \cite[Section 9.1]{CP}), and we have $c_1(\bar{\bX}) \cdot \bar{\beta}' = 2$.  Moreover we have the decompositions
$$H_2(\bar{\bX};\Q) = H_2(\bX;\Q) \oplus \Q\bar{\beta}'\textrm{ and }H_2^\textrm{eff}(\bar{\bX}) = \Z_{\geq0}\bar{\beta}' \oplus H_2^\textrm{eff}(\bX).$$
Denote by $\bar{\mathbb{L}}$, $\bar{\mathbb{K}}$ and $\bar{\mathbb{K}}_\mathrm{eff}$ respectively the counterparts for $\bar{\bX}$ of the spaces $\mathbb{L}$, $\mathbb{K}$ and $\mathbb{K}_\mathrm{eff}$ for $\bX$. Then we have the corresponding decompositions
\begin{align*}
\bar{\mathbb{L}} = \mathbb{L} \oplus \Z d_\infty,\quad \bar{\mathbb{K}} = \mathbb{K} \oplus \Z d_\infty,\quad
\bar{\mathbb{K}}_\mathrm{eff} = \mathbb{K}_\textrm{eff} \oplus \Z_{\geq0} d_\infty,
\end{align*}
where $d_\infty = e_{i_0} + e_\infty \in \widetilde{N} \oplus \Z e_\infty = \bigoplus_{i=0}^{m'-1}\Z e_i \oplus \Z e_\infty$.

Since $\alpha$ can be represented by a holomorphic map to $\bar{\bX}$ whose image is contained entirely in $\bX$ and misses $D_\infty=\bar{\bX}\setminus \bX$, we have $D_\infty\cdot \alpha=0$, and hence $c_1(\bar{\bX})\cdot\alpha=0$.  Moreover, each $\nu_i\in \mathrm{Box}'(\Sigma)$ with $\mathrm{age}(\nu_i)=1$ determines uniquely an element $\bar{\nu}_i\in \mathrm{Box}'(\bar{\Sigma})$ with $\mathrm{age}(\bar{\nu}_i)=1$.

We make some important observations about $\bar{\bX}$.

\begin{prop}\label{prop:semi_Fano}
The toric orbifold $\bar{\bX}$ with the extra vectors $\{\bb_m,\ldots,\bb_{m'-1}\}$ is semi-Fano in the sense of Definition~\ref{defn:sF}.
\end{prop}
\begin{proof}
To show that $\bar{\bX}$ is semi-Fano, we need to prove that
$ c_1 (\bar{\bX}) = \sum_{i=0}^{m-1} D_i + D_\infty $
is nef (since $\mathrm{age}(\bb_j) = 1$ for $j=m,\ldots,m'-1$), i.e. every rational orbi-curve $C$ satisfies
$$(D_0 + \ldots + D_{m-1} + D_\infty) \cdot C \geq 0.$$
Let $C \cdot D_\infty = k \in \Z$.  We must have $k \geq 0$. Otherwise, $C$ has a component contained in $D_\infty$ whose intersection with $D_\infty$ is negative.  Now $D_\infty = \{\unu = \infty\}$ is linearly equivalent to the divisor $\tilde{D} = \{\unu = c\}$ for any $c \not = 0$.\footnote{Two divisors $D_1$ and $D_2$ are said to be linear equivalent if there exists a meromorphic function $\phi$ such that $D_1$ and $D_2$ are the zero and pole divisors of $\phi$ respectively.  In such a case given a rational curve $C$, the intersection number of $C$ with $D_1$ is the same as that with $D_2$.  In our situation we take the meromorphic function $\phi$ to be $\unu - c$ for a fixed complex number $c$.}  A rational curve in $D_\infty$ has transverse intersections with $\tilde{D}$, and hence the intersection number is non-negative.  Since intersection number is topological, this implies $D_\infty$ has non-negative intersection with any curve contained in $D_\infty$ itself.  Thus $k$ cannot be negative.

Now consider $C - k C_0$, where $C_0$ is a holomorphic sphere representing the class $\beta' + \beta_\infty$ which has Chern number $c_1(\bar{\bX})\cdot C_0 = 2$.  $C - k C_0$ has zero intersection with the divisor $D_\infty$.  Moreover it can be written as a linear combination of one-dimensional toric strata of $\bX$.  Since $\bX$ is CY, $(C - k C_0) \cdot (D_0 + \ldots + D_{m-1}) = 0$.  Then
$$(D_0 + \ldots + D_{m-1} + D_\infty) \cdot C =  (D_0 + \ldots + D_{m-1} + D_\infty) \cdot (C - k C_0) + 2k = 2k \geq 0.$$
This completes the proof.
\end{proof}

\begin{prop}\label{prop:semi_proj}
The toric variety $\bar{X}$ underlying $\bar{\bX}$ is semi-projective.
\end{prop}
\begin{proof}
By \cite[Proposition 7.2.9]{CLS_toricbook}, the toric variety $X$ is semi-projective, as its moment map image is a full-dimensional lattice polyhedron $P$. The toric variety $\bar{X}$ corresponds to intersecting $P$ with a half space normal to $\bb_\infty$. The result is still a full dimensional lattice polyhedron. Hence $\bar{X}$ is semi-projective again by \cite[Proposition 7.2.9]{CLS_toricbook}.
\end{proof}

\begin{prop}\label{prop:proj}
Suppose that $\bb_{i_0} \in N$ lies in the interior of the support $|\Sigma|$. Then the fan $\bar{\Sigma}$ is complete, and hence the toric variety $\bar{X}$ underlying $\bar{\bX}$ is projective.
\end{prop}
\begin{proof}
To prove that $\bar{\Sigma}$ is complete, it suffices to see that any vector $v \in N_\R$ can be written as a non-negative linear combination of generators of the fan $\bar{\Sigma}$.  Since $\bb_{i_0}$ lies in the interior of the support $|\Sigma|$, there exists $t \in \R_{>0}$ large enough such that $v + t\bb_{i_0} \in |\Sigma|$.  Thus $v + t\bb_{i_0} = \sum_{i=0}^{m-1} a_i \bb_i $ for $a_i \in \R_{\geq 0}$.  Then
$ v = \sum_{i=0}^{m-1} a_i \bb_i -  t\bb_{i_0}  = \sum_{i=0}^{m-1} a_i \bb_i +  t\bb_{\infty}$.
\end{proof}

\begin{remark}
If $\bb_{i_0} \in N$ lies on the boundary of $|\Sigma|$, then the fan $\bar{\Sigma}$ in Construction \ref{construction:compactification} is incomplete, so the toric orbifold $\bar{\bX}$ is not projective but only semi-projective.
\end{remark}

We will need the following lemma when we analyze the curve moduli.

\begin{lemma} \label{unique sphere}
Given a generic point in $\bar{\bX}$, there exists a unique non-constant holomorphic sphere of Chern number two passing through the point.
\end{lemma}

\begin{proof}
Choose local toric coordinates $(\unu, z_1, \ldots, z_{n-1})$ such that $z_1, \ldots, z_{n-1}$ are not identically zero when restricted on $D_{i_0}$. We take the point to be in the open toric orbit $(\C^\times)^n \subset \bar{\bX}$.  Suppose it has coordinates $(c_0, c_1, \ldots, c_{n-1})$, where $c_i \not= 0$ for all $i=0,\ldots,n-1$.  Then the holomorphic sphere defined by $z_i = c_i$ for all $i=1,\ldots,n-1$ passes through the point, and it only intersects $D_{i_0}$ and $D_\infty$ once but not any other divisors.  Thus it intersects with the anti-canonical divisor (which is the sum over all toric prime divisors) twice and hence has Chern number two.

To show uniqueness, suppose we have a non-constant holomorphic sphere of Chern number two passing through a point in the open toric orbit.  It must intersect $D_\infty$, since otherwise, it will be entirely contained in the toric CY $\bX$, and by the maximum principle applied to the holomorphic function $\unu$ on the sphere, the sphere must lie entirely in the toric divisors of $\bX$, and hence cannot pass through a point in the open toric orbit.  Since it has Maslov index two, it intersects $D_\infty$ at most two times (counted with multiplicity).  The meromorphic function $\unu$ on the sphere must have both zeroes and poles, and thus it must have one zero and one pole.  This means that the sphere intersects both $D_0$ and $D_\infty$ once, and that it cannot intersect other divisors since it only has Maslov index two.  Thus the functions $z_i$'s on the sphere have neither poles nor zeroes, and hence can only be constants.  We conclude that it is precisely the holomorphic sphere defined by $z_i = c_i$ for all $i=1,\ldots,n-1$.
\end{proof}



\begin{example} \label{KF2}
The fan of the Hirzebruch surface $\mathbb{F}_2$ has primitive generators $(-1,1)$, $(0,1)$, $(1,1)$, $(0,-1)$.  The total space of the canonical line bundle $\bX = K_{\mathbb{F}_2}$ is again a toric manifold, whose fan has primitive generators $\bb_0 = (0,0,1)$, $\bb_1 = (-1,1,1)$, $\bb_2 = (0,1,1)$, $\bb_3=(1,1,1)$ and $\bb_4=(0,-1,1)$.  The polytope $\mathcal{P}$ is the convex hull of $(-1,1), (1,1), (0,-1)$ in the plane $\R^2$.  The generator $(0,1)$ lies in the boundary of $\mathcal{P}$ but is {\em not} a vertex of $\mathcal{P}$, so the toric compactification $\bar{\bX}$ corresponding to $\bb_2$ (see Construction \ref{construction:compactification}) is noncompact.

The toric prime divisor $D_2$ in $\bX$ corresponding to $\bb_2$ is noncompact and biholomorphic to $\proj^1 \times \C$.  The inclusion $(z,c): \proj^1 \hookrightarrow  \proj^1 \times \C \cong D_2$ for any constant $c \in \C$ gives a $(0,-2)$ rational curve in $\bX = K_{\mathbb{F}_2}$, whose class is denoted by $l \in H_2(\bX; \Z)$.  It has Chern number zero and {\em does} contribute to sphere bubbling so that the open GW invariants $n^{\bX}_{\beta_2+kl}$ for $k \in \Z_{\geq 0}$ are non-trivial.  We will see in Section \ref{sec:examples_mirror} that in fact $n^{\bX}_{\beta_2+kl} = 1$ when $k=0,1$ and zero otherwise.  Hence $n^{\bX}_{\beta_2+kl} = n^{\mathbb{F}_2}_{\beta_2+kl}$ where $\beta_2$ and $l$ on the right hand side of the equality denote the basic disk class corresponding to $D_2 \subset \mathbb{F}_2$ and the class of the $(-2)$-curve in $\mathbb{F}_2$ respectively.
\end{example}


\subsection{An open/closed equality}\label{sec:open_closed_equality}
Let $\iota: \{p\}\to L$ be the inclusion of a point.
\begin{defn}\label{def:moduli_spaces}
Let $\bX$ and $\bar{\bX}$ be as in Construction \ref{construction:compactification}. Consider three moduli spaces:
\begin{enumerate}
\item
Let $\mathcal{M}^{op}_{1,l}(\bX, \beta, \bx):=\mathcal{M}^{main}_{1,l}(L, \beta, \bx)$ be the moduli space of stable maps from genus $0$ bordered orbifold Riemann surfaces with one boundary component to $(\bX, L)$ of class $\beta=\beta'+\alpha$ such that there is one boundary marked point and $l$ interior marked points of type $\bx=(\bX_{\nu_1},\ldots,\bX_{\nu_l})$. Let $ev_0: \mathcal{M}^{op}_{1,l}(\bX, \beta, \bx)\to L$ denote the evaluation map at the boundary marked point. Consider the fiber product
$\mathcal{M}^{op}_{1,l}(\bX, \beta, \bx, p):=\mathcal{M}^{op}_{1,l}(\bX, \beta, \bx) \times_{ev_0, \iota} \{p\}.$

\item
Let $\mathcal{M}^{op}_{1,l}(\bar{\bX}, \beta, \bx'):=\mathcal{M}^{main}_{1,l}(L, \beta, \bx')$ be the moduli space of stable maps from genus $0$ bordered orbifold Riemann surfaces with one boundary component to $(\bar{\bX}, L)$ of class $\beta$ such that there is one boundary marked point and $l$ interior marked points of type $\bx'=(\bar{\bX}_{\nu_1},\ldots,\bar{\bX}_{\nu_l})$. Let $ev_0: \mathcal{M}^{op}_{1,l}(\bar{\bX}, \beta, \bx')\to L$ denote the evaluation map at the boundary marked point. Consider the fiber product
$\mathcal{M}^{op}_{1,l}(\bar{\bX}, \beta, \bx', p):= \mathcal{M}^{op}_{1,l}(\bar{\bX}, \beta, \bx') \times_{ev_0, \iota} \{p\}.$

\item
Let $\mathcal{M}^{cl}_{1+l}(\bar{\bX}, \bar{\beta}, \bar{\bx})$ be the moduli space of stable maps from genus $0$ orbifold Riemann surfaces to $\bar{\bX}$ of class $\bar{\beta}:=\bar{\beta}'+\alpha$ such that the $1+l$ interior marked points of are type $\bar{\bx}=(\bar{\bX}, \bar{\bX}_{\nu_1},\ldots,\bar{\bX}_{\nu_l})$. Let $ev_0: \mathcal{M}^{cl}_{1+l}(\bar{\bX}, \bar{\beta}, \bar{\bx}) \to \bar{\bX}$ denote the evaluation map at the first marked point. Consider the fiber product
$\mathcal{M}^{cl}_{1+l}(\bar{\bX}, \bar{\beta}, \bar{\bx}, p):= \mathcal{M}^{cl}_{1+l}(\bar{\bX}, \bar{\beta}, \bar{\bx}) \times_{ev_0,\iota} \{p\}.$
\end{enumerate}
\end{defn}

\begin{prop}[Compactness]\label{prop:cpt_disk_moduli}
\hfill
\begin{enumerate}
\item[(a)]
Let $D$ be a toric prime divisor of the toric CY orbifold $\bX$, $\alpha \in H_2(D; \Z)$ and $p \in D$.  Then the moduli space of rational curves in $D$ representing $\alpha$ with one marked point passing through $p$ is compact.
\item[(b)]
Let $\alpha \in H_2(\bX; \Z)$ and $p \in \bX$.  Then the moduli space of rational curves in $\bX$ representing $\alpha$ with one marked point passing through $p$ is compact.
\item[(c)]
The moduli $\mathcal{M}^{op}_{1,l}(\bX, \beta_i + \alpha, \bx)$ for $i = 0,\ldots,m'-1$ and $\alpha \in  H_2(\bX; \Z)$ is compact.
\end{enumerate}
\end{prop}

\begin{proof}
\hfill
\begin{enumerate}
\item[(a)]

The statement certainly holds when the divisor $D$ is compact.  Now suppose that $D$ is a non-compact divisor.  We are going to prove that all rational curves representing $\alpha$ with one marked point passing through $p$ must lie in a compact subvariety of $D$, and hence the moduli space is compact.

The toric divisor $D \subset \bX$ itself is a toric orbifold, whose fan $\Sigma_{D}$ is given by the quotient of $\Sigma$ in the $v$-direction and localization at zero, where $v$ is the primitive generator of $\Sigma$ corresponding to $D$.  Since $D$ is non-compact, $v$ lies in the boundary of the polytope $\mathcal{P}$.  Thus there exists a half space defined by $\{\nu \geq 0\} \subset \left(N/\langle v \rangle\right)_{\R}$ for some $\nu \in M^{\perp v}$ containing $|\Sigma_{D}|$.  Then the function on $D$ corresponding to $\nu$ is holomorphic, and by abuse of notation we also denote it by $\nu$.  By the maximum principle, $\nu$ is constant on each sphere component of a rational curve in $D$.  Since the rational curve is connected, $\nu$ takes the same constant on the whole rational curve.  Let $\nu(p) = c \in \C$.  Then any rational curve with one marked point passing through $p$ lies in the level set $\{\nu = c\} \subset D$.

The above is true for all $\nu \in M^{\perp v}$ such that the corresponding half space $\{\nu \geq 0\}$ contains $|\Sigma_D|$.  Let $\nu_1, \ldots, \nu_k$ be the extremal ones, meaning that each of the corresponding half spaces contains $|\Sigma_D|$ and a codimension-one face of $|\Sigma_D|$.  Then there exist $c_1,\ldots,c_k \in \C$ such that any rational curve with one marked point passing through $p$ lies in $\{\nu_i = c_i \textrm{ for all } i = 1,\ldots,k\}$, which is a compact subvariety of $D$.  Hence the moduli space of rational curves representing $\alpha$ with one marked point and passing through $p$ is compact.

\item[(b)]

We may assume that $p$ lies in a toric divisor of $\bX$, or otherwise the moduli space is empty since $\bX$ is a toric CY orbifold.  All rational curves in $\bX$ lie in toric divisors of $\bX$.  Thus the moduli space can be written as a fiber product of moduli spaces of rational curves in prime divisors of $\bX$.  By part (a) the moduli space of rational curves in a toric prime divisor passing through a fixed target point is compact.  Hence the fiber product is also compact.

\item[(c)]

The disk moduli $\mathcal{M}^{op}_{1,l}(\bX,\beta_i + \alpha, \bx)$ is equal to the fiber product $\mathcal{M}^{op}_{1,1}(\bX, \beta_i) \times_{ev} \mathcal{M}^{cl}_{\bullet + l}(\bX,\alpha, \bx)$, where $\mathcal{M}^{op}_{1,1}(\bX, \beta_i)$ is the moduli space of stable disks in $\bX$ representing the basic disks class $\beta_i$ with one interior marked point and one boundary marked point, $\mathcal{M}^{cl}_{\bullet+l}(\bX,\alpha, \bx)$ is the moduli space of rational curves in $\bX$ representing $\alpha$ with one marked point $\bullet$ and $l$ other marked points of type $\bx$, and the fiber product is over evaluation maps at the interior marked point of the disk and the marked point $\bullet$ of the rational curve.  Now, the moduli space $\mathcal{M}^{op}_{1,1}(\bX, \beta_i)$ is known to be compact by the classification result of Cho-Poddar \cite{CP}.  By part (b), $\mathcal{M}^{cl}_{\bullet + l}(\bX,\alpha, \bx) \times_{ev} \{\pt\}$ is compact.  Thus the fiber product $\mathcal{M}^{op}_{1,1}(\bX, \beta_i) \times_{ev} \mathcal{M}^{cl}_{\bullet + l}(\bX,\alpha, \bx)$ is also compact.
\end{enumerate}

\end{proof}

\begin{corollary}\label{cor:cpt_disk_moduli}
The moduli space $\mathcal{M}^{op}_{1,l}(\bX, \beta, \bx, p)$ in Definition \ref{def:moduli_spaces} is compact. Hence, the open orbifold GW invariant $n_{1,l,\beta}^\bX([\mathrm{pt}]_L; \mathbf{1}_{\nu_1},\ldots,\mathbf{1}_{\nu_l})$ in Definition \ref{defn:orbi-disk_inv} is well-defined.
\end{corollary}

The main result of this subsection is the following
\begin{thm}\label{thm:open_closed_equality}
\hfill
\begin{enumerate}
\item[(a)]
The moduli spaces $\mathcal{M}^{op}_{1,l}(\bX, \beta, \bx,p)$ and $\mathcal{M}^{op}_{1,l}(\bar{\bX}, \beta, \bx',p)$ are isomorphic as Kuranishi spaces.  Hence we have the following equality between genus 0 open orbifold GW invariants:
$$n_{1,l,\beta}^\bX([\mathrm{pt}]_L; \mathbf{1}_{\nu_1},\ldots,\mathbf{1}_{\nu_l}) =
n_{1,l,\beta}^{\bar{\bX}}([\mathrm{pt}]_L; \mathbf{1}_{\bar{\nu}_1},\ldots,\mathbf{1}_{\bar{\nu}_l}).$$

\item[(b)]
The moduli spaces $\mathcal{M}^{op}_{1,l}(\bar{\bX}, \beta, \bx',p)$ and $\mathcal{M}^{cl}_{1+l}(\bar{\bX}, \bar{\beta}, \bar{\bx},p)$ are isomorphic as Kuranishi spaces.  Hence we have the following equality between genus 0 open and closed orbifold GW invariants, called the {\em open/closed equality}:
\begin{equation}\label{eqn:open_closed_equality}
n_{1,l,\beta}^\bX([\mathrm{pt}]_L; \mathbf{1}_{\nu_1},\ldots,\mathbf{1}_{\nu_l}) =
\langle[\mathrm{pt}], \mathbf{1}_{\bar{\nu}_1},\ldots,\mathbf{1}_{\bar{\nu}_l} \rangle_{0,1+l, \bar{\beta}}^{\bar{\bX}}.
\end{equation}
\end{enumerate}
\end{thm}

\begin{proof}
We begin with part (a). The inclusion $\bX \subset \bar{\bX}$ gives a natural map
$$\mathcal{M}^{op}_{1,l}(\bX, \beta, \bx,p)\to \mathcal{M}^{op}_{1,l}(\bar{\bX}, \beta, \bx',p),$$
which is clearly injective. To show that this map is surjective, we need to prove that a stable disk in $\mathcal{M}^{op}_{1,l}(\bar{\bX}, \beta, \bx',p)$ is indeed contained in $\bX$.  This means there are no stable disk maps $f:(\mathcal{C}, \partial\mathcal{C})\to (\bar{\bX}, L)$ of class $\beta=\beta'+\alpha$ such that
$\mathcal{C}=\mathcal{D}\cup \mathcal{C}_0\cup \mathcal{C}_\infty$
is a union where $\mathcal{D}$ is the disk component; $\mathcal{C}_0$ is a closed (orbifold) Riemann surface whose components are contained in $\bigcup_{\bb_i\in I} D_i$; and $\mathcal{C}_\infty$ is a non-empty closed (orbifold) Riemann surface whose components are contained in $D_\infty\cup \bigcup_{\bb_j\in J}D_j$ and have non-negative intersections with divisors $D_i$, $\bb_i\in I$ (via $f$).

Suppose there is such a stable disk map. Let $A:=f_*[\mathcal{C}_0]$ and $B:=f_*[\mathcal{C}_\infty]$. Then $\alpha=A+B$. Since $c_1(\bar{\bX})\cdot\alpha=0$ and $-K_{\bar{\bX}}$ is nef, we have $c_1(\bar{\bX})\cdot A = 0 = c_1(\bar{\bX}) \cdot B$.
Writing $B=\sum_k b_k B_k$ as an effective linear combination of the classes $B_k$ of irreducible $1$-dimensional torus-invariant orbits in $\bar{\bX}$, we have $c_1(\bar{\bX})\cdot (b_kB_k)=0$ for all $k$ (again using the fact that $-K_{\bar{\bX}}$ is nef). Each $B_k$ corresponds to an $(n-1)$-dimensional cone $\sigma_k\in \bar{\Sigma}$, and by construction, either $\sigma_k$ contains $\bb_\infty$, or $\sigma_k$ and $\bb_{\infty}$ together span an $n$-dimensional cone in $\bar{\Sigma}$.

Since $f(\mathcal{C}_\infty)\subset D_\infty\cup \bigcup_{\bb_j\in J}D_j$, we see that if $\bb_i\in I$ then $\bb_i \notin \sigma_k$. Also, since $D\cdot (b_kB_k)\geq 0$ for every toric prime divisor of $\bar{\bX}$ not corresponding to a ray in $\sigma_k$, we have by\footnote{Their argument extends to the simplicial cases needed here.} \cite[Lemma 4.5]{G-I11} that $D\cdot (b_kB_k)=0$ for every toric prime divisor $D$ corresponding to an element in $(\{\bb_i\}\cup \{\bb_\infty\})\setminus F(\sigma_k)$; here $F(\sigma_k)$ is the minimal face in the fan polytope of $\bar{\Sigma}$ that contains rays in $\sigma_k$. As the divisors $D$ corresponding to $(\{\bb_i\}\cup \{\bb_\infty\})\setminus F(\sigma_k)$ span $H^2(\bar{\bX})$, we must have $b_kB_k=0$. We conclude that $B=0$.

Therefore we have a bijection between moduli spaces
$\mathcal{M}^{op}_{1,l}(\bX, \beta, \bx,p) \cong \mathcal{M}^{op}_{1,l}(\bar{\bX}, \beta, \bx',p).$
Since every stable disk in $\mathcal{M}^{op}_{1,l}(\bar{\bX}, \beta, \bx',p)$ is supported in (a compact region of) $\bX$, it is clear that it has the same deformations and obstructions as the corresponding stable disk in $\mathcal{M}^{op}_{1,l}(\bX, \beta, \bx,p)$.  By the same arguments as in Part(C) of the proof of \cite[Propostion 5.6]{CLLT12} (which can be adapted to the orbifold setting here in a straightforward way), it follows that the above bijection gives an isomorphism of Kuranishi structures. This proves (a).

The proof of part (b) is basically the same as that of \cite[Theorem 35]{CCLT12}.  First of all, for a stable disk map in $\mathcal{M}^{op}_{1,l}(\bar{\bX}, \beta, \bx',p)$, it consists of a unique disk component $u_0$ and a rational curve component $C'$.  We denote such a stable disk by $u_0 + C'$.  The disk component represents a basic (orbi-)disk class and hence is regular by \cite[Propositions 8.3 and 8.6]{CP}.  Thus the obstruction merely comes from the rational curve component.

On the other hand, by Lemma \ref{unique sphere}, there is a unique holomorphic sphere $C_0$ with Chern number two in $\bar{\bX}$ passing through a generic point $p \in \bar{\bX}$.  So for a stable curve in $\mathcal{M}^{cl}_{1+l}(\bar{\bX}, \bar{\beta}, \bar{\bx},p)$, since it passes through $p$ and it has Chern number two, it has $C_0$ as one of its components, and the rest is a rational curve $C'$ with Chern number zero contained in the toric divisors.  We denote such a rational curve by $C_0 + C'$.  Since $C_0$ is a holomorphic sphere whose normal bundle is trivial, it is unobstructed.  Thus the obstruction of $C_0 + C'$ merely comes from $C'$.  A bijective map between $\mathcal{M}^{op}_{1,l}(\bar{\bX}, \beta, \bx',p)$ and $\mathcal{M}^{cl}_{1+l}(\bar{\bX}, \bar{\beta}, \bar{\bx},p)$ is given by sending $u_0 + C'$ to $C_0 + C'$ and vice versa.  They have the same deformations and obstructions (which are contributed from the same rational curve component $C'$), and hence, as Kuranishi structures, we have
$$\mathcal{M}^{op}_{1,l}(\bar{\bX}, \beta, \bx',p) \cong \mathcal{M}^{cl}_{1+l}(\bar{\bX}, \bar{\beta}, \bar{\bx},p).$$

The identification of the two Kuranishi structures can be done as explained in Step 3 of the proof of \cite[Theorem 35]{CCLT12}, except that the choices of obstruction bundles have to be suitably modified in order to obtain smoothly compatible Kuranishi charts which can be glued together to obtain a global structure (see \cite{McDuff-Wehrheim12, FOOO12}).

Recall that in the general scheme developed by Fukaya, Oh, Ohta and Ono in constructing Kuranishi structures of a moduli space, one first constructs a Kuranishi neighborhood for each point of the moduli space. To obtain a global Kuranishi structure which is smoothly compatible, one then chooses a sufficiently dense finite set of points in the moduli space, and redefines the Kuranishi neighborhood by considering a new obstruction bundle obtained as the direct sum of parallel transports of the obstruction bundles over the finite set of points. When the domain of the stable map is not stable, however, one has to further consider a stabilization of the domain and extra care is needed in choosing the obstruction bundles. See \cite[Section 3.2]{FOOO12} for a brief description and \cite[Sections 15-18]{FOOO12} for the detailed construction.

The construction of Kuranishi neighborhoods given in the proof \cite[Theorem 35]{CCLT12} corresponds to the case where the domain of a stable map is also stable, in which the above description of the obstruction bundles already suffices. But for the moduli spaces we consider here, the domain of a stable map may not be stable, so we need the general construction as described in \cite[Part 4]{FOOO12}. Nevertheless, we emphasize that all these (or any such) constructions can be carried out in the same way for the open and closed moduli spaces because the obstruction bundles on the disk component $u_0$ and the sphere component $C_0$ both vanish, and therefore the Kuranishi structures are naturally identified with each other.
\end{proof}

\begin{remark}\label{rmk:curve_moduli_cpt}
The proof of Theorem \ref{thm:open_closed_equality} identifies the moduli space $\mathcal{M}^{cl}_{1+l}(\bar{\bX}, \bar{\beta}, \bar{\bx},p)$ with the moduli space $\mathcal{M}^{op}_{1,l}(\bX, \beta, \bx,p)$, which is compact by Corollary \ref{cor:cpt_disk_moduli}. So $\mathcal{M}^{cl}_{1+l}(\bar{\bX}, \bar{\beta}, \bar{\bx},p)$ is also compact, and hence the closed orbifold GW invariant $\langle[\mathrm{pt}], \mathbf{1}_{\bar{\nu}_1},\ldots,\mathbf{1}_{\bar{\nu}_l} \rangle_{0,1+l, \bar{\beta}}^{\bar{\bX}}$ is well-defined even when $\bar{\bX}$ is noncompact.
\end{remark}


\subsection{Calculation via mirror theorem}\label{sec:computation_w_J}
By the open/closed equality \eqref{eqn:open_closed_equality}, the open orbifold GW invariants of $\bX$ we need may be computed by evaluating the genus 0 closed orbifold GW invariants $\langle[\mathrm{pt}], \mathbf{1}_{\bar{\nu}_1},\ldots,\mathbf{1}_{\bar{\nu}_l} \rangle_{0,1+l, \bar{\beta}}^{\bar{\bX}}$ of $\bar{\bX}$.
These closed orbifold GW invariants are certain coefficients of the $J$-function of $\bar{\bX}$. We evaluate them by extending the approach developed in \cite{CLT11} to the orbifold setting.

The idea is to use closed mirror theorems for toric orbifolds to explicitly compute these coefficients via the combinatorially defined $I$-function of $\bar{\bX}$. However, since $\bar{\bX}$ may not be compact, we cannot directly apply the closed mirror theorem (Theorem \ref{closed_mirror_thm}) to $\bar{\bX}$ as in \cite{CLT11}. We get around this by first applying the {\em equivariant} mirror theorem (Theorem \ref{equiv_closed_mirror_thm}) to evaluate the genus 0 {\em equivariant} closed orbifold GW invariants of $\bar{\bX}$:
$\langle[\mathrm{pt}]_\mathbb{T}, \mathbf{1}_{\bar{\nu}_1},\ldots,\mathbf{1}_{\bar{\nu}_l} \rangle_{0,1+l, \bar{\beta}}^{\bar{\bX}, \mathbb{T}},$
where $[\text{pt}]_\mathbb{T}\in H^*_\mathbb{T}(\bar{\bX})$ is the equivariant lift of $[\text{pt}]\in H^*(\bar{\bX})$ represented by a $\mathbb{T}$-fixed point, and then evaluating $\langle[\mathrm{pt}], \mathbf{1}_{\bar{\nu}_1},\ldots,\mathbf{1}_{\bar{\nu}_l} \rangle_{0,1+l, \bar{\beta}}^{\bar{\bX}}$ by taking non-equivariant limits.


\subsubsection{Identifying the invariants}
We now begin the computation of the relevant equivariant orbifold GW invariants.
The $\mathbb{T}$-equivariant $J$-function of $\bar{\bX}$ (cf. Definition \ref{defn:equiv_J_function}) expands as a series in $1/z$ as follows:
\begin{align*}
J_{\bar{\bX}, \mathbb{T}}(q,z)
= & \conste^{\tau_{0,2}/z} (1+\sum_\alpha\sum_{\substack{(d,l)\neq(0,0)\\ d\in H_2^\mathrm{eff}(\bar{\bX})}}
\frac{q^d}{l!}\frac{1}{z}\sum_{k\geq 0}\left\langle 1,\tau_\mathrm{tw},\ldots,\tau_\mathrm{tw},\phi_\alpha\psi^k\right\rangle^{\bar{\mathcal{X}}, \mathbb{T}}_{0,l+2,d} \frac{\phi^\alpha}{z^k} )\\
= & \left(1+\frac{\tau_{0,2}}{z}+O\left(\frac{1}{z^2}\right)\right)
(1+\sum_\alpha\sum_{\substack{(d,l)\neq(0,0)\\ d\in H_2^\mathrm{eff}(\bar{\bX})}}
\frac{q^d}{l!}\frac{1}{z}\sum_{k\geq 0}\left\langle \tau_\mathrm{tw},\ldots,\tau_\mathrm{tw},\phi_\alpha\psi^{k-1}\right\rangle^{\bar{\mathcal{X}}, \mathbb{T}}_{0,l+1,d} \frac{\phi^\alpha}{z^k}),
\end{align*}
where we use the string equation in the second equality. Note that $\tau_{0,2}\in H^2_\mathbb{T}(\bar{\bX})$, and $\phi_\alpha=[\mathrm{pt}]_\mathbb{T}$ if and only if $\phi^\alpha={\bf 1}\in H^0(\bar{\bX})$. If we consider
\begin{equation}\label{eqn:choice_of_tau_tw}
\tau_\mathrm{tw}=\sum_{\nu\in \mathrm{Box}'(\Sigma)^{\mathrm{age}=1}}\tau_\nu \mathbf{1}_{\bar{\nu}},
\end{equation}
then the closed equivariant orbifold GW invariants
$\langle[\mathrm{pt}]_\mathbb{T}, \mathbf{1}_{\bar{\nu}_1},\ldots,\mathbf{1}_{\bar{\nu}_l} \rangle_{0,1+l, \bar{\beta}}^{\bar{\bX}, \mathbb{T}}$
occur as the coefficients of $q^{\bar{\beta}}\tau_{\nu_1}\cdots\tau_{\nu_l}$ in the $1/z^2$-term of $J_{\bar{\bX}, \mathbb{T}}(q,z)$ that takes values in $H^0(\bar{\bX})$.

Since $\bar{\bX}$ is semi-Fano (by Proposition~\ref{prop:semi_Fano}) and semi-projective (by Proposition~\ref{prop:semi_proj}), we can apply the equivariant toric mirror theorem (Theorem \ref{equiv_closed_mirror_thm}) which says that
$$e^{q_0(y)/z}J_{\bar{\bX}, \mathbb{T}}(q,z)=I_{\bar{\bX}, \mathbb{T}}(y(q,\tau),z)$$
via the inverse $y = y(q,\tau)$ of the toric mirror map. Recall that the equivariant $I$-function here is the one defined using the extended stacky fan
$$(\bar{\Sigma}, \{\bb_i \mid 0\leq i\leq m-1\} \cup \{\bb_\infty\} \cup \{\bb_j \mid m\leq j\leq m'-1\}),$$
where
$$\{\bb_j \mid m\leq j \leq m'-1\}=\{\nu\in\mathrm{Box}'(\Sigma) \mid \mathrm{age}(\nu)=1\}.$$
Therefore our next task is to explicitly identify the part of the $1/z^2$-term of the equivariant $I$-function of $\bar{\bX}$ that takes values in $H^0(\bar{\bX})$. According to the definition of the equivariant $I$-function in Definition \ref{defn:equiv_I_function}, the part taking values in $H^0(\bar{\bX})$ arises from terms with $d\in \bar{\mathbb{K}}_\mathrm{eff}$ such that
\begin{equation}\label{eqn:condition_d_1}
\nu(d)=0, \textrm{ i.e. } \mathbf{1}_{\nu(d)}=\mathbf{1}\in H^0(\bar{\bX}).
\end{equation}
And for $d\in \bar{\mathbb{K}}_\mathrm{eff}$ to satisfy \eqref{eqn:condition_d_1}, we must have
$$\langle D_i, d \rangle \in \Z,  \textrm{ for } i\in \{0,\ldots,m'-1\} \cup \{\infty\}.$$
This follows from the definition of $\nu(d)$.

Let $d\in \bar{\mathbb{K}}_\mathrm{eff}$ be such that $\nu(d)=0$. We examine the $(1/z)$-series expansion of the corresponding term in the equivariant $I$-function of $\bar{\bX}$:
\begin{equation}\label{eqn:term_in_I}
y^d\prod_{i\in\{0,\ldots,m'-1\}\cup\{\infty\}}
\frac{\prod_{k=\lceil\langle D_i,d\rangle\rceil}^\infty(\bar{D}^\mathbb{T}_i+(\langle D_i,d\rangle-k)z)}
{\prod_{k=0}^\infty(\bar{D}^\mathbb{T}_i+(\langle D_i,d\rangle-k)z)}.
\end{equation}
Recall that $\bar{D}^\mathbb{T}_0,\ldots,\bar{D}^\mathbb{T}_{m-1},\bar{D}^\mathbb{T}_\infty \in H^2(\bar{\bX})$ are $\mathbb{T}$-divisor classes corresponding to $\bb_0,\ldots,\bb_{m-1},\bb_\infty$, and $\bar{D}^\mathbb{T}_j =0$ in $H^2_\mathbb{T}(\bar{\bX})$ for $m\leq j\leq m'-1$. We may factor out copies of $z$ to rewrite \eqref{eqn:term_in_I} as
\begin{equation}\label{eqn:term_in_I_2}
\frac{y^d}{z^{\langle \hat{\rho}(\bar{\bX}), d \rangle}}
\prod_{i\in\{0,\ldots,m'-1\}\cup\{\infty\}}
\frac{\prod_{k=\lceil\langle D_i,d\rangle\rceil}^\infty(\bar{D}^\mathbb{T}_i/z+(\langle D_i,d\rangle-k))}
{\prod_{k=0}^\infty(\bar{D}^\mathbb{T}_i/z+(\langle D_i,d\rangle-k))}.
\end{equation}
where $\hat{\rho}(\bar{\bX})=\sum_{i=0}^{m-1} D_i + D_\infty + \sum_{j=m}^{m'-1} D_j$. So we need
\begin{equation}\label{eqn:condition_d_2}
\langle \hat{\rho}(\bar{\bX}), d\rangle= \sum_{i=0}^{m-1} \langle D_i, d\rangle +\langle D_\infty, d\rangle + \sum_{j=m}^{m'-1} \langle D_j, d\rangle \leq 2.
\end{equation}

Since we need the part taking values in $H^0(\bar{\bX})$, we need the terms in \eqref{eqn:term_in_I_2} in which the divisor classes $\bar{D}^\mathbb{T}_0, \ldots, \bar{D}^\mathbb{T}_{m-1}, \bar{D}^\mathbb{T}_\infty$ do not occur. For $0\leq i\leq m-1$ or $i=\infty$, the fraction
$$\frac{\prod_{k=\lceil\langle D_i,d\rangle\rceil}^\infty(\bar{D}^\mathbb{T}_i/z+(\langle D_i,d\rangle-k))}{\prod_{k=0}^\infty(\bar{D}^\mathbb{T}_i/z+(\langle D_i,d\rangle-k))}$$
is proportional to $\bar{D}^\mathbb{T}_j$ if $\langle D_j,d\rangle=\lceil\langle D_j,d\rangle\rceil<0$. Thus we need
\begin{equation}\label{eqn:condition_d_3}
\langle D_i,d\rangle\geq 0, \quad i \in \{0,\ldots,m-1\}\cup\{\infty\}.
\end{equation}
Also observe that since $d\in \bar{\mathbb{K}}_{\mathrm{eff}}$, $\langle D_j, d\rangle \geq 0$ for $m\leq j\leq m'-1$. So there are only two possible cases: either
\begin{itemize}
\item there is exactly one $j$ such that $\langle D_j, d\rangle = 2$ in \eqref{eqn:condition_d_2} and $\langle D_i, d\rangle=0$ for $i\neq j$; or\\

\item there are $j_1, j_2$ such that $\langle D_{j_1}, d\rangle = \langle D_{j_2}, d\rangle = 1$ in \eqref{eqn:condition_d_2} and $\langle D_i, d\rangle=0$ for $i\neq j_1, j_2$.
\end{itemize}

By the fan sequence \eqref{eqn:fan_seq}, an element $d\in \bar{\mathbb{K}}_\mathrm{eff}$ corresponds to an element
$$\sum_{0\leq i\leq m-1}  \langle D_i, d\rangle e_i +\langle D_\infty, d\rangle e_\infty+\sum_{m\leq j\leq m'-1} \langle D_j, d\rangle e_j \in
\bigoplus_{0\leq j\leq m-1}  \mathbb{Z} e_j \oplus \mathbb{Z} e_\infty \oplus \bigoplus_{m\leq j\leq m'-1} \mathbb{Z} e_j$$
such that
$$\sum_{0\leq i\leq m-1}  \langle D_i, d\rangle \bb_i + \langle D_\infty, d\rangle \bb_\infty
+ \sum_{m\leq j\leq m'-1} \langle D_j, d\rangle \bb_j = 0.$$
In order for this equality to hold, we cannot have $\langle D_i, d\rangle=0$ for all but one $i$. So we must be in the other case, namely, there are exactly two indices $j_1, j_2$ such that $\langle D_{j_1}, d\rangle = \langle D_{j_2}, d\rangle = 1$, and $\langle D_i, d\rangle = 0$ for $i\neq j_1, j_2$. Since the vectors $\bb_0,\ldots,\bb_{m-1}, \bb_m,\ldots,\bb_{m'-1}$ belong to the half-space in $N_\R\oplus\mathbb{R}$ opposite to the half-space containing $\bb_\infty$, we must have $\infty\in \{j_1, j_2\}$. As noted in Remark \ref{rmk:beta_depend}, the fan $\bar{\Sigma}$ depends on the disk class $\beta\in \pi_2(\bX, L)$ in question. There are two possibilities:
\begin{itemize}
\item
{\bf Case 1: $\beta$ is a smooth disk class.} This means that $\beta = \beta'+\alpha$ with $\alpha \in H_2(\bX)$ and $\beta'\in \pi_2(\bX, L)$ is the class of a basic smooth disk. In this case $\partial \beta' = \bb_{i_0}$ for some $0\leq i_0\leq m-1$ and $\bb_\infty=-\bb_{i_0}$. So the only possible $d\in \bar{\mathbb{K}}_\mathrm{eff}$ comes from the relation $\bb_{i_0}+\bb_\infty = 0$. In this case the necessary term in the equivariant $I$-function of $\bar{\bX}$ is $y^{d_\infty}$, where $d_\infty = e_{i_0} + e_\infty = \bar{\beta}' \in H_2(\bar{\bX};\Q)$.\\

\item
{\bf Case 2: $\beta$ is an orbi-disk class.} This means that $\beta = \beta'+\alpha$ with $\alpha \in H_2(\bX)$ and $\beta'=\beta_{\nu_{j_0}}\in \pi_2(\bX, L)$ is the class of a basic orbi-disk corresponding to $\bb_{j_0} \in \mathrm{Box}'(\Sigma)^{\textrm{age}=1}$ for some $m\leq j_0\leq m'-1$. In this case $\partial \beta' = \bb_{j_0}$ and $\bb_\infty=-\bb_{j_0}$. So the only possible $d\in \bar{\mathbb{K}}_{\mathrm{eff}}$ comes from the relation $\bb_{j_0}+\bb_\infty = 0$. In this case the necessary term in the equivariant $I$-function of $\bar{\bX}$ is $y^{d_\infty}$, where $d_\infty = e_{j_0} + e_\infty$. Note that in this case, $d_\infty$ is {\em not} a class in $H_2(\bar{\bX};\Q)$.
\end{itemize}
Equating the relevant $1/z^2$-terms in the equivariant $I$- and $J$-functions yields
\begin{equation}\label{eqn:formula_needed_equiv_closed_invs}
y^{d_\infty}=\frac{q_0(y)^2}{2}+\sum_{d\in H_2^\textrm{eff}(\bar{\bX})}\sum_{l\geq 0}\sum_{\nu_1,\ldots,\nu_l\in \mathrm{Box}'(\Sigma)^{\mathrm{age}=1}}
\frac{\prod_{i=1}^l \tau_{\nu_i}}{l!}\langle [\mathrm{pt}]_\mathbb{T}, \prod_{i=1}^l \mathbf{1}_{\bar{\nu}_i}\rangle_{0, l+1, d}^{\bar{\bX}, \mathbb{T}}q^d.
\end{equation}

\subsubsection{Computing toric mirror maps}\label{sec:toric_mir_maps}
To explicitly evaluate (\ref{eqn:formula_needed_equiv_closed_invs}), we compute the toric mirror map for $\bar{\bX}$, which is part of the $1/z$-term in the expansion of the equivariant $I$-function.

Let $d\in \bar{\mathbb{K}}_\mathrm{eff}$. Similar to the calculations in the previous section, we first examine the $(1/z)$-series expansion of the corresponding term in the equivariant $I$-function of $\bar{\bX}$:
\begin{align*}
& y^d\prod_{i\in\{0,\ldots,m'-1\}\cup\{\infty\}}
\frac{\prod_{k=\lceil\langle D_i,d\rangle\rceil}^\infty(\bar{D}^\mathbb{T}_i+(\langle D_i,d\rangle-k)z)}
{\prod_{k=0}^\infty(\bar{D}^\mathbb{T}_i+(\langle D_i,d\rangle-k)z)}\mathbf{1}_{\nu(d)}\\
= & \frac{y^d}{z^{\langle \hat{\rho}(\bar{\bX}), d \rangle+\textrm{age}(\nu(d))}}
\prod_{i\in\{0,\ldots,m'-1\}\cup\{\infty\}}
\frac{\prod_{k=\lceil\langle D_i,d\rangle\rceil}^\infty(\bar{D}^\mathbb{T}_i/z+(\langle D_i,d\rangle-k))}
{\prod_{k=0}^\infty(\bar{D}^\mathbb{T}_i/z+(\langle D_i,d\rangle-k))}\mathbf{1}_{\nu(d)}.
\end{align*}
We need the $1/z$-term that takes value in $H_{\mathrm{CR}, \mathbb{T}}^{\leq 2}(\bar{\bX})$. There are three types.

\begin{itemize}
\item
{\bf degree $0$ term:} This requires that $\nu(d)=0$. As noted above, this implies $\langle D_i, d\rangle\in \Z$ for all $i$. Furthermore, we must have $\langle D_i, d\rangle\geq 0$ for all $i$ in order for the term to be of cohomological degree $0$. Also, we need $1/z^{\langle \hat{\rho}(\bar{\bX}), d\rangle+\mathrm{age}(\nu(d))}=1/z$, which means that $\langle \hat{\rho}(\bar{\bX}), d\rangle=1$.  Consequently $\langle D_i, d\rangle=1$ for exactly one $D_i$ and $=0$ otherwise. As we have seen, such a class $d\in \bar{\mathbb{K}}_\mathrm{eff}$ does not exist. So there is no $H^0(\bar{\bX})$-term.\\

\item
{\bf degree $2$ term from untwisted sector:} This means terms proportional to $\mathbb{T}$-divisors $\bar{D}_i^\mathbb{T}$. Again this requires that $\nu(d)=0$, which implies $\langle D_i, d\rangle\in \Z$ for all $i$. Furthermore, we must have exactly one $\bar{D}^\mathbb{T}_j/z$, which requires $\langle D_j, d\rangle<0$ for this $j$ and $\langle D_i, d\rangle\geq0$ for all $i\neq j$. To get the $1/z$-term, we need $\langle \hat{\rho}(\bar{\bX}), d\rangle+\mathrm{age}(\nu(d))=0$, so we should have $\langle \hat{\rho}(\bar{\bX}),d\rangle = 0$.

For each $j\in \{0,1,\ldots,m-1\}\cup\{\infty\}$, we define
$$\Omega^{\bar{\bX}}_j:=\{d\in \bar{\mathbb{K}}_\mathrm{eff} \mid \langle \hat{\rho}(\bar{\bX}),d\rangle = 0, \nu(d)=0, \langle D_j,d\rangle \in \Z_{<0}\textrm{ and } \langle D_i,d\rangle \in \Z_{\geq 0}\ \forall i\neq j\},$$
and set
$$A_j^{\bar{\bX}}(y):=\sum_{d\in \Omega^{\bar{\bX}}_j}y^d \frac{(-1)^{-\langle D_j,d\rangle-1}(-\langle D_j,d\rangle-1)!}{\prod_{i\neq j}\langle D_i,d\rangle!}.$$
Then the degree $2$ term from untwisted sector is given by $$\sum_{j=0}^{m-1} A^{\bar{\bX}}_j(y)\bar{D}^\mathbb{T}_j/z+A^{\bar{\bX}}_\infty(y) \bar{D}^\mathbb{T}_\infty/z.$$

\item
{\bf degree $2$ term from twisted sectors:} This requires that $\nu(d)=\nu$. Since $\mathrm{age}(\nu)=1$, we must have $\langle \hat{\rho}(\bar{\bX}), d\rangle=0$. In order to avoid being proportional to a $\mathbb{T}$-divisor, $\langle D_i,d\rangle$ cannot be a negative integer for any $i$.

For each $j\in \{m,m+1,\ldots,m'-1\}$, we define
$$\Omega^{\bar{\bX}}_j:=\{d\in \bar{\mathbb{K}}_\mathrm{eff} \mid \langle \hat{\rho}(\bar{\bX}), d\rangle = 0, \nu(d)=\bb_j
\textrm{ and } \langle D_i,d\rangle \notin \Z_{<0}\ \forall i\},$$
and set
$$A^{\bar{\bX}}_j(y) := \sum_{d\in \Omega^{\bar{\bX}}_j} y^d \prod_{i\in\{0,\ldots,m'-1\}\cup\{\infty\}}
\frac{\prod_{k=\lceil\langle D_i,d\rangle\rceil}^\infty(\langle D_i,d\rangle-k)}{\prod_{k=0}^\infty(\langle D_i,d\rangle-k)}.$$
Then the degree $2$ term from twisted sectors is $$\sum_{j=m}^{m'-1} A^{\bar{\bX}}_j(y)\mathbf{1}_{\bb_j}/z.$$
\end{itemize}

The fan sequence of $\bar{\bX}$ is given by $0\to \text{ker}\to \widetilde{N}^{-}:=\widetilde{N}\oplus \mathbb{Z}\to N\to 0,$
and the divisor sequence of $\bar{\bX}$ is given by $0\to M\to \widetilde{M}^{-}:=(\widetilde{N}^{-})^\vee\to \bar{\mathbb{L}}^\vee\to 0.$
Observe that $\textrm{rk}(\bar{\mathbb{L}}^\vee) = \textrm{rk}({\mathbb{L}^\vee})+1 = r+1 = m'+1-n$ and $\textrm{rk}(H^2(\bar{\bX})) = \textrm{rk}(H^2(\bX)) + 1 = r'+1 = m+1-n$. We choose an integral basis
$$\{p_1, \ldots, p_r, p_\infty\}\subset \bar{\mathbb{L}}^\vee$$
such that $p_a$ is in the closure of $\widetilde{C}_{\bar{\bX}}$ for all $a$ and $p_{r'+1}, \ldots, p_r \in \sum_{i=m}^{m'-1}\mathbb{R}_{\geq 0}D_i$ so that the images $\{\bar{p}_1,\ldots,\bar{p}_{r'},\bar{p}_\infty\}$ of $\{p_1, \ldots, p_{r'}, p_\infty\}$ under the quotient $\bar{\mathbb{L}}^\vee\otimes \Q\to H^2(\bar{\bX}; \Q)$ form a nef basis of $H^2(\bar{\bX};\Q)$ and $\bar{p}_a = 0$ for $a=r'+1,\ldots,r$. And we pick $\{p_1^\mathbb{T},...,p_r^{\mathbb{T}}, p_\infty^{\mathbb{T}}\}\subset \widetilde{M}^{-}$ in the way described in Section \ref{sec:div_cone_etc}. We further assume that $\{p_1, \ldots, p_r\}$ gives the original basis of $\mathbb{L}^\vee$ chosen for $\bX$.

Expressing $D_i$ in terms of the basis $\{p_a\}$ defines an integral matrix $(Q_{ia})$ by
$$D_i = \sum_{a\in\{1,\ldots,r\}\cup\{\infty\}} Q_{ia} p_a,\quad Q_{ia}\in\mathbb{Z}.$$
As above, the image of $D_i$ under the quotient $\bar{\mathbb{L}}^\vee\otimes \Q\to H^2(\bar{\bX}; \Q)$ is denoted by $\bar{D}_i$. Then for $i\in\{0,\ldots,m-1\}\cup\{\infty\}$, the class $\bar{D}^\mathbb{T}_i$ of the toric prime $\mathbb{T}$-divisor $D^\mathbb{T}_i$ is given by
$$\bar{D}^\mathbb{T}_i = \lambda_i+\sum_{a\in\{1,\ldots,r'\}\cup\{\infty\}} Q_{ia}\bar{p}^\mathbb{T}_a, \quad \lambda_i\in H_\mathbb{T}^2(\text{pt});$$
and for $i=m,\ldots, m'-1$, $\bar{D}^\mathbb{T}_i=0$ in $H^2(\bX; \R)$.

Hence the coefficient of the $1/z$-term in the equivariant $I$-function can be expressed as
\begin{equation}\label{eqn:1/z_terms_I}
\begin{split}
  & \sum_{a\in\{1,\ldots,r'\}\cup\{\infty\}} \bar{p}^\mathbb{T}_a\log y_a + \sum_{j\in \{0,\ldots,m-1\}\cup\{\infty\}} A^{\bar{\bX}}_j(y)\bar{D}^\mathbb{T}_j
  + \sum_{j=m}^{m'-1} A^{\bar{\bX}}_j(y)\mathbf{1}_{\bb_j} \\
= & \sum_{a\in\{1,\ldots,r'\}\cup\{\infty\}} \left(\log y_a + \sum_{j\in \{0,\ldots,m-1\}\cup\{\infty\}} Q_{ja} A^{\bar{\bX}}_j(y)\right) \bar{p}^\mathbb{T}_a
+ \sum_{j=m}^{m'-1} A^{\bar{\bX}}_j(y)\mathbf{1}_{\bb_j}+ \sum_{j\in \{0,\ldots,m-1\}\cup\{\infty\}} \lambda_jA^{\bar{\bX}}_j(y).
\end{split}
\end{equation}
On the other hand, the coefficient of the $1/z$-term in the $J$-function is given by
\begin{equation}\label{eqn:1/z_terms_J}
\sum_{a\in\{1,\ldots,r'\}\cup\{\infty\}} \bar{p}^\mathbb{T}_a\log q_a + \tau_\mathrm{tw} = \sum_{a=1}^r \bar{p}^\mathbb{T}_a\log q_a +\sum_{j=m}^{m'-1}\tau_{\bb_j} \mathbf{1}_{\bb_j}.
\end{equation}
The toric mirror map for $\bar{\bX}$ is obtained by comparing \eqref{eqn:1/z_terms_I} and \eqref{eqn:1/z_terms_J}:
\begin{equation}\label{eqn:toric_mirror_map}
\begin{split}
\log q_a & = \log y_a + \sum_{j\in \{0,\ldots,m-1\}\cup\{\infty\}} Q_{ja}A^{\bar{\bX}}_j(y), \quad a \in \{1,\ldots,r'\}\cup\{\infty\},\\
\tau_{\bb_j} & = A^{\bar{\bX}}_j(y), \quad j=m,\ldots,m'-1,
\end{split}
\end{equation}
and set $q_0(y):= \sum_{j\in \{0,\ldots,m-1\}\cup\{\infty\}} \lambda_jA^{\bar{\bX}}_j(y)$.

Let us have a closer look at the toric mirror map \eqref{eqn:toric_mirror_map} for $\bar{\bX}$. First of all, recall that $\bar{\mathbb{K}}_\textrm{eff} = \mathbb{K}_\textrm{eff} \oplus \Z_{\geq0} d_\infty$, so we can decompose any $d \in \bar{\mathbb{K}}_\textrm{eff}$ as
$d = d' + k d_\infty$, where $d' \in \mathbb{K}_\textrm{eff}$ and $k \in \Z_{\geq0}$. Suppose that $\langle \hat{\rho}(\bar{\bX}), d\rangle = 0$. Then we have
$ 0 = \sum_{i=0}^{m'-1} \langle D_i, d'\rangle + \langle D_\infty, d\rangle = \langle \hat{\rho}(\bX), d'\rangle + k$.
But $\bX$ is semi-Fano, so $\langle \hat{\rho}(\bX), d'\rangle \geq 0$. This implies that $\langle D_\infty, d\rangle = k = 0$, and hence $d = d' \in \mathbb{K}_\textrm{eff}$.

As an immediate consequence, we have $A^{\bar{\bX}}_\infty = 0$, since $d \in \Omega^{\bar{\bX}}_\infty$ implies that $\langle \hat{\rho}(\bar{\bX}), d\rangle = 0$ and $\langle D_\infty, d\rangle < 0$ which is impossible and so $\Omega^{\bar{\bX}} = \emptyset$. Also for $j\in \{0,1,\ldots,m-1,m,\ldots, m'-1\}$, $d \in \Omega^{\bar{\bX}}_j$ implies that $\langle \hat{\rho}(\bar{\bX}), d\rangle = 0$, so $d$ lies in $\mathbb{K}_\textrm{eff}$ and hence we have $\Omega^{\bar{\bX}}_j = \Omega^\bX_j$, where
$$\Omega^\bX_j := \{d\in \mathbb{K}_\mathrm{eff} \mid \nu(d)=0, \langle D_j,d\rangle \in \Z_{<0}\textrm{ and } \langle D_i,d\rangle \in \Z_{\geq 0}\ \forall i\neq j\}, \quad j = 0, 1, \ldots, m-1,$$
$$\Omega^\bX_j:=\{d\in \mathbb{K}_\mathrm{eff} \mid \nu(d)=\bb_j\textrm{ and } \langle D_i,d\rangle \notin \Z_{<0}\ \forall i\}, \quad j = m, m+1, \ldots, m'-1.$$
Here we have used the fact that $\hat{\rho}(\bX) = 0$.

\begin{prop}\label{prop:mir_map_barX}
The toric mirror map of the toric compactification $\bar{\bX}$ is of the form
\begin{equation}\label{eqn:toric_mirror_map_barX_1}
\begin{split}
\log q_a = & \log y_a + \sum_{j=0}^{m-1} Q_{ja}A^\bX_j(y), \quad a = 1,\ldots,r',\\
\log q_\infty = & \log y_\infty + A^\bX_{i_0}(y),\\
\tau_{\bb_j} = & A^\bX_j(y), \quad j = m,\ldots,m'-1,
\end{split}
\end{equation}
when $\beta = \beta_{i_0} + \alpha$ is a smooth disk class, and of the form
\begin{equation}\label{eqn:toric_mirror_map_barX_2}
\begin{split}
\log q_a = & \log y_a + \sum_{j=0}^{m-1} Q_{ja}A^\bX_j(y), \quad a = 1,\ldots,r',\\
\log q_\infty = & \log y_\infty,\\
\tau_{\bb_j} = & A^\bX_j(y), \quad j=m,\ldots,m'-1,
\end{split}
\end{equation}
when $\beta = \beta_{\nu_{j_0}} + \alpha$ is an orbi-disk class, where
\begin{equation}\label{eqn:toric_mirror_map_revised1}
A^\bX_j(y) := \sum_{d\in \Omega^\bX_j}y^d \frac{(-1)^{-\langle D_j,d\rangle-1}(-\langle D_j,d\rangle-1)!}{\prod_{i\neq j}\langle D_i,d\rangle!}, \quad j = 0, 1, \ldots, m-1,
\end{equation}
\begin{equation}\label{eqn:toric_mirror_map_revised2}
A^\bX_j(y) := \sum_{d\in \Omega^\bX_j}y^d \prod_{i=0}^{m'-1}
\frac{\prod_{k=\lceil\langle D_i,d\rangle\rceil}^\infty(\langle D_i,d\rangle-k)}{\prod_{k=0}^\infty(\langle D_i,d\rangle-k)}, \quad j = m, m+1, \ldots, m'-1.
\end{equation}
\end{prop}
\begin{proof}
We already have $\Omega^{\bar{\bX}}_\infty = \emptyset$ and $\Omega^{\bar{\bX}}_j = \Omega^\bX_j$ for $j = 0,\ldots,m'-1$. Also, $d\in \Omega^{\bar{\bX}}_j = \Omega^\bX_j$ implies that $\langle D_\infty, d\rangle = 0$. Thus we have $A^{\bar{\bX}}_\infty = 0$ and $A^{\bar{\bX}}_j = A^\bX_j$ for $j = 0,\ldots,m'-1$. Finally, when $\beta = \beta_{i_0} + \alpha$ is a smooth disk class, we have $Q_{j\infty}=1$ for $j\in \{i_0,\infty\}$ and $Q_{j\infty}=0$ for $j\notin\{i_0,\infty\}$; whereas when $\beta = \beta_{\nu_{j_0}} + \alpha$ is an orbi-disk class, we have $Q_{j\infty}=1$ for $j\in \{j_0,\infty\}$ and $Q_{j\infty}=0$ for $j\notin\{j_0,\infty\}$, and in particular, $Q_{j\infty}=0$ for all $j=0,\ldots,m-1$. The result now follows from \eqref{eqn:toric_mirror_map}.
\end{proof}

A key observation is that in both cases \eqref{eqn:toric_mirror_map_barX_1} and \eqref{eqn:toric_mirror_map_barX_2}, the toric mirror map of $\bar{\bX}$ contains parts which depend only on $\bX$:
\begin{prop}
The toric mirror map for the toric CY orbifold $\bX$ is given by
\begin{equation}\label{eqn:toric_mirror_map_X}
\begin{split}
\log q_a = & \log y_a + \sum_{j=0}^{m-1} Q_{ja}A^\bX_j(y), \quad a = 1,\ldots,r',\\
\tau_{\bb_j} = & A^\bX_j(y), \quad j=m,\ldots,m'-1,
\end{split}
\end{equation}
where the functions $A^\bX_j(y)$ are defined in \eqref{eqn:toric_mirror_map_revised1} and \eqref{eqn:toric_mirror_map_revised2} in Proposition \ref{prop:mir_map_barX}.
\end{prop}
\begin{proof}
This can be seen by exactly the same calculations as in this subsection applied to the equivariant $I$-function of $\bX$; see also \cite[Section 4.1]{fang-liu-tseng}.
\end{proof}
\begin{remark} \label{rmk:non_equiv_lim_mir_map}
\hfill
\begin{enumerate}
\item
In the non-equivariant limit $H_\mathbb{T}^*(\text{pt})\to H^*(\text{pt})$, we have $\lambda_i\to 0$. Hence $q_0(y)\to 0$ in the non-equivariant limit.
\item
It is clear from the description that (\ref{eqn:toric_mirror_map_barX_1}), (\ref{eqn:toric_mirror_map_barX_2}), (\ref{eqn:toric_mirror_map_X}) do not depend on $\mathbb{T}$-actions, and remain unchanged in the non-equivariant limit $H_\mathbb{T}^*(\text{pt})\to H^*(\text{pt})$.
\item
Also note that, for $j=m,m+1,\ldots,m'-1$,
$$A^\bX_j(y) = y^{D^\vee_j} + \textrm{higher order terms},$$
where $D_j^{\vee}\in \mathbb{K}_\textrm{eff}$ is the class described in \eqref{eqn:dual_of_D_j}.
\end{enumerate}
\end{remark}

\subsection{Explicit formulas}
In this subsection we combine previous discussions to derive explicit formulas for generating functions of genus $0$ open orbifold GW invariants of $\bX$. First we discuss non-equivariant limits.
\begin{prop}\label{prop:non_equiv_lim_inv}
The non-equivariant limit of $\langle [\mathrm{pt}]_\mathbb{T}, \prod_{i=1}^l \mathbf{1}_{\bar{\nu}_i}\rangle_{0, l+1, d}^{\bar{\bX}, \mathbb{T}}$ is $\langle [\mathrm{pt}], \prod_{i=1}^l \mathbf{1}_{\bar{\nu}_i}\rangle_{0, l+1, d}^{\bar{\bX}}$.
\end{prop}
\begin{proof}
If $\bar{\bX}$ is projective (this is the case when $\bb_{i_0}\in N$ lies in the interior of the support $|\Sigma|$ by Proposition \ref{prop:proj}), then moduli spaces of stable maps to $\bar{\bX}$ of fixed genus, degree, and number of marked points is compact. In this case the result follows by the discussion in Section \ref{sec:equiv_GW_inv}.

Suppose that $\bar{\bX}$ is semi-projective but not projective. As noted in Remark \ref{rmk:curve_moduli_cpt}, the moduli space $\mathcal{M}^{cl}_{1+l}(\bar{\bX}, \bar{\beta}, \bar{\bx},p)$ used to define the invariant $\langle [\mathrm{pt}], \prod_{i=1}^l \mathbf{1}_{\bar{\nu}_i}\rangle_{0, l+1, d}^{\bar{\bX}}$ is compact for $p\in L$. In fact it is straightforward to check that $\mathcal{M}^{cl}_{1+l}(\bar{\bX}, \bar{\beta}, \bar{\bx},p)$ is compact for any $p$, using the arguments in the proof of Proposition \ref{prop:cpt_disk_moduli}. A standard cobordism argument shows that the invariant $\langle [\mathrm{pt}], \prod_{i=1}^l \mathbf{1}_{\bar{\nu}_i}\rangle_{0, l+1, d}^{\bar{\bX}}$ does not depend on the choice of $p$. If $p\in \bar{\bX}$ is a $\mathbb{T}$-fixed point, then $\mathbb{T}$ acts on $\mathcal{M}^{cl}_{1+l}(\bar{\bX}, \bar{\beta}, \bar{\bx},p)$ and for such $p$ the moduli space $\mathcal{M}^{cl}_{1+l}(\bar{\bX}, \bar{\beta}, \bar{\bx},p)$ can be used to define $\mathbb{T}$-equivariant GW invariant $\langle [\mathrm{pt}]_\mathbb{T}, \prod_{i=1}^l \mathbf{1}_{\bar{\nu}_i}\rangle_{0, l+1, d}^{\bar{\bX}, \mathbb{T}}$. Choose $p\in \bar{\bX}$ to be a $\mathbb{T}$-fixed point and argue as in Section \ref{sec:equiv_GW_inv}, the result follows.
\end{proof}

This proposition allows us to obtain the following
\begin{prop}
Using the notations in Section \ref{sec:open_closed_equality}, we have
\begin{equation}\label{eqn:formula_for_generating_function}
y^{d_\infty}=q^{\bar{\beta}'}\sum_{\alpha\in H_2^\textrm{eff}(\bX)}\sum_{l\geq 0}\sum_{\nu_1,\ldots,\nu_l\in \mathrm{Box}'(\Sigma)^{\mathrm{age}=1}}\frac{\prod_{i=1}^l\tau_{\nu_i}}{l!}n^\bX_{1,l,\beta'+\alpha}([\mathrm{pt}]_L; \prod_{i=1}^l \mathbf{1}_{\nu_i})q^\alpha.
\end{equation}
\end{prop}
\begin{proof}
In view of Remark \ref{rmk:non_equiv_lim_mir_map} and Proposition \ref{prop:non_equiv_lim_inv}, the non-equivariant limit of (\ref{eqn:formula_needed_equiv_closed_invs}) gives
\begin{equation}\label{eqn:formula_needed_closed_invs}
y^{d_\infty}=\sum_{d\in H_2^\textrm{eff}(\bar{\bX})}\sum_{l\geq 0}\sum_{\nu_1,\ldots,\nu_l\in \mathrm{Box}'(\Sigma)^{\mathrm{age}=1}}
\frac{\prod_{i=1}^l \tau_{\nu_i}}{l!}\langle [\mathrm{pt}], \prod_{i=1}^l \mathbf{1}_{\bar{\nu}_i}\rangle_{0, l+1, d}^{\bar{\bX}}q^d.
\end{equation}
By dimension reason, the invariant $\langle[\mathrm{pt}], \prod_{i=1}^l \mathbf{1}_{\bar{\nu}_i}\rangle_{0, l+1, d}^{\bar{\bX}}$ vanishes unless $c_1(\bar{\bX})\cdot d = 2$. Now we have $H_2^\textrm{eff}(\bar{\bX}) = \Z_{\geq0}\bar{\beta}' \oplus H_2^\textrm{eff}(\bX)$. Also $\bar{\bX}$ is semi-Fano and $c_1(\bar{\bX})\cdot \bar{\beta}' = 2$. So $c_1(\bar{\bX}) \cdot d = 2$ implies that $d$ must be of the form $\bar{\beta}' + \alpha$ where $\alpha\in H_2^\textrm{eff}(\bX)$ has Chern number $c_1(\bar{\bX})\cdot \alpha = 0$. The formula \eqref{eqn:formula_for_generating_function} then follows from the open/closed equality \eqref{eqn:open_closed_equality}.
\end{proof}

Recall the choice of $\tau_\mathrm{tw}$ in (\ref{eqn:choice_of_tau_tw}). \eqref{eqn:formula_for_generating_function} can also be written in a more succinct way as
\begin{align*}
y^{d_\infty} = q^{\bar{\beta}'}\sum_{\alpha\in H_2^\textrm{eff}(\bX)} \sum_{l\geq 0} \frac{1}{l!} n^\bX_{1,l,\beta'+\alpha}([\mathrm{pt}]_L; \prod_{i=1}^l \tau_\mathrm{tw})q^\alpha,
\end{align*}
where $\tau_\mathrm{tw} = \sum_{\nu\in \mathrm{Box}'(\Sigma)^{\mathrm{age}(\nu)=1}}\tau_\nu \mathbf{1}_{\bar{\nu}}$.

Recall that \eqref{eqn:Lag_isotopy_X} gives a Lagrangian isotopy between a moment map fiber $L$ and a fiber $F_r$ of the Gross fibration when $r$ lies in the chamber $B_+$. Hence \eqref{eqn:formula_for_generating_function} also computes the generating functions of genus 0 open orbifold GW invariants defined in \eqref{eqn:generating_functions_open_GW}:
\begin{align*}
y^{d_\infty} = q^{\bar{\beta}'} (1 + \delta_j),
\end{align*}
when $\beta'$ corresponds to $\beta_j(r)$ under the isotopy \eqref{eqn:Lag_isotopy_X}, and
\begin{align*}
y^{d_\infty} = q^{\bar{\beta}'} \tau_\nu(1 + \delta_\nu),
\end{align*}
when $\beta'$ corresponds to $\beta_\nu(r)$ under the isotopy \eqref{eqn:Lag_isotopy_X}.

The formula \eqref{eqn:formula_for_generating_function} identifies the generating function of genus 0 open orbifold GW invariants with $y^{d_\infty}q^{-\bar{\beta}'}$. We can now derive an even more explicit formula for computing the orbi-disk invariants using our results in the previous subsection.

\begin{thm}\label{thm:sm_disk_gen_function}
If $\beta'=\beta_{i_0}$ is a basic smooth disk class corresponding to the ray generated by $\bb_{i_0}$ for some $i_0\in \{0,1,\ldots,m-1\}$, then we have
\begin{equation}\label{eqn:formula_for_generating_function_sm_disk}
\sum_{\alpha\in H_2^\textrm{eff}(\bX)}\sum_{l\geq 0}\sum_{\nu_1,\ldots,\nu_l\in \mathrm{Box}'(\Sigma)^{\mathrm{age}=1}}\frac{\prod_{i=1}^l\tau_{\nu_i}}{l!}n^\bX_{1,l,\beta_{i_0}+\alpha}([\mathrm{pt}]_L; \prod_{i=1}^l \mathbf{1}_{\nu_i})q^\alpha = \exp\left(-A^\bX_{i_0}(y(q,\tau))\right)
\end{equation}
via the inverse $y = y(q,\tau)$ of the toric mirror map \eqref{eqn:toric_mirror_map_X} of $\bX$.
\end{thm}
\begin{proof}
Recall that in this case, we have $d_\infty = \bar{\beta}'$. Also, $D_\infty = p_\infty$. So $\langle p_\infty, d_\infty \rangle = 1$. On the other hand, since $d_\infty \in H_2(\bar{\bX};\Q)$, we have $\langle \bar{D}_i, d_\infty\rangle = \langle D_i, d_\infty\rangle$ for any $i$ and $\langle \bar{p}_a, d_\infty\rangle = \langle p_a, d_\infty\rangle$ for any $a$. Using the toric mirror map \eqref{eqn:toric_mirror_map_barX_1} for $\bar{\bX}$, we have
\begin{align*}
\log q^{d_\infty}
& = \sum_{a=1}^{r'} \langle \bar{p}_a, d_\infty\rangle \log q_a
+ \langle \bar{p}_\infty, d_\infty\rangle \log q_\infty\\
& = \sum_{a=1}^{r'} \langle \bar{p}_a, d_\infty\rangle \left(\log y_a + \sum_{i=0}^{m-1} Q_{ia}A^\bX_i(y)\right)
+ \left(\log y_\infty + A^\bX_{i_0}(y)\right)\\
& = \log y^{d_\infty} + A^\bX_{i_0}(y) + \sum_{i=0}^{m-1} \left(\langle D_i, d_\infty\rangle - Q_{i\infty}\right)A^\bX_i(y).
\end{align*}
But $\langle D_i, d_\infty\rangle = Q_{i\infty}$ for $i=0,\ldots,m-1$, so we arrive at the desired formula.
\end{proof}

\begin{thm}\label{thm:orbi_disk_gen_function}
If $\beta'=\beta_{\nu_{j_0}}$ is a basic orbi-disk class corresponding to $\nu_{j_0}\in \textrm{Box}'(\Sigma)^{\textrm{age}=1}$ for some $j_0\in \{m,m+1,\ldots,m'-1\}$, then we have
\begin{equation}\label{eqn:formula_for_generating_function_orb_disk}
\sum_{\alpha\in H_2^\textrm{eff}(\bX)}\sum_{l\geq 0}\sum_{\nu_1,\ldots,\nu_l\in \mathrm{Box}'(\Sigma)^{\mathrm{age}=1}}\frac{\prod_{i=1}^l\tau_{\nu_i}}{l!}n^\bX_{1,l,\beta_{\nu_{j_0}}+\alpha}([\mathrm{pt}]_L; \prod_{i=1}^l \mathbf{1}_{\nu_i})q^\alpha = y^{D^\vee_{j_0}} \exp\left(-\sum_{i\notin I_{j_0}}c_{j_0i}A^\bX_i(y(q,\tau)) \right),
\end{equation}
via the inverse $y=y(q,\tau)$ of the toric mirror map \eqref{eqn:toric_mirror_map_X} of $\bX$, where $D_{j_0}^\vee \in \mathbb{K}_\textrm{eff}$ is the class defined in \eqref{eqn:dual_of_D_j}, $I_{j_0}\in \mathcal{A}$ is the anticone of the minimal cone containing $\bb_{j_0} = \nu_{j_0}$ and $c_{j_0i} \in \Q\cap [0,1)$ are rational numbers such that $\bb_{j_0} = \sum_{i\notin I_{j_0}} c_{j_0i}\bb_i$.
\end{thm}
\begin{proof}
In this case, the class $\bar{\beta}' \in H_2(\bar{\bX};\Q)$ is given by
$$\bar{\beta}' = (\sum_{i\notin I_{j_0}} c_{j_0i} e_i) + e_\infty \in \widetilde{N} \oplus \Z e_\infty = \bigoplus_{i=0}^{m'-1}\Z e_i \oplus \Z e_\infty;$$
while $d_\infty = e_{j_0} + e_\infty$ (recall that this $d_\infty$ is {\em not} a class in $H_2(\bar{\bX};\Q)$). Hence $d_\infty - \bar{\beta}'$ is precisely the class $D_{j_0}^\vee \in \mathbb{K}_\textrm{eff}$. So we can write $y^{d_\infty}q^{-\bar{\beta}'} = y^{D^\vee_{j_0}} y^{\bar{\beta}'}q^{-\bar{\beta}'}$.

Now,
$$\log y^{\bar{\beta}'} = \sum_{a=1}^r \langle p_a, \bar{\beta}'\rangle \log y_a + \langle p_\infty, \bar{\beta}'\rangle \log y_\infty,$$
and using the toric mirror map \eqref{eqn:toric_mirror_map_barX_2} for $\bar{\bX}$, we have
\begin{align*}
\log q^{\bar{\beta}'}
& = \sum_{a=1}^{r'} \langle \bar{p}_a, \bar{\beta}'\rangle \log q_a
+ \langle \bar{p}_\infty, \bar{\beta}'\rangle \log q_\infty\\
& = \sum_{a=1}^{r'} \langle \bar{p}_a, \bar{\beta}'\rangle \log y_a + \sum_{i=0}^{m-1} \left( \sum_{a=1}^{r'} Q_{ia} \langle \bar{p}_a, \bar{\beta}'\rangle \right) A^\bX_i(y) + \langle \bar{p}_\infty, \bar{\beta}'\rangle \log y_\infty.
\end{align*}
Since $Q_{i\infty} = 0$ for $i=0,\ldots,m-1$, we have $\sum_{a=1}^{r'} Q_{ia} \langle \bar{p}_a, \bar{\beta}'\rangle = \langle \bar{D}_i, \bar{\beta}'\rangle$. Also, since $\bar{\beta}' \in H_2(\bar{\bX};\Q)$, we have $\langle \bar{D}_i, \bar{\beta}'\rangle = \langle D_i, \bar{\beta}'\rangle$ for any $i$ (and $\langle \bar{p}_a, \bar{\beta}'\rangle = \langle p_a, \bar{\beta}'\rangle$ for any $a$), so
$$\sum_{i=0}^{m-1} \left( \sum_{a=1}^{r'} Q_{ia} \langle \bar{p}_a, \bar{\beta}'\rangle \right) A^\bX_i(y) = \sum_{i\notin I_{j_0}} c_{j_0i}A^\bX_i(y),$$
and hence
$\log y^{\bar{\beta}'} - \log q^{\bar{\beta}'} = - \sum_{i\notin I_{j_0}} c_{j_0i}A^\bX_i(y)$.
The formula follows.
\end{proof}


As a by-product of our calculations, we obtain the following convergence result:
\begin{corollary}\label{cor:convergence}
The generating series of genus 0 open orbifold GW invariants
\begin{equation*}
\sum_{\alpha\in H_2^\textrm{eff}(\bX)}\sum_{l\geq 0}\sum_{\nu_1,\ldots,\nu_l\in \mathrm{Box}'(\Sigma)^{\mathrm{age}=1}}\frac{\prod_{i=1}^l\tau_{\nu_i}}{l!}n^\bX_{1,l,\beta'+\alpha}([\mathrm{pt}]_L; \prod_{i=1}^l \mathbf{1}_{\nu_i})q^\alpha.
\end{equation*}
in \eqref{eqn:formula_for_generating_function} and hence those in \eqref{eqn:generating_functions_open_GW} are convergent power series in the variables $q_a$'s and $\tau_{\nu_i}$'s.
\end{corollary}
\begin{proof}
As noted in \cite[Section 4.1]{iritani09}, the toric mirror map \eqref{eqn:toric_mirror_map_X} is a local isomorphism near $y=0$. The inverse of \eqref{eqn:toric_mirror_map_X} is therefore also analytic near $q=0$, which allows us to express the variables $y_a$'s as convergent power series in the variables $q_a$'a and $\tau_{\nu_i}$'s. Also note that the expressions in \eqref{eqn:formula_for_generating_function_sm_disk} and \eqref{eqn:formula_for_generating_function_orb_disk} are convergent power series in the variables $y_a$. The result follows.
\end{proof}

\subsection{Examples}\label{sec:examples_mirror}
\noindent{(1)} $\bX = [\cpx^2 / \Z_m]$ (Example (1) in Section \ref{sect:eg_SYZ}).  There are $m-1$ twisted sectors $\nu_j$, $j = 1, \ldots, m-1$, and each corresponds to a basic orbi-disk class $\beta_{\nu_j}$.  The generating functions of genus 0 open orbifold GW invariants are $\tau_j + \delta_{\nu_j}(\tau)$ given in \eqref{eqn:gen_ex1}.
By Theorem \ref{thm:orbi_disk_gen_function}, this is equal to the inverse of the toric mirror map. The toric mirror map for $\bX$ was computed explicitly in \cite{CCIT09}: $ \tau_r = g_r (y)$, where
\begin{align*}
g_r(y) = \sum_{\substack{k_1, \ldots, k_{m-1} \geq 0\\ \langle b(k) \rangle = r/m}} \frac{y_1^{k_1}\ldots y_{n-1}^{k_{m-1}}}{k_1! \ldots k_{m-1}!} \frac{\Gamma(\langle D_0(k)\rangle)}{\Gamma(1+ D_0(k))} \frac{\Gamma(\langle D_m(k)\rangle)}{\Gamma(1+D_m(k))},\\
b(k) = \sum_{i=1}^{m-1} \frac{i}{n} k_i,\ D_0(k) = -\frac{1}{m} \sum_{i=1}^{m-1} (m-i) k_i,\ D_m(k) = -\frac{1}{m} \sum_{i=1}^{m-1} i k_i.
\end{align*}
Denote the inverse of $(g_1(y), \ldots, g_{m-1}(y))$ by $(f_1(\tau),\ldots,f_{m-1}(\tau))$.  Then
$ f_j(\tau) = \tau_j + \delta_{\nu_j}(\tau)$ for $j=1,...,m-1$.
The inverse mirror map $(f_1(\tau),\ldots,f_{m-1}(\tau))$ was computed in \cite[Proposition 6.2]{CCIT09}:
\begin{equation*}
f_j(\tau)=(-1)^{m-j}e_{m-j}(\kappa_0,...,\kappa_{m-1}), \quad j=1,...,m-1,
\end{equation*}
where $e_j$ is the $j$-th elementary symmetric polynomial in $m$ variables, and
\begin{equation}\label{eqn:kappas}
\kappa_k(\tau_1,...,\tau_{m-1})=\zeta^{2k+1}\prod_{r=1}^{m-1}\exp\left(\frac{1}{m}\zeta^{(2k+1)r}\tau_r \right), \quad  \zeta:=\exp(\pi\sqrt{-1}/m).
\end{equation}
Using these calculations, the SYZ mirror of $[\C^2/\Z_m]$ can be written in a nice form as follows.
Recall that the mirror curve is given by $uv = 1 + z^m + \sum_{j=1}^{m-1} (\tau_j + \delta_{\nu_j}(\tau)) z^j$.
As $\tau_j + \delta_{\nu_j}(\tau) = f_j(\tau) = (-1)^{m-j}e_{m-j}(\kappa_0,...,\kappa_{m-1})$ and it is easy to check that $1 = (-1)^{m} \kappa_0 \cdots \kappa_{m-1}$, the SYZ mirror of $[\C^2/\Z_m]$ is given by
\begin{equation}\label{eq:ammirf}
 uv = \prod_{j=0}^{m-1} ( z - \kappa_j).
\end{equation}

For the crepant resolution $Y$ of $X = \cpx^2 / \Z_m$, its genus 0 open GW invariants have been computed in \cite{LLW_surfaces}.  The result can be stated as follows.  Let $D_0, \ldots, D_m$ be the toric prime divisors corresponding to the primitive generators $(0,1), \ldots, (m,1)$ of the fan, $\beta_1,\ldots,\beta_m$ be the corresponding basic disks, and $q_i$ for $i = 1,\ldots,m-1$ be the K\"ahler parameters corresponding to the $(-2)$-curves $D_i$.  It turns out that the generating functions of genus 0 open GW invariants
$$q_{j-1} q_{j-2}^2 \ldots q_1^{j-1} (1 + \delta_{j}(q)) = q_{j-1} q_{j-2}^2 \ldots q_1^{j-1} \left(\sum_{\alpha} n_{\beta_j + \alpha} q^{\alpha}\right)$$
are equal to the coefficients of $z^j$ of the polynomial
$ (1 + z)(1 + q_1 z)(1+q_1q_2z)\ldots (1+q_1\ldots q_{m-1}z).$

\noindent{(2)} $\bX=[\cpx^3/\Z_{2g+1}]$ (Example (2) in Section \ref{sect:eg_SYZ}). In this case $[\cpx^3/\Z_{2g+1}]$ is obtained as the quotient orbifold of $\cpx^3$ by the $\Z_{2g+1}$-action with weights $(1,1,2g-1)$. The standard $(\cpx^*)^3$-action on $\cpx^3$ commutes with this $\Z_{2g+1}$-action and induces a $(\cpx^*)^3$-action on the quotient $[\cpx^3/\Z_{2g+1}]$.

There is an alternative route to derive the mirror map of $[\cpx^3/\Z_{2g+1}]$ as follows. The $J$-function of $(\cpx^*)^3$-equivariant GW theory of $[\cpx^3/\Z_{2g+1}]$ coincides with a suitable {\em twisted} $J$-function of the orbifold $B\Z_{2g+1}$, considered in \cite{orbQRR} and \cite{CCIT09}. The $J$-function of $B\Z_{2g+1}$ has been computed in \cite{Jarvis-Kimura_BG} (see also \cite[Proposition 6.1]{CCIT09}:
\begin{equation*}
J^{B\Z_{2g+1}}(y,z)=\sum_{k_0,...,k_{2g}\geq 0}\frac{1}{z^{k_0+...+k_{2g}}}\frac{y_0^{k_0}...y_{2g}^{k_{2g}}}{k_0!...k_{2g}!}{\bf 1}_{\langle \sum_{i=0}^{2g} i\frac{k_i}{2g+1}\rangle}.
\end{equation*}
The twisted GW theory we need is the GW theory of $B\Z_{2g+1}$ twisted by the inverse $(\cpx^*)^3$-equivariant Euler class and the vector bundle $L_1\oplus L_1\oplus L_{2g-1}$, where $L_k$ is the line bundle on $B\Z_{2g+1}$ defined by the $1$-dimensional representation $\cpx_k$ of $\Z_{2g+1}$ on which $1\in \Z_{2g+1}$ acts with eigenvalue $\exp(\frac{2\pi\sqrt{-1}k}{2g+1})$. The generalities of twisted GW theory are developed in \cite{orbQRR}. The $J$-function of the twisted GW theory can be computed by applying \cite[Theorem 4.8]{CCIT09}:
\begin{equation*}
I^{tw}(y,z)= \sum_{k_0,...,k_{2g}\geq 0}\frac{M_{1,k}M_{2,k}M_{3,k}}{z^{k_0+...+k_{2g}}}\frac{y_0^{k_0}...y_{2g}^{k_{2g}}}{k_0!...k_{2g}!}{\bf 1}_{\langle \sum_{i=0}^{2g} i\frac{k_i}{2g+1}\rangle},
\end{equation*}
where
\begin{equation*}
\begin{split}
M_{1,k} & = \prod_{m=0}^{\lfloor b(k)\rfloor-1}\left(\lambda_1-\left(\langle b(k)\rangle+m \right)z\right),\
M_{2,k} = \prod_{m=0}^{\lfloor b(k)\rfloor-1}\left(\lambda_2-\left(\langle b(k)\rangle+m \right)z\right),\\
M_{3,k} & = \prod_{N(k)+1\leq m\leq 0}\left(\lambda_3+\left(m-\left(1-\langle c(k)\rangle\right)\right)z\right),\\
\end{split}
\end{equation*}
and
\begin{equation*}
\begin{split}
b(k):= \sum_{i=1}^{2g}\frac{ik_i}{2g+1}, \ c(k):=-\sum_{i=1}^{2g}\frac{ik_i}{2g+1}(2g-1), \
N(k):=1+\sum_{i=1}^{2g}\lfloor \frac{i(2g-1)}{2g+1}\rfloor k_i+\lfloor c(k)\rfloor.
\end{split}
\end{equation*}
Here $\lambda_k, k=1,2,3$ is the weight of the $k$-th factor of $(\cpx^*)^3$ acting on the $k$-th factor of $\cpx^3$.
By \cite[Theorem 4.8]{CCIT09} it is then straightforward to extract the $J$-function of $[\cpx^3/\Z_{2g+1}]$, the mirror map, and generating functions of orbi-disk invariants from $I^{tw}(y,z)$. We leave the details to the readers.

\noindent{(3)} $\bX = [\cpx^n / \Z_n]$ (Example (3) in Section \ref{sect:eg_SYZ}).  In this case there is only one twisted sector $\nu$ of age one. Let $\tau$ be the corresponding orbifold parameter.  The toric mirror map has been computed explicitly in \cite{CCLT12}:
$$ \tau = g(y) = \sum_{k=0}^{\infty} \frac{((-\frac{1}{n}) \ldots (1-k-\frac{1}{n}))^n}{(kn+1)!} y^{kn+1}.$$
Then Theorem \ref{thm:orbi_disk_gen_function} tells us that the generating function
$\tau + \delta_\nu(\tau) = \sum_{k \geq 1} \frac{\tau^{k}}{k!} n_{1,k,\beta_\nu} ([\pt]_L; (\mathbf{1}_\nu)^{k})$
of genus 0 open orbifold GW invariants is equal to the inverse series of $g(y)$.

The total space of the canonical line bundle of $\proj^{n-1}$, $Y = K_{\proj^{n-1}}$, is a crepant resolution of $X = \cpx^n / \Z_n$.  Its cohomology is generated by the line class $l$ of $\proj^{n-1}$; let $q$ denote the corresponding K\"ahler parameter.  Let $\beta_0$ be the basic disk class corresponding to the zero-section.  The generating function of genus 0 open GW invariants
$ 1 + \delta (q) = \sum_{k \geq 0} n_{\beta_0 + kl} q^{k}$
is equal to $\exp g(y)$, where
$$ g(y) = \sum_{k>0} (-1)^{nk} \frac{(nk-1)!}{(k!)^n} y^k,$$
and $q$ and $y$ are related by the mirror map $q = y \exp (-n g(y))$.

\noindent{(4)} $\bX = K_{\mathbb{F}_2}$ (Example \ref{KF2}).  $\bX$ is a smooth toric manifold whose fan has primitive generators $\bb_0 = (0,0,1)$, $\bb_1 = (-1,1,1)$, $\bb_2 = (0,1,1)$, $\bb_3=(1,1,1)$ and $\bb_4=(0,-1,1)$.  Note that the Hirzebruch surface $\mathbb{F}_2$ is not Fano (but semi-Fano).  We remark that $\bX = K_{\mathbb{F}_2}$ is a new example whose open GW invariants were not computed in previous works.

The primitive generators which are not vertices of $\mathcal{P}$ (the convex hull of $\bb_1$, $\bb_3$ and $\bb_4$) are $\bb_0$ and $\bb_2$.  Hence
$n_{\beta_i + \alpha} = 0$
for $i=1,3,4$ and $\alpha \neq 0$.  Also $n_{\beta_i} = 1$ for $i=0,\ldots,4$.  Only the open GW invariants $n_{\beta_0 + \alpha}$ and $n_{\beta_2 + \alpha}$ for $\alpha \neq 0$ can be non-trivial.

Take $p_1 = D_0, p_2 = D_2$ to be the basis of $H^2(\bX, \rat)$, and let $C_1,C_2$ be the dual basis.  Denote the $(-2)$ exceptional curve class of $\mathbb{F}_2$ by $e$, and denote the fiber curve class of $\mathbb{F}_2$ by $f$.  $e$ and $f$ form a basis of $H_2(\bX; \Z)$.  By computing the intersection numbers of $e$ and $f$ with $p_1$ and $p_2$, we obtain the relations
$f = C_2 - 2 C_1$ and
$e = -2 C_2$.

The K\"ahler parameters of $C_1$ and $C_2$ are denoted as $q_1$ and $q_2$ respectively, while that of $e$ and $f$ are denoted as $q^e$ and $q^f$ respectively.  we have $q^f = q_2 q_1^{-2}, q^e = q_2^{-2}$.
The corresponding parameters of the complex moduli of the mirror are denoted by $(y_1,y_2)$, and we have $y^f = y_2 y_1^{-2}, y^e = y_2^{-2}$.
The mirror map is given by
$q_1 = y_1 \exp (A_1^{\bX}(y_1,y_2)),\ q_2 = y_2 \exp (A_2^{\bX}(y_1,y_2))$,
where
$$A^\bX_j(y) := \sum_{d\in \Omega^\bX_j}y^d \frac{(-1)^{-\langle D_j,d\rangle-1}(-\langle D_j,d\rangle-1)!}{\prod_{i\neq j}\langle D_i,d\rangle!}$$
by Equation \eqref{eqn:toric_mirror_map_revised1}, and $\Omega^\bX_j := \{d\in \mathbb{K}_\mathrm{eff} \mid \langle D_j,d\rangle \in \Z_{<0}\textrm{ and } \langle D_i,d\rangle \in \Z_{\geq 0}\ \forall i\neq j\}$.

First consider $A_2^{\bX}$.  For $C = a e + b f$ where $a, b \in \Z$, $C \cdot D_2 = -2 a + b < 0$ and $C \cdot D_0 = -2b \geq 0$ imply that $b = 0$ and $a \geq 0$.  Also $C \cdot D_i \geq 0$ for $i\not= 2$.  Hence $\Omega^\bX_2 = \{k e: k \in \nat\}$, and
$$A^\bX_2(y_1,y_2) = \sum_{k=1}^\infty y^{ke} \frac{(-1)^{2k-1}(2k-1)!}{(k!)^2} = -\log 2 + \log (1 + \sqrt{1-4 y^e}). $$
Thus
$$ q^e = q_2^{-2} = y^e \exp (-2 A_2^{\bX}(y_1,y_2)) = \frac{4 y^e}{(1 + \sqrt{1-4 y^e})^2}. $$
Taking the inverse, we obtain
$$ y^e = \frac{q^e}{(1+q^e)^2}, \quad y_2 = y^{-e/2} = (1+q^e) q_2.  $$
Comparing with $y_2 = q_2 \exp(-A_2^{\bX}(y_1,y_2))$, this implies
$\exp(-A_2^{\bX}(y_1,y_2)) = 1+q^e $
under the mirror map.  By Theorem \ref{thm:sm_disk_gen_function}, we have
$ \sum_{\alpha} n_{\beta_2 + \alpha} q^\alpha = 1 + q^e. $
Thus $n_{\beta_2 + \alpha} =1$ when $\alpha = 0, e$, and zero for all other classes $\alpha$.

The hypergeometric series $A^\bX_2$ above also gives the mirror map of  $\mathbb{F}_2$.  This is the analytic reason why the open GW invariants above are the same as those of $\mathbb{F}_2$:
$ n^{\bX}_{\beta_2 + \alpha} = n^{\mathbb{F}_2}_{\beta_2 + \alpha}. $
It is geometrically intuitive: the bubbling contributions of the curve class $e$ to $\beta_2$ in $\mathbb{F}_2$ are the same as that in $K_{\mathbb{F}_2}$, because $D_2$ in $K_{\mathbb{F}_2}$ is just the product of the corresponding divisor in $\mathbb{F}_2$ with the complex line $\C$.

Now consider $A_1^{\bX}$.  For $C = a e + b f$ where $a, b \in \Z$,
$C \cdot D_2 = -2 a + b \geq 0, C \cdot D_0 = -2b < 0$
imply that $b \geq 2a > 0$.  Also $C \cdot D_i \geq 0$ for $i\not= 2$.  Hence $\Omega^\bX_1 = \{kf + a(e+2f):  a \in \nat, k \in \Z_{\geq 0}\}$, and
$$A^\bX_1(y_1,y_2) = \sum_{a=1}^\infty \sum_{k=0}^{\infty} y^{kf + a(e+2f)} \frac{(-1)^{2(2a+k)-1}(2(2a+k)-1)!}{(k!)(a!)^2(2a+k)!}. $$
By Theorem \ref{thm:sm_disk_gen_function}, this gives
$ \sum_{\alpha} n_{\beta_0 + \alpha} q^\alpha = \exp(-A_1^{\bX}(y_1(q),y_2(q))) $
where the mirror map $q(y)$ is
\begin{align*}
q^f = y^f \exp (-2 A_1(y^e,y^f) + A_2(y^e)), \ q^e = y^e \exp (-2 A_2(y^e)).
\end{align*}

The following table can be obtained by inverting the mirror map using computers:\\

\begin{center}
	\begin{tabular}{|c|c|c|c|c|c|c|c|}
		\hline
		$n_{\beta_0 + ae+bf}$ & $a=0$ & $a=1$ & $a=2$ & $a=3$ & $a=4$ & $a=5$ & $a=6$ \\
		\hline
		$b=0$ & $1$ & $0$ & $0$ & $0$ & $0$ & $0$ & $0$ \\
		\hline
		$b=1$ & $0$ & $0$ & $0$ & $0$ & $0$ & $0$ & $0$ \\
		\hline
		$b=2$ & $0$ & $-3$ & $0$ & $0$ & $0$ & $0$ & $0$ \\
		\hline
		$b=3$ & $0$ & $-20$ & $-20$ & $0$ & $0$ & $0$ & $0$ \\
		\hline
		$b=4$ & $0$ & $-105$ & $-294$ & $-105$ & $0$ & $0$ & $0$ \\
		\hline
		$b=5$ & $0$ & $-504$ & $-2808$ & $-2808$ & $-504$ & $0$ & $0$ \\
		\hline
		$b=6$ & $0$ & $-2310$ & $-21835$ & $-42867$ & $-21835$ & $-2310$ & 0 \\
		\hline		
	\end{tabular}
\end{center}

\section{Open mirror theorems}\label{sec:open_mirror_thms}

In this section we define the SYZ map, and prove an open mirror theorem which says that the SYZ map coincides with the inverse of the toric mirror map. For toric CY manifolds, this theorem implies that the inverse of a mirror map defined using period integrals (so this is {\em not} the toric mirror map) can be expressed explicitly in terms of generating functions of genus 0 open GW invariants defined by Fukaya-Oh-Ohta-Ono \cite{FOOO1}. This confirms in the affirmative a conjecture of Gross-Siebert \cite[Conjecture 0.2]{gross07}, which was later made precise in \cite[Conjecture 1.1]{CLL} in the toric CY case.

\subsection{The SYZ map} \label{sec:two_mirror_maps}

\subsubsection{K\"ahler moduli} \label{sec:kahler_moduli}
As before, $\bX$ is a toric CY orbifold as in Setting \ref{setting:toricCY}. Let $\widetilde{C}_\bX\subset \mathbb{L}^\vee\otimes \R$ be the extended K\"ahler cone of $\bX$ as defined in Section \ref{sec:toric_mirror_theorem}. Recall that there is a splitting $\widetilde{C}_\bX = C_\bX+\sum_{j=m}^{m'-1} \mathbb{R}_{>0}D_j\subset \mathbb{L}^\vee\otimes \R$, where $C_\bX \subset H^2(\bX;\R)$ is the K\"ahler cone of $\bX$. We define the complexified (extended) K\"ahler moduli space of $\bX$ as
$$\mathcal{M}_K(\bX) := \left(\widetilde{C}_\bX+\sqrt{-1}H^2(\bX,\R)\right)/H^2(\bX,\Z) + \sum_{j=m}^{m'-1} \C D_j.$$
Elements of $\mathcal{M}_K(\bX)$ are represented by complexified (extended) K\"ahler class
$\omega^\C = \omega + \sqrt{-1}B+ \sum_{j=m}^{m'-1}\tau_j D_j,$
where $\omega\in C_\bX$, $B\in H^2(\bX,\R)$ and $\tau_j \in \C$.

We identify $\mathcal{M}_K(\bX)$ with $(\Delta^*)^{r'} \times \C^{r-r'}$, where $\Delta^*$ is the punctured unit disk,
via the coordinates $q_a = \exp \left(-2\pi\int_{\gamma_a} \left(\omega + \sqrt{-1}B \right)\right)$ for $a = 1, \ldots, r'$ and
$\tau_j \in \C,\quad j = m, \ldots, m'-1$,
where $\{\gamma_1, \ldots, \gamma_{r'}\}$ is the integral basis of $H_2(\bX;\Z)$ we chose in Section \ref{sec:toric_mirror_theorem}.
A partial compactification of $\mathcal{M}_K(\bX)$ is given by $(\Delta^*)^{r'}\times \C^{r-r'} \subset \Delta^{r'}\times \C^{r-r'}$.

\subsubsection{Complex moduli}

On the mirror side, recall that
$$\mathcal{P} \cap N = \{\bb_0,\ldots,\bb_{m-1},\bb_m,\ldots,\bb_{m'-1}\}$$
and $\mathcal{P}$ is contained in the hyperplane $\{v\in N_\R \mid \pairing{(0,1)}{v} = 1\}$. Denote by $L(\mathcal{P}) \simeq \C^{m'}$ the space of Laurent polynomials $G\in \C[z_1^{\pm1},\ldots,z_{n-1}^{\pm1}]$ of the form $\sum_{i=0}^{m'-1} C_i z^{\bb_i}$, i.e. those with Newton polytope $\mathcal{P}$. Let $\proj_\mathcal{P}$ be the projective toric variety defined by the normal fan of $\mathcal{P}$. In Batyrev \cite{B93}, a Laurent polynomial $G\in L(\mathcal{P})$ is defined to be {\em $\mathcal{P}$-regular} if the intersection of the closure $\bar{Z}_f\subset\proj_\mathcal{P}$, of the associated affine hypersurface $Z_f := \{(z_1,\ldots,z_{n-1})\in(\C^\times)^{n-1} \mid f(z_1,\ldots,z_{n-1})=0\}$ in $(\C^\times)^{n-1}$, with every torus orbit $O\subset\proj_\mathcal{P}$ is a smooth subvariety of codimension 1 in $O$. Denote by $L_\textrm{reg}(\mathcal{P})$ the space of all $\mathcal{P}$-regular Laurent polynomials.

Following Batyrev \cite{B93} and Konishi-Minabe \cite{konishi09}, we define the complex moduli space $\mathcal{M}_\C(\check{\bX})$ of the mirror $\check{\bX}$ to be the GIT quotient of $L_\textrm{reg}(\mathcal{P})$ by a natural $(\C^\times)^n$-action, which is nonempty and has complex dimension $r = m'-n$ \cite{B93}. It parametrizes a family of non-compact CY manifolds $\{\check{\bX}_y\}$:
\begin{equation}\label{eqn:CY_mirror_family}
\check{\bX}_y := \left\{(u, v, z_1, \ldots, z_{n-1})\in \C^2\times (\C^\times)^{n-1} \mid uv = G_y(z_1,\ldots,z_{n-1}) \right\},
\end{equation}
where
$$G_y(z_1,\ldots,z_{n-1}) = \sum_{i=0}^{m-1} \check{C}_i z^{\bb_i} + \sum_{j=m}^{m'-1} \check{C}_{\nu_j} z^{\nu_j},$$
and the coefficients $\check{C}_i, \check{C}_{\nu_j} \in \C$ are subject to the following constraints:
\begin{equation*}
\begin{split}
\prod_{i=0}^{m-1} \check{C}_i^{Q_{ia}} & = y_a, \quad a = 1,\ldots,r',\\
\prod_{i=0}^{m-1} \check{C}_i^{Q_{ia}}\prod_{j=m}^{m'-1} \check{C}_{\nu_j}^{Q_{ja}} & = y_a, \quad a = r'+1,\ldots,r.
\end{split}
\end{equation*}
Note that the non-compact CY manifolds in the family \eqref{eqn:CY_mirror_family} may become singular and develop orbifold singularities when some of the $y_a$'s go to zero.

To define period integrals, let $\check{\Omega}_y$ be the holomorphic volume form on $\check{\bX}_y$ defined by
\begin{equation*}
\check{\Omega}_y = \mathrm{Res}\left(\frac{1}{uv-G_y(z_1,\ldots,z_{n-1})} d\log z_0\wedge \cdots \wedge d \log z_{n-1} \wedge du\wedge dv \right),
\end{equation*}
where $G_y(z_1,\ldots,z_{n-1}) := \sum_{i=0}^{m-1} \check{C}_i z^{\bb_i} + \sum_{j=m}^{m'-1} \check{C}_{\nu_j} z^{\nu_j}$.
\subsubsection{Two mirror maps}

\begin{defn}\label{defn:SYZ_map}
We define the {\em SYZ map} as follows:
\begin{equation}\label{eqn:SYZ_map}
\begin{split}
\mathcal{F}^{\mathrm{SYZ}} & :\mathcal{M}_K(\bX)\to \mathcal{M}_\C(\check{\bX}), \quad y \mapsto \mathcal{F}^{\mathrm{SYZ}}(q, \tau)\\
y_a & := q_a \prod_{i=0}^{m-1}\left(1+\delta_i\right)^{Q_{ia}}, \quad a = 1,\ldots,r',\\
y_a & := \prod_{i=0}^{m-1}\left(1+\delta_i\right)^{Q_{ia}}
\prod_{j=m}^{m'-1}\left(q^{-D^\vee_j} \left(\tau_{\nu_j}+\delta_{\nu_j}\right)\right)^{Q_{ja}}, \quad a = r'+1,\ldots,r,
\end{split}
\end{equation}
where $q^{-D^\vee_j} := \prod_{a=1}^{r'} q_a^{\langle p_a, D^\vee_j\rangle}$, and $1+\delta_i$ and $\tau_{\nu_j}+\delta_{\nu_j}$ are the generating functions of genus 0 open orbifold GW invariants in $\bX$ relative to a Lagrangian torus fiber of a Gross fibration $\mu:\bX \to B$, defined in \eqref{eqn:generating_functions_open_GW}.
\end{defn}
By Theorems \ref{thm:sm_disk_gen_function} and \ref{thm:orbi_disk_gen_function}, we have
\begin{equation}\label{eqn:formula_for_generating_function_sm_disk_F_r}
1 + \delta_i = \exp\left(-A^\bX_i(y(q,\tau))\right), \text{ for } i = 0,1,\ldots,m-1,
\end{equation}
\begin{equation}\label{eqn:formula_for_generating_function_orb_disk_F_r}
\tau_{\nu_j} + \delta_{\nu_j} = y^{D^\vee_j} \exp\left(-\sum_{i\notin I_j}c_{ji}A^\bX_i(y(q,\tau)) \right), \text{ for } j = m,m+1,\ldots,m'-1.
\end{equation}

On the other hand, recall that the toric mirror map \eqref{eqn:toric_mirror_map_X} for $\bX$ is given by
\begin{equation*}
\begin{split}
\mathcal{F}^{\mathrm{mirror}} & :\mathcal{M}_\C(\check{\bX})\to \mathcal{M}_K(\bX), \quad (q, \tau) \mapsto \mathcal{F}^{\mathrm{mirror}}(y)\\
q_a & = y_a \prod_{j=0}^{m-1} \exp\left(A^\bX_j(y)\right)^{Q_{ja}}, \quad a = 1,\ldots,r',\quad
\tau_{\bb_j}  = A^\bX_j(y), \quad j=m,\ldots,m'-1.
\end{split}
\end{equation*}

\subsection{Open mirror theorems}

\subsubsection{Proof of Theorem \ref{thm:open_mirror}}
Recall that the toric mirror map $\mathcal{F}^{\mathrm{mirror}}$ is a local isomorphism near $y = 0$, so we can consider its inverse $\left(\mathcal{F}^{\mathrm{mirror}}\right)^{-1}$ given by $y = y(q,\tau)$ near $(q,\tau) = 0$.

For $a = 1, \ldots, r'$, we have, by the formula \eqref{eqn:formula_for_generating_function_sm_disk_F_r},
\begin{align*}
\log q_a + \sum_{i=0}^{m-1} Q_{ia}(1 + \delta_i) = \log q_a - \sum_{i=0}^{m-1} Q_{ia} A^\bX_i(y(q,\tau)) = \log y_a.
\end{align*}

For $a = r'+1, \ldots, r$, we have, by the formulas \eqref{eqn:formula_for_generating_function_sm_disk_F_r} and \eqref{eqn:formula_for_generating_function_orb_disk_F_r},
\begin{equation}\label{eqn:y_a}
\begin{split}
& \sum_{j=m}^{m'-1} Q_{ja} \left(\log q^{-D^\vee_j} + \log(\tau_{\nu_j} + \delta_{\nu_j})\right)\\
= & \sum_{j=m}^{m'-1} Q_{ja} \left(-\sum_{b=1}^{r'}\langle p_b, D^\vee_j\rangle \log q_b + \sum_{b=1}^r\langle p_b, D^\vee_j\rangle \log y_b
- \sum_{i\notin I_j} c_{ji}A^\bX_i(y(q,\tau))\right)\\
= & \sum_{b=r'+1}^r\left(\sum_{j=m}^{m'-1} Q_{ja}\langle p_b, D^\vee_j\rangle \right)\log y_b
+ \sum_{j=m}^{m'-1} Q_{ja}\left(\sum_{b=1}^{r'} \langle p_b, D^\vee_j\rangle \log \left(y_b q_b^{-1}\right)\right) - \sum_{j=m}^{m'-1} Q_{ja} \left(\sum_{i\notin I_j} c_{ji}A^\bX_i(y(q,\tau))\right).
\end{split}
\end{equation}

Now, the definition of $D^\vee_j$ implies that $\langle D_i, D^\vee_j\rangle = \delta_{ij}$ for $m\leq i,j\leq m'-1$. Since $D_i = \sum_{a=1}^r Q_{ia}p_a$ and $Q_{ia} = 0$ for $1\leq a\leq r'$ and $m\leq i\leq m'-1$, we have $\sum_{a=r'+1}^r Q_{ia}\langle p_a, D^\vee_j\rangle = \delta_{ij}$ for $m\leq i,j\leq m'-1$. This shows that the $(r-r')\times (r-r')$ square matrices $(Q_{ia})$ and $(\langle p_a, D^\vee_i\rangle)$ (where $m\leq i \leq m'-1$ and $r'+1\leq a\leq r$) are inverse to each other (note that $r-r'=m'-m$), so
$\sum_{j=m}^{m'-1} Q_{ja}\langle p_b, D^\vee_j\rangle = \delta_{ab}$
for $r'+1\leq a,b\leq r$. Hence the first term of the last expression in \eqref{eqn:y_a} is precisely $\log y_a$.

On the other hand, we have
\begin{align*}
\sum_{b=1}^{r'} \langle p_b, D^\vee_j\rangle \log \left(y_b q_b^{-1}\right)
 = \sum_{b=1}^{r'} \langle p_b, D^\vee_j\rangle \left(-\sum_{k=0}^{m-1} Q_{kb}A^\bX_k(y)\right)
 = -\sum_{k=0}^{m-1} \left( \sum_{b=1}^{r'} Q_{kb} \langle p_b, D^\vee_j\rangle\right) A^\bX_k(y),
\end{align*}
and using the above formula $\sum_{j=m}^{m'-1} Q_{ja}\langle p_b, D^\vee_j\rangle = \delta_{ab}$ again, we can write
\begin{align*}
\sum_{k=0}^{m-1} Q_{ka} \log(1 + \delta_k) & = - \sum_{k=0}^{m-1} Q_{ka} A^\bX_k(y) = - \sum_{k=0}^{m-1} \left(\sum_{b=r'+1}^r Q_{kb}\left(\sum_{j=m}^{m'-1}Q_{ja} \langle p_b, D^\vee_j\rangle\right)\right) A^\bX_k(y)\\
& = - \sum_{j=m}^{m'-1} Q_{ja} \left(\sum_{b=r'+1}^r \langle p_b, D^\vee_j\rangle \left(\sum_{k=0}^{m-1} Q_{kb}A^\bX_k(y) \right) \right)
\end{align*}

We compute the sum
\begin{align*}
& \sum_{k=0}^{m-1} Q_{ka} \log(1 + \delta_k)
+ \sum_{j=m}^{m'-1} Q_{ja} \left(\sum_{b=1}^{r'} \langle p_b, D^\vee_j\rangle \log \left(y_b q_b^{-1}\right)\right)\\
= & - \sum_{j=m}^{m'-1} Q_{ja} \left(\sum_{b=r'+1}^r \langle p_b, D^\vee_j\rangle \left(\sum_{k=0}^{m-1} Q_{kb}A^\bX_k(y) \right) \right)
- \sum_{j=m}^{m'-1} Q_{ja} \left(\sum_{b=1}^{r'} \langle p_b, D^\vee_j\rangle \left(\sum_{k=0}^{m-1} Q_{kb}A^\bX_k(y) \right) \right)\\
= & - \sum_{j=m}^{m'-1} Q_{ja} \left(\sum_{k=0}^{m-1} \langle D_k, D^\vee_j\rangle A^\bX_k(y)\right)
=  \sum_{j=m}^{m'-1} Q_{ja} \left(\sum_{k\notin I_j} c_{jk} A^\bX_k(y)\right),
\end{align*}
which cancels with the third term of the last expression in \eqref{eqn:y_a}. Hence we conclude that
$$\sum_{i=0}^{m-1} Q_{ia} \log(1 + \delta_i) + \sum_{j=m}^{m'-1} Q_{ja} \left(\log q^{-D^\vee_j} + \log(\tau_{\nu_j} + \delta_{\nu_j})\right) = \log y_a$$
for $a = r'+1, \ldots, r$. This proves the theorem.

\subsubsection{Connection with period integrals}
Traditionally, mirror maps are defined in terms of period integrals, which are integrals $\int_\Gamma \check{\Omega}_y$ of the holomorphic volume form $\check{\Omega}_y$ over middle-dimensional cycles $\Gamma \in H_n(\check{\bX}_y;\C)$ (see, e.g. \cite[Chapter 6]{Cox-Katz}). Theorem \ref{thm:period_mirror} shows that the inverse of such a mirror map also coincides with the SYZ map.
When $\bX$ is a toric CY {\em manifold}, we do not have extra vectors so that $m'=m$ and $r=r'$, and there are no twisted sectors insertions in the invariants $n^\bX_{1,l,\beta_i+\alpha}([\mathrm{pt}]_L)$. Theorem \ref{thm:period_mirror} in this case specializes to Corollary \ref{cor:period_mirror_manifold}.

Theorem~\ref{thm:period_mirror} and Corollary~\ref{cor:period_mirror_manifold} give an enumerative meaning to period integrals, which was first envisioned by Gross and Siebert in \cite[Conjecture 0.2 and Remark 5.1]{gross07} where they conjectured that period integrals of the mirror can be interpreted as (virtual) counting of {\em tropical} disks (instead of holomorphic disks) in the base of an SYZ fibration for a {\em compact} CY manifold; in \cite[Example 5.2]{gross08}, they also observed a precise relation between the so-called {\em slab functions}, which appeared in their program, and period computations for the toric CY 3-fold $K_{\proj^2}$ in \cite{graber-zaslow01}. A more precise relation in the case of toric CY manifolds was later formulated in \cite[Conjecture 1.1]{CLL}.\footnote{It was wrongly asserted that the cycles $\Gamma_1,\ldots,\Gamma_r$ form a basis of $H_n(\check{\bX}_y;\C)$ in \cite[Conjecture 1.1]{CLL} while they should just be linearly independent cycles; see \cite[Conjecture 2]{CLT11} for the correct version.}

We point out that Corollary \ref{cor:period_mirror_manifold} is weaker than \cite[Conjecture 1.1]{CLL} because the cycles $\Gamma_1,\ldots,\Gamma_r$ are allowed to have complex coefficients instead of being {\em integral}. In the special case where $\bX$ is the total space of the canonical bundle over a compact toric Fano manifold, Corollary \ref{cor:period_mirror_manifold} was proven in \cite{CLT11}. As discussed in \cite[Section 5.2]{CLT11}, to enhance Corollary \ref{cor:period_mirror_manifold} to \cite[Conjecture 1.1]{CLL}, one needs to study the monodromy of $H_n(\check{\bX}_y;\Z)$ around the limit points in the complex moduli space $\mathcal{M}_\C(\check{\bX})$.

Theorem \ref{thm:period_mirror} is essentially a consequence of Theorem \ref{thm:open_mirror} and the analysis of the relationships between period integrals over $n$-cycles of the mirror and GKZ hypergeometric systems in \cite[Section 4]{CLT11}. Recall that the {\em Gel'fand-Kapranov-Zelevinsky (GKZ) system} \cite{GKZ89, GKZ90} of differential equations (also called $A$-hypergeometric system) associated to $\bX$, or to the set of lattice points $\Sigma(1)=\{\bb_0,\bb_1,\ldots,\bb_{m-1}\}$, is the following system of partial differential equations on functions $\Phi(\check{C})$ of $\check{C}=(\check{C}_0,\check{C}_1,\ldots,\check{C}_{m-1})\in \C^m$:
\begin{equation}\label{eq:GKZ}
\begin{split}
\left(\sum_{i=0}^{m-1} \bb_i \check{C}_i\partial_i\right)\Phi(\check{C}) & = 0,\\
\left(\prod_{i:\langle D_i, d\rangle>0} \partial_i^{\langle D_i, d\rangle} -
\prod_{i:\langle D_i, d\rangle<0} \partial_i^{-\langle D_i, d\rangle}\right)\Phi(\check{C}) & = 0,\quad d \in \mathbb{L},
\end{split}
\end{equation}
where $\partial_i = \partial/\partial \check{C}_i$ for $i=0,1,\ldots,m-1$. Note that the first equation in \eqref{eq:GKZ} consists of $n$ equations, so there are $n+r = m$ equations in total. By \cite[Proposition 14]{CLT11}, the period integrals
$$\int_\Gamma \check{\Omega}_y,\quad \Gamma\in H_n(\check{\bX}_y;\Z),$$
provide a $\C$-basis of solutions to the GKZ hypergeometric system \eqref{eq:GKZ}; see also \cite{hosono06} and \cite[Corollary A.16]{konishi09}.
Now Theorem \ref{thm:period_mirror} follows from the following
\begin{lemma}\label{lem:GKZ_mirror_map}
The components of the toric mirror map \eqref{eqn:toric_mirror_map_X} of a toric CY orbifold $\bX$,
\begin{equation*}
\begin{split}
\log q_a & = \log y_a + \sum_{j=0}^{m-1} Q_{ja}A^\bX_j(y), \quad a = 1,\ldots,r',\\
\tau_{\bb_j} & = A^\bX_j(y), \quad j=m,\ldots,m'-1,
\end{split}
\end{equation*}
are solutions to the GKZ hypergeometric system \eqref{eq:GKZ}.
\end{lemma}
\begin{proof}
The proof is more or less the same as that of \cite[Theorem 12]{CLT11}, which in turn is basically a corollary of a result of Iritani \cite[Lemma 4.6]{iritani09}. We first fix $i_0 \in \{0,\ldots,m'-1\}$, and consider the corresponding toric compactification $\bar{\bX}$. For $i \in \{0,\ldots,m-1\} \cup \{\infty\}$, set
$$\mathcal{D}_i = \sum_{a \in \{1,\ldots,r\} \cup \{\infty\}} Q_{ia} y_a\frac{\partial}{\partial y_a},$$
and, for $d \in \bar{\mathbb{L}}$, we define a differential operator
\begin{equation*}
\Box_d := \prod_{i:\langle D_i, d\rangle>0} \prod_{k=0}^{\langle D_i, d\rangle-1}(\mathcal{D}_i-k) -
y^d \prod_{i:\langle D_i, d\rangle<0} \prod_{k=0}^{-\langle D_i, d\rangle-1}(\mathcal{D}_i-k).
\end{equation*}

Now \cite[Lemma 4.6]{iritani09} says that the $I$-function $I_{\bar{\bX}}(y,z)$ satisfy the following system of GKZ-type differential equations:
\begin{equation}\label{eqn:GKZ_type}
\Box_d \Psi = 0,\quad d \in \bar{\mathbb{L}}.
\end{equation}
In particular, the components
\begin{equation*}
\begin{split}
\log q_a & = \log y_a + \sum_{j=0}^{m-1} Q_{ja}A^\bX_j(y), \quad a = 1,\ldots,r',\\
\tau_{\bb_j} & = A^\bX_j(y), \quad j=m,\ldots,m'-1,
\end{split}
\end{equation*}
of the toric mirror map of $\bX$, which are contained in the toric mirror map \eqref{eqn:toric_mirror_map_barX_1} of $\bar{\bX}$, are solutions to the above system.

Hence, it suffices to show that solutions to the above system also satisfy the GKZ hypergeometric system \eqref{eq:GKZ}. This was shown in the proof of \cite[Theorem 12]{CLT11}, so we will just describe the argument briefly. First of all, we have $\sum_{i=0}^{m'-1} Q_{ia}=0$ for $a=1,\ldots,r$. Together with the fact that $y_a = \prod_{i=0}^{m-1} \check{C}_i^{Q_{ia}}$ for $a=1,\ldots,r$, one can see that the first $n$ equations in \eqref{eq:GKZ} are satisfied by any solution of \eqref{eqn:GKZ_type}. On the other hand, it is not hard to compute, using the fact that $\langle D_\infty, d\rangle=0$ for $d \in \mathbb{L} \oplus 0 \subset \bar{\mathbb{L}}$, that
$$\prod_{i:\langle D_i, d\rangle>0} \partial_i^{\langle D_i, d\rangle}
- \prod_{i:\langle D_i, d\rangle<0} \partial_i^{-\langle D_i, d\rangle}
= \left(\prod_{i:\langle D_i, d\rangle>0}\check{C}_i^{-\langle D_i, d\rangle}\right) \Box_d$$
for $d \in \mathbb{L}$. Hence the other set of equations in \eqref{eq:GKZ} are also satisfied. The lemma follows.
\end{proof}




\section{Application to crepant resolutions}\label{sec:openCRC}

Let $\mathcal{Z}$ be a compact Gorenstein toric orbifold. Suppose the underlying simplicial toric variety $Z$ admits a toric crepant resolution $\widetilde{Z}$. In \cite{CCLT12}, a conjecture on the relationship between genus 0 open GW invariants of $\widetilde{Z}$ and $\mathcal{Z}$ was formulated and studied. In this section we consider the following setting. Let $\bX$ be a toric CY orbifold as in Setting \ref{setting:toricCY}. It is well-known (see e.g. \cite{Fu}) that toric crepant birational maps to the coarse moduli space $X$ of $\bX$ can be obtained from regular subdivisions of the fan $\Sigma$ satisfying certain conditions. More precisely, let $\bX'=\bX_{\Sigma'}$ be the toric orbifold obtained from the fan $\Sigma'$, where $\Sigma'$ is a regular subdivision of $\Sigma$. Then the morphism $X'\to X$ between the coarse moduli spaces is crepant if and only if for each ray of $\Sigma'$ with minimal lattice generator $u$, we have $(\underline{\nu}, u)=1$. We prove the following:
\begin{thm}[Open crepant resolution theorem]\label{thm:openCR_toric_CY}
Let $\bX$ be a toric CY orbifold as in Setting \ref{setting:toricCY}. Let $\bX'$ be a toric orbifold obtained by a regular subdivision of the fan $\Sigma$, such that the natural map $X'\to X$ between the coarse moduli spaces is crepant. Denoted by $(q,\tau)$ and $(Q,\mathcal{T})$ the flat coordinates on the K\"ahler moduli of $\bX$ and $\bX'$ respectively, and $r$ is the dimension of the extended complexified K\"ahler moduli space of $\bX$ (which is equal to that of $\bX'$). Then there exists
\begin{enumerate}
\item $\epsilon>0$;
\item a coordinate change $(Q(q,\tau), \mathcal{T}(q,\tau))$, which is a holomorphic map $(\Delta(\epsilon)-\real_{\leq 0})^r \to (\cpx^\times)^r$, and $\Delta(\epsilon)$ is an open disk of radius $\epsilon$ in the complex plane;
\item a choice of an analytic continuation of the SYZ map $\mathcal{F}^{\mathrm{SYZ}}_{\bX'}(Q,\mathcal{T})$ to the target of the holomorphic map $(Q(q,\tau), \mathcal{T}(q,\tau))$,
\end{enumerate}
such that
\begin{equation*}
\mathcal{F}^{\mathrm{SYZ}}_{\mathcal{X}}(q,\tau) = \mathcal{F}^{\mathrm{SYZ}}_{\bX'}(Q(q,\tau), \mathcal{T}(q,\tau)).
\end{equation*}
\end{thm}

Theorem \ref{thm:openCR_toric_CY} may be interpreted as saying that generating functions of genus 0 open GW invariants of $\bX'$ coincide with those of $\bX$ after analytical continuations and changes of variables. See \cite[Conjecture 1, Theorem 3]{CCLT12} for related statements for compact toric orbifolds.

Our proof of Theorem \ref{thm:openCR_toric_CY} employs the general strategy described in \cite{CCLT12}. Namely we use the open mirror theorem (Theorem \ref{thm:open_mirror}) to relate genus 0 open (orbifold) GW invariants of $\bX$ and $\bX'$ to their toric mirror maps. These toric mirror maps are explicit hypergeometric series and their analytic continuations can be done by using Mellin-Barnes integrals techniques. See Appendix \ref{app:analy_conti}.

\begin{proof}[Proof of Theorem \ref{thm:openCR_toric_CY}]
We adapt the strategy used in \cite{CCLT12} for proving related results for compact toric orbifolds. By Theorem \ref{thm:open_mirror}, we may replace $\mathcal{F}^{\mathrm{SYZ}}$ by $\left(\mathcal{F}^{\mathrm{mirror}}\right)^{-1}$, which are given by the toric mirror maps \eqref{eqn:toric_mirror_map_X}. It suffices to show that an analytical continuation of the toric mirror map exists. The necessary change of variables is given by composing the inverse of the (analytically continued) toric mirror map of $\bX'$ with the toric mirror map of $\bX$.

Now the crepant birational map $X'\to X$ may be decomposed into a sequence of crepant birational maps each of which is obtained by a regular subdivision that introduces only one new ray. If we can construct an analytical continuation of the toric mirror map for each of these simpler crepant birational maps, then we would obtain the necessary analytical continuation of the toric mirror map of $\bX'$ by composition. Therefore we may assume that the fan $\Sigma'$ is obtained by a regular subdivision of $\Sigma$ which introduces only one new ray. In terms of secondary fans, this means that $X'\to X$ is obtained by crossing a single wall. Therefore it remains to construct an analytic continuation of the mirror map in case of a crepant birational map corresponding to crossing a single wall in the secondary fan. This is done in Appendix \ref{app:analy_conti}.
\end{proof}

\begin{example}
In the case when $\bX=[\cpx^2/\Z_m]$ (see Example (1) of Section \ref{sect:eg_SYZ}), and $\bX'$ the minimal resolution of $\bX$, an analytic continuation of the inverse mirror map was explicitly constructed in \cite{CCIT09}. We reproduce the result here. Denote by $g^0_{\bX'}(y'),...,g^{m-1}_{\bX'}(y')$ the inverse mirror map of $\bX'$, and denote by $g_0(y),...,g_{m-1}(y)$ the inverse mirror map of $\bX$. Then according to \cite[Proposition A.7]{CCIT09}, for $1\leq i\leq m-1$, there is an analytic continuation of $g^i_{\bX'}(y')$ such that
\begin{equation*}
g^i_{\bX'}(y')=-\frac{2\pi\sqrt{-1}}{m}+\frac{1}{m}\sum_{k=1}^{m-1}\zeta^{2ki}(\zeta^{-k}-\zeta^k)g_k(y), \text{ where } \zeta=\exp(\frac{\pi\sqrt{-1}}{m}).
\end{equation*}
It may be checked that this yields an identification between the mirrors of $\bX$ and $\bX'$.
\end{example}

\begin{remark}
In the case when $\bX=[\cpx^n/\Z_n]$ (see Example (3) of Section \ref{sect:eg_SYZ}), and $\bX'=\mathcal{O}_{\mathbb{P}^{n-1}}(-n)$, an analytic continuation of the inverse mirror map was explicitly carried out in \cite{CCLT12}. We refer the readers to \cite[Section 6.2]{CCLT12} for more details.
\end{remark}

\appendix

\section{Analytic continuation of mirror maps}\label{app:analy_conti}
We explicitly construct analytic continuations of the toric mirror maps in case of crepant partial resolutions obtained by crossing a single wall in the secondary fan. This is needed in the proof of Theorem \ref{thm:openCR_toric_CY}. The technique of constructing analytical continuations using Mellin-Barnes integrals is well-known and has appeared in e.g. \cite{candelas91}, \cite{BH06}, \cite{CIT09}.

\subsection{Toric basics}
In this subsection we describe the geometric and combinatorial set-up that we are going to consider. Much of the toric geometry needed here is discussed in Section \ref{sec:toric_orbifolds} and repeated here in order to properly set up the notations.

Let $\bX_1$ be a toric CY orbifold given by the stacky fan
\begin{equation}\label{stacky_fan1}
\left(\Sigma_1\subset N_\R, \{\bb_0,\ldots,\bb_{m-1}\}\cup\{\bb_m,\ldots,\bb_{m'-1}\} \right)
\end{equation}
where $N$ is a lattice of rank $n$, $\Sigma_1\subset N_\R$ is a simplicial fan, $\bb_0,\ldots,\bb_{m-1}\in N$ are primitive generators of the rays of $\Sigma_1$, and $\bb_m,\ldots,\bb_{m'-1}$ are extra vectors chosen from $\mathrm{Box}(\Sigma_1)^{\mathrm{age}=1}$. The CY condition means there exists $\underline{\nu}\in M := N^\vee = \textrm{Hom}(N,\Z)$ such that $(\underline{\nu}, \bb_i)=1$ for $i=0,\ldots,m-1$. We also assume that $\bX_1$ is as in Setting \ref{setting:toricCY} so that Assumption \ref{assumption} is satisfied.

The fan sequence of this stacky fan reads
$0\longrightarrow \mathbb{L}_1:=\mathrm{Ker}(\phi_1)\overset{\psi_1}{\longrightarrow} \bigoplus_{i=0}^{m'-1}\Z e_i\overset{\phi_1}{\longrightarrow} N\longrightarrow 0.$
Tensoring with $\C^\times$ yields
$0 \longrightarrow G_1:=\mathbb{L}_1\otimes_\Z \C^\times \longrightarrow (\C^\times)^{m'}\longrightarrow N\otimes_\Z \C^\times \to 0.$
The set of anti-cones of the stacky fan \eqref{stacky_fan1} is given by
$\mathcal{A}_1:=\left\{I\subset \{0,\ldots,m'-1\} \mid \sum_{i\notin I} \R_{\geq 0} \bb_i \text{ is a cone in }\Sigma_1 \right\}.$
Note that $\{0,\ldots,m'-1\}\setminus \{i\} \in \mathcal{A}_1$ if and only if $i\in \{0,\ldots,m-1\}$. Hence if $I\in \mathcal{A}_1$, then $\{m,\ldots,m'-1\}\subset I$. Therefore we may define the following
\begin{equation*}
\mathcal{A}'_1:=\left\{I'\subset \{0,\ldots,m-1\} \mid I'\cup \{m,\ldots,m'-1\}\in \mathcal{A}_1 \right\}.
\end{equation*}
The divisor sequence
$0\longrightarrow M \overset{\phi_1^\vee}{\longrightarrow} \bigoplus_{i=0}^{m-1}\Z e_i^\vee\overset{\psi_1^\vee}{\longrightarrow}\mathbb{L}_1^\vee\longrightarrow 0$
is obtained by dualizing the fan sequence.

For each $i=0,\ldots,m'-1$, we put $D_i:=\psi_1^\vee(e_i^\vee)\in \mathbb{L}_1^\vee$. The extended K\"ahler cone $\widetilde{C}_{\bX_1}$ of $\bX_1$ and the K\"ahler cone $C_{\bX_1}$ of $\bX_1$ are defined to be
\begin{equation*}
\widetilde{C}_{\bX_1}:=\bigcap_{I\in \mathcal{A}_1}\left(\sum_{i\in I}\mathbb{R}_{>0}D_i \right)\subset \mathbb{L}_1^\vee\otimes\mathbb{R},\quad {C}_{\bX_1}:=\bigcap_{I'\in \mathcal{A}'_1}\left(\sum_{i\in I}\mathbb{R}_{>0}\bar{D}_i \right)\subset H^2(\bX_1, \R).
\end{equation*}
We understood that $C_{\bX_1}$ is the image of $\widetilde{C}_{\bX_1}$ under the quotient map
\begin{equation*}
\mathbb{L}_1^\vee\otimes \mathbb{R}\to \mathbb{L}_1^\vee \otimes \R /\sum_{i=m}^{m'-1}\R D_i \simeq H^2(\bX_1, \R).
\end{equation*}
There is a splitting
$\mathbb{L}_1^\vee\otimes \mathbb{R} =\mathrm{Ker}\left(\left(D_{m}^\vee,\ldots,D_{m'-1}^\vee \right): \mathbb{L}_1^\vee\otimes \mathbb{R}\to \mathbb{R}^{m'-m}\right)\oplus \bigoplus_{j=m}^{m'-1}\mathbb{R}D_j,$
and the extended K\"ahler cone is decomposed accordingly:
$\widetilde{C}_{\bX_1}=C_{\bX_1}+\sum_{j=m}^{m'-1} \mathbb{R}_{>0}D_j.$

Let $\omega_1\in \widetilde{C}_{\bX_1}$ be an extended K\"ahler class of $\bX_1$. According to \cite[Section 3.1.1]{iritani09}, the defining condition of $\mathcal{A}_1$ may also be formulated as
$\omega_1\in \sum_{i\in I} \R_{> 0} D_i$.
The extended canonical class of $\bX_1$ is $\hat{\rho}_{\bX_1}:=\sum_{i=0}^{m'-1}D_i$. By \cite[Lemma 3.3]{iritani09}, we have $\hat{\rho}_{\bX_1}=\sum_{i=0}^{m-1}D_i +\sum_{i=m}^{m'-1}\left(1-\mathrm{age}(\bb_i)\right)D_i.$
Since we choose $\bb_i, i=m,\ldots,m'-1$ to have age $1$, we see that $\hat{\rho}_{\bX_1} = \sum_{i=0}^{m-1}D_i = c_1(\bX_1) = 0$.

\subsection{Geometry of wall-crossing}
As mentioned earlier, we want to consider toric crepant birational maps obtained by introducing a new ray. We now describe this in terms of wall-crossing. We refer to \cite[Chapters 14--15]{CLS_toricbook} for the basics of wall-crossings in the toric setting.

By definition, a wall is a subspace
\begin{equation*}
\widetilde{W}=W\oplus \bigoplus_{j=m}^{m'-1} \R D_j \subset \mathbb{L}_1^\vee\otimes \R,
\end{equation*}
where $W$ is a hyperplane given by a linear functional $l$, such that
(1) $C_{\bX_1}\subset \{l>0\}$, and
(2) the intersection $\overline{C}_{\bX_1}\cap W$ of the closure of $C_{\bX_1}$ with $W$ is a top-dimensional cone in $W$.
Let $C_{\bX_1}(W)\subset \overline{C}_{\bX_1}\cap W$ be the relative interior and let $\widetilde{C}_{\bX_1}(W):=C_{\bX_1}(W) \oplus \bigoplus_{j=m}^{m'-1} \R D_j$.

We want to consider a crepant birational map obtained by introducing one new ray. This means that there is exactly one $D_i$ lying outside the K\"ahler cone ${C}_{\bX_1}$. By relabeling the 1-dimensional cones, we may assume that $D_{m-1}$ lies outside $C_{\bX_1}$. More precisely, we assume
\begin{equation}\label{eqn:sings_of_lD}
\left\{
\begin{array}{lll}
l(D_i)>0 & \text{for $0\leq i\leq a-1$},\\
l(D_i)=0 &  \text{for $a\leq i\leq m-2$},\\
l(D_{m-1})<0 &
\end{array}\right.
\end{equation}

Let $\omega_2$ be an extended K\"ahler class in the chamber\footnote{The chamber structure is given by the secondary fan associated to $\Sigma_1$.} adjacent to $(\overline{C}_{\bX_1}\cap W) \oplus \bigoplus_{j=m}^{m'-1} \R D_j$. Following \cite[Section 3.1.1]{iritani09}, we may use $\omega_2$ to define another toric orbifold $\bX_2$ as follows. The set of anti-cones is defined to be
$\mathcal{A}_2:=\left\{ I\subset \{0,\ldots,m'-1\} \mid \omega_2\in \sum_{i\in I} \R_{>0} D_i \right\}$.
The toric orbifold $\bX_2$ is then defined to be the following stack quotient
\begin{equation*}
\bX_2 := \left[\left(\C^{m'}\setminus \bigcup_{I\notin \mathcal{A}_2}\C^I \right)/G_1\right],
\end{equation*}
where $\C^I:=\{(z_0,\ldots,z_{m'-1})\in \C^{m'} \mid z_i=0 \textrm{ for } i\notin I \}$.
The fan $\Sigma_2$ of this toric orbifold is defined from $\mathcal{A}_2$ as follows: $\sum_{i \notin I} \R_{\geq 0} b_i$ is a cone of $\Sigma_2$ if and only if $I\in \mathcal{A}_2$. We also define
$\mathcal{A}'_2 := \left\{I'\subset \{0,\ldots,m-1\} \mid I'\cup \{m,\ldots,m'-1\}\in \mathcal{A}_2 \right\}$.

Next we make a few observations about the two sets $\mathcal{A}_1$, $\mathcal{A}_2$ of anti-cones.
\begin{lemma}\label{lem:anticone1}
Let $I\in \mathcal{A}_1$. Then $I\in \mathcal{A}_2$ if and only if $m-1\in I$.
\end{lemma}
\begin{proof}
Suppose $I\in \mathcal{A}_2$. Then $\omega_2\in \sum_{i\in I} \R_{>0} D_i$. Since $l(D_i)\geq 0$ for all $i$ except $i=m-1$, and $l(\omega_2)<0$, in order for $\omega_2\in \sum_{i\in I}\R_{>0} D_i$ we must have $m-1\in I$.
Suppose that $I\notin \mathcal{A}_2$. Then $\omega_2\notin \sum_{i\in I} \R_{>0} D_i$. But this means that $\R_{>0}\omega_2\notin \sum_{i\in I} \R_{>0} D_i$. This implies $m-1\notin I$.
\end{proof}

\begin{lemma}\label{lem:anticone2}
Let $I\in \mathcal{A}_1$ and $I\notin\mathcal{A}_2$. Then
\begin{enumerate}
\item
$\left(I\cup \{m-1\}\right)\setminus \{0,\ldots,a-1\} \in \mathcal{A}_2$.
\item
If $|I|=\textrm{dim }G_1$, then $I\cap \{0,\ldots,a-1\} = \{i_I\}$ is a singleton, so $(I\cup \{m-1\})\setminus \{i_I\}\in \mathcal{A}_2$.
\end{enumerate}
\end{lemma}
\begin{proof}
The statement (1) follows from the fact that $l(D_i)\leq 0$ for all $i\in \left(I\cup \{m-1\}\right)\setminus \{0,\ldots,a-1\}$. The statement (2) follows from the fact that the minimal size of an anti-cone is equal to $\textrm{dim }G_1$.
\end{proof}

Moving the K\"ahler class $\omega_1$ across the wall $W$ to $\omega_2$ induces a birational map
\begin{equation}\label{eqn:birational_map}
X_1\to X_2.
\end{equation}
between the toric varieties underlying $\bX_1$ and $\bX_2$. In the theory of toric GIT, this map is induced from the variation of GIT quotients by moving the stability parameter from $\omega_1$ to $\omega_2$.

We may describe the birational map $X_1\to X_2$ in terms of the fans. By Lemmas \ref{lem:anticone1} and \ref{lem:anticone2}, if $\sum_{i\notin I} \R_{\geq 0} \bb_i$ is a cone in $\Sigma_1$, then either this cone is also in $\Sigma_2$ (in which case $\R_{\geq 0} \bb_{m-1}$ is not a ray of this cone), or $\sum_{i\notin (I\cup \{m-1\})\setminus \{0,\ldots,a-1\}} \R_{\geq 0} \bb_i$ is a cone in $\Sigma_2$. This shows that the fan $\Sigma_1$ is an refinement of $\Sigma_2$ obtained by adding a new ray $\R_{\geq 0}\bb_{m-1}$. The birational map $X_1\to X_2$ in (\ref{eqn:birational_map}) is induced from this refinement, in a manner described more generally in e.g. \cite[Section 1.4]{Fu}.

\begin{lemma}
The birational map \eqref{eqn:birational_map} contracts the divisor $\bar{D}_{m-1}\subset X_1$ and is crepant.
\end{lemma}
\begin{proof}
The fan sequence implies that \eqref{eqn:birational_map} contracts the divisor $\bar{D}_{m-1}$. Since $\bX_1$ is toric CY, there exists $\underline{\nu}\in N^\vee$ such that $(\underline{\nu}, \bb_i)=1$ for $i=0,...,m-1$. We conclude that $X_1\to X_2$ is crepant by applying the criterion for being crepant (see e.g. \cite[Section 3.4]{Fu} and \cite[Remark 7.2]{BCS}) with the support function $(\underline{\nu}, -)$.
\end{proof}

\subsection{Analytic continuations}
Recall that
\begin{equation*}
\begin{split}
\mathbb{K}_1 := \left\{d\in \mathbb{L}_1\otimes \Q \mid \{ i \mid \langle D_i, d\rangle\in \Z \}\in \mathcal{A}_1 \right\},\\
\mathbb{K}_2 := \left\{d\in \mathbb{L}_1\otimes \Q \mid \{ i \mid \langle D_i, d\rangle\in \Z \}\in \mathcal{A}_2 \right\}.
\end{split}
\end{equation*}
As defined in \eqref{eqn:reduction_func}, there are reduction functions
\begin{equation*}
\begin{split}
\nu: \mathbb{K}_1\to \mathrm{Box}(\Sigma_1), \quad
\nu: \mathbb{K}_2\to \mathrm{Box}(\Sigma_2),
\end{split}
\end{equation*}
which are surjective and have kernels $\mathbb{L}_1$. This gives the identifications
\begin{equation}\label{eqn:identification_box}
\begin{split}
\mathbb{K}_1/\mathbb{L}_1= \mathrm{Box}(\Sigma_1),\quad
\mathbb{K}_2/\mathbb{L}_1= \mathrm{Box}(\Sigma_2).
\end{split}
\end{equation}

We now discuss the toric mirror map. By \eqref{eqn:toric_mirror_map_X}, the toric mirror map of $\bX_1$ is given by
\begin{equation}\label{eqn:toric_mirror_map_app}
\begin{split}
\log q_a = & \log y_a + \sum_{j=0}^{m-1} Q_{ja}A^\bX_j(y), \quad a = 1,\ldots,r',\\
\tau_{\bb_j} = & A^\bX_j(y), \quad j=m,\ldots,m'-1,
\end{split}
\end{equation}

Some explanations are in order. Fix an integral basis $\{p_1,\ldots,p_r\}\subset \mathbb{L}_1^\vee$, where $r = m'-n$. For $d\in \mathbb{L}_1\otimes \Q$, we write $q^d=\prod_{a=1}^{r'} q_a^{\langle \bar{p}_a, d\rangle}$, $y^d=\prod_{a=1}^r y_a^{\langle p_a, d\rangle}$
which define $q_a$ and $y_a$, where $r' = m-n$ and $\{\bar{p}_1,\ldots,\bar{p}_{r'}\}$ are images of $\{p_1,\ldots,p_{r'}\}$ under the quotient map $\mathbb{L}_1^\vee\otimes \Q\to H^2(\bX_1; \Q)$ and they give a nef basis for $H^2(\bX_1;\Q)$. Also, $Q_{ia}$ are chosen so that
\begin{equation}\label{eqn:div_basis}
D_i=\sum_{a=1}^r Q_{ia}p_a, \quad i=0,\ldots,m-1.
\end{equation}

For $j=0,1,\ldots,m-1$, we have
\begin{equation*}
\begin{split}
\Omega^{\bX_1}_j & = \{d\in (\mathbb{K}_1)_\mathrm{eff} \mid \nu(d)=0, \langle D_j,d\rangle \in \Z_{<0}\textrm{ and } \langle D_i,d\rangle \geq 0 \in \Z_{\geq 0}\ \forall i\neq j\},\\
A^{\bX_1}_j(y) & = \sum_{d\in \Omega^{\bX_1}_j}y^d \frac{(-1)^{-\langle D_j,d\rangle-1}(-\langle D_j,d\rangle-1)!}{\prod_{i\neq j}\langle D_i,d\rangle!}.
\end{split}
\end{equation*}
For $j=m,\ldots,m'-1$, we have
\begin{equation*}
\begin{split}
\Omega^{\bX_1}_j & = \{d\in (\mathbb{K}_1)_\mathrm{eff} \mid \nu(d)=\bb_j\textrm{ and }\langle D_i,d\rangle \notin \Z_{<0}\ \forall i\},\\
A^{\bX_1}_j(y) & = \sum_{d\in \Omega^{\bX_1}_j}y^d \prod_{i=0}^{m'-1}
\frac{\prod_{k=\lceil\langle D_i,d\rangle\rceil}^\infty(\langle D_i,d\rangle-k)}{\prod_{k=0}^\infty(\langle D_i,d\rangle-k)}.
\end{split}
\end{equation*}

To study the analytic continuation of \eqref{eqn:toric_mirror_map_app}, we first need to be more precise about the variables involved. We pick $p_1,\ldots, p_r$ such that $p_1$ is contained in the closure of $\widetilde{C}_{\bX_1}$ and $p_2,\ldots,p_r\in \widetilde{C}_{\bX_1}(W)$. Applying the linear functional $l\oplus 0$ to \eqref{eqn:div_basis} gives
$$l(D_i)=Q_{i1}l(p_1)+\sum_{a\geq 2} Q_{ia} l(p_a).$$
By the choice of $p_1,\ldots,p_r$, we have $l(p_1)>0$ and $l(p_a)=0$ for $a\geq 2$. The signs of $l(D_j)$ are given in \eqref{eqn:sings_of_lD}. This implies that
\begin{equation*}
\left\{
\begin{array}{lll}
Q_{i1}>0 & \mbox{ for } 0\leq i\leq a-1,\\
Q_{i1}=0 &  \mbox{ for } a\leq i\leq m-2,\\
Q_{m-1,1}<0 &
\end{array}\right.
\end{equation*}
Since $0=\sum_{i=0}^{m'-1}D_i=\sum_{i=0}^{m'-1}\sum_{a=1}^r Q_{ia}p_a$, we have $\sum_{i=0}^{m'-1} Q_{ia}=0$ for all $a = 1,\ldots,r$. Also note that $Q_{ia} = 0$ for $1\leq a\leq r'$ and $m\leq i\leq m'-1$.

We now proceed to construct an analytic continuation of $A_j(y)$ where $j\in \{0,\ldots,m'-1\}$. We do this in details only for $j\in \{m,\ldots,m'-1\}$ because the case when $j\in \{0,\ldots, m-1\}$ is similar.

Let $j\in \{m,\ldots, m'-1\}$. The element $\bb_j\in \mathrm{Box}(\Sigma_1)^{\mathrm{age}=1}$ corresponds to a component $\bX_{1,\bb_j}$ of the inertia orbifold $I\bX_1$. According to \cite[Lemma 4.6]{BCS}, $\bX_{1,\bb_j}$ is the toric Deligne-Mumford stack associated to the quotient stacky fan $\Sigma_1/\sigma(\bb_j)$, where $\sigma(\bb_j)$ is the minimal cone in $\Sigma_1$ that contains $\bb_j$. Let $d_{\bb_j}\in \mathbb{K}_1$ be the unique element such that $\nu(d_{\bb_j}) = \bb_j$ and $\langle p_a, d_{\bb_j}\rangle\in [0,1)$. Then by the identification of $\text{Box}$ in (\ref{eqn:identification_box}), every $d\in \mathbb{K}_1$ with $\nu(d) = \bb_j$ can be written as
$$d=d_{\bb_j}+d_0 \text{ with $d_0\in \mathbb{L}_1$}.$$

We consider $A^{\bX_1}_j(y)$. Put
\begin{equation*}
\mathcal{A}_{1,\bb_j} := \left\{I\subset \{0,\ldots,m'-1\} \mid \sum_{i\notin I} \R_{\geq 0} \bb_i \textrm{ is a cone in } \Sigma_1, \langle D_i, d_{\bb_j}\rangle \in \Z \textrm{ for } i \in I \right\}\subset \mathcal{A}_1,
\end{equation*}
and define
$\widetilde{C}_{\bX_{1,\bb_j}} := \bigcap_{I\in \mathcal{A}_{1,\bb_j}}\left(\sum_{i\in I}\mathbb{R}_{>0}D_i \right) = C_{\bX_{1,\bb_j}}+\sum_{i=m}^{m'-1}\R_{\geq 0} D_i$.
Clearly $\widetilde{C}_{\bX_1}\subset \widetilde{C}_{\bX_{1,\bb_j}}$. Taking duals gives
$$\overline{NE}(\bX_{1,\bb_j}) := \widetilde{C}_{\bX_{1,\bb_j}}^\vee\subset \widetilde{C}_{\bX_1}^\vee =: \overline{NE}(\bX_1).$$
By definition, $A_j(y)$ is a series in $y$ whose exponents are contained in $\Omega_j$. It is straightforward to check that $\Omega_j\subset \overline{NE}(\bX_{1,\bb_j})$. In this way we interpret $A_j(y)$ as a function on $\widetilde{C}_{\bX_{1,\bb_j}}$ and a function on $\widetilde{C}_{\bX_1}$ by restriction.

If we also have $\widetilde{C}_{\bX_2}\subset \widetilde{C}_{\bX_{1,\bb_j}}$, then $A_j(y)$ can also be interpreted as a function on $\widetilde{C}_{\bX_2}$ by restriction. So in this case no analytic continuation is needed.

It remains to consider those $\bb_j$ such that $\widetilde{C}_{\bX_2}$ is not contained in $\widetilde{C}_{\bX_{1,\bb_j}}$. First observe that $A_j(y)$ can be rewritten as follows:
\begin{equation*}
A_j(y)=\sum_{d_0\in \mathbb{L}_1} y^{d_{\bb_j}}y^{d_0} \prod_{i=0}^{m'-1}\frac{\Gamma(\{\langle D_i, d_{\bb_j}+d_0\rangle\} +1)}{\Gamma(\langle D_i, d_{\bb_j}+d_0\rangle +1)}.
\end{equation*}
We put $\Gamma_{\bb_j}:= \prod_{i=0}^{m'-1}\Gamma(\{\langle D_i, d_{\bb_j}+d_0\rangle\} +1)$ so that we can write
\begin{equation*}
A_j(y)=\sum_{d_0\in \mathbb{L}_1}y^{d_{\bb_j}}y^{d_0}\Gamma_{\bb_j}\frac{1}{\Gamma(\langle D_{m-1}, d_{\bb_j}+d_0\rangle +1)}\frac{1}{\prod_{i\neq m-1}\Gamma(\langle D_i, d_{\bb_j}+d_0\rangle +1)}.
\end{equation*}
Since $\Gamma(s)\Gamma(1-s)=\pi/\sin (\pi s)$, we have
\begin{equation*}
\frac{1}{\Gamma(\langle D_{m-1}, d_{\bb_j}+d_0\rangle +1)}=-\frac{\sin (\pi\langle D_{m-1}, d_{\bb_j}+d_0\rangle)}{\pi}\Gamma(-\langle D_{m-1}, d_{\bb_j}+d_0\rangle),
\end{equation*}
and
\begin{equation*}
A_j(y)=\sum_{d_0\in \mathbb{L}_1}y^{d_{\bb_j}}y^{d_0}\frac{\Gamma_{\bb_j}}{\pi} \sin (\pi\langle D_{m-1}, d_{\bb_j}+d_0\rangle)\frac{-\Gamma(-\langle D_{m-1}, d_{\bb_j}+d_0\rangle)}{\prod_{i\neq m-1}\Gamma(\langle D_i, d_{\bb_j}+d_0\rangle +1)}.
\end{equation*}

We put $d_{0a}:=\langle p_a, d_0\rangle$. In view of \eqref{eqn:div_basis}, we have
\begin{equation*}
\frac{-\Gamma(-\langle D_{m-1}, d_{\bb_j}+d_0\rangle)}{\prod_{i\neq m-1}\Gamma(\langle D_i, d_{\bb_j}+d_0\rangle +1)}=\frac{-\Gamma(-\langle D_{m-1}, d_{\bb_j}\rangle-Q_{m-1,1}d_{01}-\sum_{a\neq 1} Q_{m-1a}d_{0a})}{\prod_{i\neq m-1}\Gamma(\langle D_i, d_{\bb_j}\rangle +1+Q_{m-1,1}d_{01}+\sum_{a\neq 1} Q_{m-1a}d_{0a})}.
\end{equation*}
Since $y^{d_0} = \prod_{a=1}^r y_a^{\langle p_a, d_0\rangle} = \prod_{a=1}^{r} y_a^{d_{0a}}$, we have
\begin{equation*}
\begin{split}
A_j(y)
= & \frac{\Gamma_{\bb_j}}{\pi}\sum_{d_{01},\ldots,d_{0r}\geq 0}y^{d_{\bb_j}}\left(\prod_{a\geq 2} y_a^{d_{0a}} \right) \sin (\pi\langle D_{m-1}, d_{\bb_j}+d_0\rangle)\\
& \quad \quad \quad \quad \quad \quad \times \frac{-\Gamma(-\langle D_{m-1}, d_{\bb_j}\rangle-Q_{m-1,1}d_{01}-\sum_{a\neq 1} Q_{m-1a}d_{0a})}{\prod_{i\neq m-1}\Gamma(\langle D_i, d_{\bb_j}\rangle +1+Q_{m-1,1}d_{01}+\sum_{a\neq 1} Q_{m-1a}d_{0a})}\\
= & \frac{\Gamma_{\bb_j}}{\pi}\sum_{d_{02},\ldots,d_{0r}\geq 0}y^{d_{\bb_j}}\left(\prod_{a\geq 2} y_a^{d_{0a}} \right) \sin \left(\pi\langle D_{m-1}, d_{\bb_j}\rangle+\sum_{a \neq 1} Q_{m-1, a} d_{0a}\right)\\
& \quad \quad \quad \quad \quad \quad \times\left( \sum_{d_{01}\geq 0}\left((-1)^{Q_{m-1,1}}y_1 \right)^{d_{01}} \frac{-\Gamma(-\langle D_{m-1}, d_{\bb_j}\rangle-Q_{m-1,1}d_{01}-\sum_{a\neq 1} Q_{m-1a}d_{0a})}{\prod_{i\neq m-1}\Gamma(\langle D_i, d_{\bb_j}\rangle +1+Q_{m-1,1}d_{01}+\sum_{a\neq 1} Q_{m-1a}d_{0a})}\right).
\end{split}
\end{equation*}

Now observe that
\begin{equation}\label{eqn:residue_1}
\begin{split}
& \sum_{d_{01}\geq 0}\left((-1)^{Q_{m-1,1}}y_1 \right)^{d_{01}} \frac{-\Gamma(-\langle D_{m-1}, d_{\bb_j}\rangle-Q_{m-1,1}d_{01}-\sum_{a\neq 1} Q_{m-1a}d_{0a})}{\prod_{i\neq m-1}\Gamma(\langle D_i, d_{\bb_j}\rangle +1+Q_{m-1,1}d_{01}+\sum_{a\neq 1} Q_{m-1a}d_{0a})}\\
= & \mathrm{Res}_{s\in \mathbb{N} \cup \{0\}} ds \frac{-\Gamma(-s)((-1)^{Q_{m-1,1}}y_1)^s \Gamma(-\langle D_{m-1}, d_{\bb_j}\rangle-Q_{m-1,1}s-\sum_{a\neq 1} Q_{m-1a}d_{0a})}{\prod_{i\neq m-1}\Gamma(\langle D_i, d_{\bb_j}\rangle +1+Q_{m-1,1}s+\sum_{a\neq 1} Q_{m-1a}d_{0a})}.
\end{split}
\end{equation}
Fix a sign of $y_1$ so that $(-1)^{Q_{m-1,1}}y_1\in \R_{>0}$. By using the Mellin-Barnes integral technique (see e.g. \cite[Section 4]{BH06} and \cite[Lemma A.6]{BH06}), we have that the right-hand side of \eqref{eqn:residue_1} is
\begin{equation}\label{eqn:residue_2}
\begin{split}
 \oint_{C_{d_{02},\ldots,d_{0r}}} ds \frac{-\Gamma(-s)((-1)^{Q_{m-1,1}}y_1)^s \Gamma(-\langle D_{m-1}, d_{\bb_j}\rangle-Q_{m-1,1}s-\sum_{a\neq 1} Q_{m-1a}d_{0a})}{\prod_{i\neq m-1}\Gamma(\langle D_i, d_{\bb_j}\rangle +1+Q_{m-1,1}s+\sum_{a\neq 1} Q_{m-1a}d_{0a})},
\end{split}
\end{equation}
where $C_{d_{02},\ldots,d_{0r}}$ is a contour on the plane with (complex) coordinate $s$ that runs from $s=-\sqrt{-1}\infty$ to $s=+\sqrt{-1}\infty$, dividing the plane into two parts so that
\begin{equation}\label{eqn:pole_left}
\mathrm{Pole}_L := \left\{\frac{\langle D_{m-1}, d_{\bb_j}\rangle+\sum_{a\neq 1} Q_{m-1a}d_{0a}-l}{-Q_{m-1,1}} \mid l=0,1,\ldots\right\}
\end{equation}
lies on one part and $\{0,1,\ldots\}$ lies on the other part. Note that $-Q_{m-1,1}>0$.

To analytically continue to the region where $|y_1|$ is large, we close the contour $C_{d_{02},\ldots,d_{0r}}$ to the left to enclose all poles in $\mathrm{Pole}_L$. This shows that \eqref{eqn:residue_2} is
\begin{equation*}
\begin{split}
\mathrm{Res}_{s\in \mathrm{Pole}_L} ds \frac{-\Gamma(-s)((-1)^{Q_{m-1,1}}y_1)^s \Gamma(-\langle D_{m-1}, d_{\bb_j}\rangle-Q_{m-1,1}s-\sum_{a\neq 1} Q_{m-1a}d_{0a})}{\prod_{i\neq m-1}\Gamma(\langle D_i, d_{\bb_j}\rangle +1+Q_{m-1,1}s+\sum_{a\neq 1} Q_{m-1a}d_{0a})},
\end{split}
\end{equation*}
which is equal to
\begin{equation*}
\begin{split}
& \sum_{l\geq 0}\frac{(-1)^l}{l!}\frac{\Gamma\left(\frac{\langle D_{m-1}, d_{\bb_j}\rangle+\sum_{a\neq 1} Q_{m-1a}d_{0a}-l}{Q_{m-1,1}}\right)\left((-1)^{Q_{m-1,1}}y_1\right)^{\frac{\langle D_{m-1}, d_{\bb_j}\rangle+\sum_{a\neq 1} Q_{m-1a}d_{0a}-l}{-Q_{m-1,1}}}}{\prod_{i\neq m-1}\Gamma\left(\langle D_i, d_{\bb_j}\rangle +1+Q_{m-1,1}\times \frac{\langle D_{m-1}, d_{\bb_j}\rangle+\sum_{a\neq 1} Q_{m-1a}d_{0a}-l}{-Q_{m-1,1}} +\sum_{a\neq 1} Q_{m-1a}d_{0a}\right)}\\
= & \sum_{l\geq 0}\frac{(-1)^l}{l!}\frac{\left((-1)^{Q_{m-1,1}}y_1\right)^{\frac{\langle D_{m-1}, d_{\bb_j}\rangle+\sum_{a\neq 1} Q_{m-1a}d_{0a}-l}{-Q_{m-1,1}}}\frac{\pi}{-Q_{m-1,1}\sin \pi \left(\frac{\langle D_{m-1}, d_{\bb_j}\rangle+\sum_{a\neq 1} Q_{m-1a}d_{0a}-l}{-Q_{m-1,1}} \right)}}{\prod_{i\neq m-1}\Gamma\left(\langle D_i, d_{\bb_j}\rangle +1+Q_{m-1,1}\times \frac{\langle D_{m-1}, d_{\bb_j}\rangle+\sum_{a\neq 1} Q_{m-1a}d_{0a}-l}{-Q_{m-1,1}} +\sum_{a\neq 1} Q_{m-1a}d_{0a}\right)}\times\\
& \quad \quad \quad \times \frac{1}{\Gamma\left(1 - \frac{\langle D_{m-1}, d_{\bb_j}\rangle+\sum_{a\neq 1} Q_{m-1a}d_{0a}-l}{Q_{m-1,1}}\right)},
\end{split}
\end{equation*}
where we again use $\Gamma(s)\Gamma(1-s)=\pi/\sin (\pi s)$.

This gives an analytic continuation of $A_j(y)$:
\begin{equation}\label{eqn:ana_conti_A_j}
\begin{split}
 A_j(y)
= & \frac{\Gamma_{\bb_j}}{\pi}\sum_{d_{02},\ldots,d_{0r}\geq 0}y^{d_{\bb_j}}\left(\prod_{a\geq 2} y_a^{d_{0a}} \right) \sin \left(\pi\langle D_{m-1}, d_{\bb_j}\rangle+\pi\sum_{a \neq 1} Q_{m-1, a} d_{0a}\right)\\
&\times \sum_{l\geq 0}\frac{(-1)^l}{l!}\frac{\left((-1)^{Q_{m-1,1}}y_1\right)^{\frac{\langle D_{m-1}, d_{\bb_j}\rangle+\sum_{a\neq 1} Q_{m-1a}d_{0a}-l}{-Q_{m-1,1}}}\frac{\pi}{-Q_{m-1,1}\sin \pi \left(\frac{\langle D_{m-1}, d_{\bb_j}\rangle+\sum_{a\neq 1} Q_{m-1a}d_{0a}-l}{-Q_{m-1,1}} \right)}}{\prod_{i\neq m-1}\Gamma\left(\langle D_i, d_{\bb_j}\rangle +1+Q_{m-1,1}\times \frac{\langle D_{m-1}, d_{\bb_j}\rangle+\sum_{a\neq 1} Q_{m-1a}d_{0a}-l}{-Q_{m-1,1}} +\sum_{a\neq 1} Q_{m-1a}d_{0a}\right)}\\
& \quad \quad \quad \times \frac{1}{\Gamma\left(1 - \frac{\langle D_{m-1}, d_{\bb_j}\rangle+\sum_{a\neq 1} Q_{m-1a}d_{0a}-l}{Q_{m-1,1}}\right)}.
\end{split}
\end{equation}

It remains to show that the expression in \eqref{eqn:ana_conti_A_j} can be interpreted as a function on $\widetilde{C}_{\bX_2}$. To do this, we need a new set of variables. Pick another integral basis of $\{\hat{p}_1,\ldots,\hat{p}_r\}\subset \mathbb{L}_1^\vee\otimes \Q$ such that
$\hat{p}_1:=D_{m-1}$ and $\hat{p}_a:=p_a$ for $a=2,\ldots,r$.
Introduce the corresponding variables $\hat{y}_1,\ldots,\hat{y}_r$, namely $y^d=\hat{y}^d=\prod_{a=1}^r \hat{y}_a^{\langle \hat{p}_a,d \rangle}$. From this it is easy to see that
$\hat{y}_1=y_1^{1/Q_{m-1,1}}$ and $\hat{y}_a= y_1^{-Q_{m-1,a}/Q_{m-1,1}}y_a$ for $a=2,\ldots,r$.
We may express $D_i$ in terms of $\hat{p}_1,\ldots,\hat{p}_r$ as follows:
\begin{equation*}
\begin{split}
D_i=\sum_{a=1}^r Q_{ia}p_a=Q_{i1}p_1+\sum_{a\geq 2} Q_{ia}p_a
= \frac{Q_{i1}}{Q_{m-1,1}}\hat{p}_1+\sum_{a\geq 2}\left(Q_{ia}-\frac{Q_{i1}Q_{m-1,a}}{Q_{m-1,1}} \right)\hat{p}_a.
\end{split}
\end{equation*}

Next we interpret the expression in \eqref{eqn:ana_conti_A_j} as a series in $\hat{y}$ whose exponents are contained in $\overline{NE}(\bX_{2})=\widehat{C}_{\bX_2}^\vee$. Define $\hat{d}_{\bb_j}\in \mathbb{L}_1\otimes \Q$ to be the unique class such that
\begin{equation}\label{eqn:hat_d_b_j}
\langle \hat{p}_1, \hat{d}_{\bb_j}\rangle =0, \quad \langle \hat{p}_a, \hat{d}_{\bb_j}\rangle =\langle p_a, d_{\bb_j}\rangle, \textrm{ for } a=2,\ldots,r.
\end{equation}
Given $l, d_{02},\ldots,d_{0r}\geq 0$, define $\hat{d}_0 \in \mathbb{L}_1\otimes \Q $ to be the unique class such that
\begin{equation}\label{eqn:hat_d_0}
\langle \hat{p}_1, \hat{d}_0\rangle =l, \quad \langle \hat{p}_a, \hat{d}_0\rangle =d_{0a}, \text{ for } a=2,\ldots,r.
\end{equation}

\begin{lemma}
Given $l, d_{02},\ldots,d_{0r}\geq 0$. Then $\hat{d}:=\hat{d}_{\bb_j}+\hat{d}_0$ is contained in $\mathbb{K}_2$.
\end{lemma}
\begin{proof}
First note that $\langle D_{m-1}, \hat{d}\rangle=\langle \hat{p}_1, \hat{d}_{\bb_j}+\hat{d}_0\rangle =l\in \Z$.
Let $i\in \{a, \ldots, m-2\}$. We consider $\langle D_i, \hat{d}\rangle$. Let $\hat{p}_1^\vee, \ldots, \hat{p}_r^\vee$ be such that $\langle \hat{p}_a, \hat{p}_b^\vee\rangle =\delta_{ab}$. We calculate $\langle \hat{p}_1, d_0\rangle=\sum_{a\geq 1} Q_{m-1,a}d_{0a}$ and $\langle \hat{p}_a, d_0\rangle=d_{0a}$ for $a\geq 2$. So
$$d_0=\left(\sum_{a\geq 1} Q_{m-1,a}d_{0a}\right)\hat{p}_1^\vee+\sum_{a\geq 2}d_{0a} \hat{p}_a^\vee.$$
By \eqref{eqn:hat_d_b_j} and \eqref{eqn:hat_d_0}, we have
\begin{equation*}
\begin{split}
\hat{d}=\hat{d}_{\bb_j}+\hat{d}_0 & = d_{\bb_j}-\langle p_a, d_{\bb_j}\rangle \hat{p}_1^\vee+ d_0+ \left(l- \sum_{a\geq 1}Q_{m-1,a}d_{0a}\right) \hat{p}_1^\vee\\
& = d_{\bb_j}+d_0+\left(l-\langle p_a, d_{\bb_j}\rangle- \sum_{a\geq 1}Q_{m-1,a}d_{0a} \right) \hat{p}_1^\vee.
\end{split}
\end{equation*}
Since $i\in \{a, \ldots, m-2\}$, we have $D_i\in \widetilde{C}_{\bX_1}(W)$. So $D_i$ is a linear combination of $\hat{p}_2,\ldots,\hat{p}_r$. This implies that $\langle D_i, \hat{p}_1^\vee\rangle=0$, and hence
$\langle D_i, \hat{d}\rangle =\langle D_i, d_{\bb_j}+d_0\rangle$.
We know that $\langle D_i, d_0\rangle=\sum_{a=1}^rQ_{ia}\langle p_a, d_0\rangle=\sum_{a=1}^r Q_{ia}d_{0a}\in \Z$. So $\langle D_i, \hat{d}\rangle =\langle D_i, d_{\bb_j}+d_0\rangle \in \Z$ if and only if $\langle D_i, d_{\bb_j}\rangle\in \Z$.

By assumption, $\widetilde{C}_{\bX_2}$ is not contained in  $\widetilde{C}_{\bX_{1,\bb_j}}$. It follows easily that
$\sum_{\substack{i\in \{a,\ldots,m-2\}\\ \langle D_i, d_{\bb_j}\rangle \in \Z}} \R_{> 0}D_i$
must contain $\overline{C}_{\bX_1}\cap W$. Thus
$\R_{>0}D_{m-1}+ \sum_{\substack{i\in \{a,\ldots,m-2\}\\ \langle D_i, d_{\bb_j}\rangle \in \Z}} \R_{\geq 0}D_i$
contains the K\"ahler class $\omega_2$, and $\{m-1\}\cup \{i\in \{a,\ldots,m-2\} \mid \langle D_i, d_{\bb_j}\rangle \in \Z\}$ is in $\mathcal{A}_2'$.
Since $\langle D_i, \hat{d}\rangle \in \Z$ for all $i\in \{m-1\}\cup \{i\in \{a,\ldots,m-2\} \mid \langle D_i, d_{\bb_j}\rangle \in \Z\}$, we conclude that $\hat{d}\in \mathbb{K}_2$ by the definition of $\mathbb{K}_2$.
\end{proof}

We calculate
\begin{equation*}
\begin{split}
& \langle D_i, d_{\bb_j}\rangle +1+Q_{m-1,1}\times \frac{\langle D_{m-1}, d_{\bb_j}\rangle+\sum_{a\neq 1} Q_{m-1a}d_{0a}-l}{-Q_{m-1,1}} +\sum_{a\neq 1} Q_{m-1a}d_{0a}\\
= & \frac{Q_{i1}}{Q_{m-1,1}}l+\sum_{a\neq 1} \left(Q_{ia}-\frac{Q_{i1}Q_{m-1,a}}{Q_{m-1,1}} \right)d_{0a}-\frac{Q_{i1}}{Q_{m-1,1}}\langle D_{m-1}, d_{\bb_j}\rangle + \langle D_i, d_{\bb_j} \rangle\\
= & \langle D_i, \hat{d}_0\rangle +\langle D_i-\frac{Q_{i1}}{Q_{m-1,1}}D_{m-1}, \hat{d}_{b_j}\rangle.
\end{split}
\end{equation*}
Also,
\begin{equation*}
\begin{split}
 \left((-1)^{Q_{m-1,1}}y_1\right)^{\frac{\langle D_{m-1}, d_{\bb_j}\rangle+\sum_{a\neq 1} Q_{m-1a}d_{0a}-l}{-Q_{m-1,1}}}
=  (-1)^{(\langle D_{m-1}, d_{\bb_j}\rangle+\sum_{a\neq 1} Q_{m-1a}d_{0a}-l)}\hat{y}_1^{-(\langle D_{m-1}, d_{\bb_j}\rangle+\sum_{a\neq 1} Q_{m-1a}d_{0a}-l)},
\end{split}
\end{equation*}
\begin{equation*}
\begin{split}
y_a^{d_{0a}}=\hat{y}_a^{d_{0a}}\hat{y}_1^{Q_{m-1,a}d_{0a}} \textrm{ for } a\geq 2,
\end{split}
\end{equation*}
which gives
\begin{equation*}
\begin{split}
 y^{d_{\bb_j}}\left(\prod_{a\geq 2} y_a^{d_{0a}}\right)\left((-1)^{Q_{m-1,1}}y_1\right)^{\frac{\langle D_{m-1}, d_{\bb_j}\rangle+\sum_{a\neq 1} Q_{m-1a}d_{0a}-l}{-Q_{m-1,1}}}
=(-1)^{Q_{m-1,1}\times\frac{\langle D_{m-1}, d_{\bb_j}\rangle+\sum_{a\neq 1} Q_{m-1a}d_{0a}-l}{Q_{m-1,1}}}\hat{y}^{\hat{d}_{\bb_j}}\hat{y}^{\hat{d}_0}.
\end{split}
\end{equation*}
Also
\begin{equation*}
\begin{split}
\frac{\langle D_{m-1}, d_{\bb_j}\rangle+\sum_{a\neq 1} Q_{m-1a}d_{0a}-l}{Q_{m-1,1}}=\langle \frac{D_{m-1}}{Q_{m-1,1}}, d_{\bb_j}\rangle +\langle \frac{\hat{p}_1-\sum_{a\neq 1} Q_{m-1,a}\hat{p}_a}{Q_{m-1,1}}, \hat{d}_0\rangle.
\end{split}
\end{equation*}
From these calculations it is easy to see that the expression in \eqref{eqn:ana_conti_A_j} can be interpreted as a series in $\hat{y}$ whose exponents are contained in $\overline{NE}(\bX_{2})=\widehat{C}_{\bX_2}^\vee$. This completes the construction of the analytic continuation.

\bibliographystyle{amsplain}
\bibliography{geometry}

\providecommand{\bysame}{\leavevmode\hbox to3em{\hrulefill}\thinspace}
\providecommand{\MR}{\relax\ifhmode\unskip\space\fi MR }
\providecommand{\MRhref}[2]{%
  \href{http://www.ams.org/mathscinet-getitem?mr=#1}{#2}
}
\providecommand{\href}[2]{#2}
\begin{thebibliography}{10}

\bibitem{AAK12}
M.~Abouzaid, D.~Auroux, and L.~Katzarkov, \emph{Lagrangian fibrations on
  blowups of toric varieties and mirror symmetry for hypersurfacese}, preprint,
  \href{http://arxiv.org/abs/1205.0053}{arXiv:1205.0053}.

\bibitem{AGV08}
D.~Abramovich, T.~Graber, and A.~Vistoli, \emph{Gromov-{W}itten theory of
  {D}eligne-{M}umford stacks}, Amer. J. Math. \textbf{130} (2008), no.~5,
  1337--1398. \MR{2450211, Zbl 1193.14070}

\bibitem{auroux07}
D.~Auroux, \emph{Mirror symmetry and {$T$}-duality in the complement of an
  anticanonical divisor}, J. G\"okova Geom. Topol. GGT \textbf{1} (2007),
  51--91. \MR{2386535, Zbl 1181.53076}

\bibitem{auroux09}
\bysame, \emph{Special {L}agrangian fibrations, wall-crossing, and mirror
  symmetry}, Surveys in differential geometry. {V}ol. {XIII}. {G}eometry,
  analysis, and algebraic geometry: forty years of the {J}ournal of
  {D}ifferential {G}eometry, Surv. Differ. Geom., vol.~13, Int. Press,
  Somerville, MA, 2009, pp.~1--47. \MR{2537081, Zbl 1184.53085}

\bibitem{B93}
V.~Batyrev, \emph{Variations of the mixed {H}odge structure of affine
  hypersurfaces in algebraic tori}, Duke Math. J. \textbf{69} (1993), no.~2,
  349--409. \MR{1203231, Zbl 0812.14035}

\bibitem{BCS}
L.~Borisov, L.~Chen, and G.~Smith, \emph{The orbifold {C}how ring of toric
  {D}eligne-{M}umford stacks}, J. Amer. Math. Soc. \textbf{18} (2005), no.~1,
  193--215 (electronic). \MR{2114820, Zbl 1178.14057}

\bibitem{BH06}
Lev~A. Borisov and R.~Paul Horja, \emph{Mellin-{B}arnes integrals as
  {F}ourier-{M}ukai transforms}, Adv. Math. \textbf{207} (2006), no.~2,
  876--927. \MR{2271990, Zbl 1137.14314}

\bibitem{brini-cavalieri11}
A.~Brini and R.~Cavalieri, \emph{Open orbifold {G}romov-{W}itten invariants of
  {$[\Bbb C^3/\Bbb Z_n]$}: localization and mirror symmetry}, Selecta Math.
  (N.S.) \textbf{17} (2011), no.~4, 879--933. \MR{2861610, Zbl 1236.14046}

\bibitem{BCR13-2}
A.~Brini, R.~Cavalieri, and D.~Ross, \emph{Crepant resolutions and open
  strings}, preprint, \href{http://arxiv.org/abs/1309.4438}{arXiv:1309.4438}.

\bibitem{Bryan-Graber09}
J.~Bryan and T.~Graber, \emph{The crepant resolution conjecture}, Algebraic
  geometry---{S}eattle 2005. {P}art 1, Proc. Sympos. Pure Math., vol.~80, Amer.
  Math. Soc., Providence, RI, 2009, pp.~23--42. \MR{2483931, Zbl 1198.14053}

\bibitem{candelas91}
P.~Candelas, X.~de~la Ossa, P.~Green, and L.~Parkes, \emph{A pair of
  {C}alabi-{Y}au manifolds as an exactly soluble superconformal theory},
  Nuclear Phys. B \textbf{359} (1991), no.~1, 21--74. \MR{1115626, Zbl
  1098.32506}

\bibitem{Cavalieri-Ross11}
R.~Cavalieri and D.~Ross, \emph{Open {G}romov-{W}itten theory and the crepant
  resolution conjecture}, Michigan Math. J. \textbf{61} (2012), no.~4,
  807--837. \MR{3049291, Zbl 1273.14111}

\bibitem{Chan10}
K.~Chan, \emph{A formula equating open and closed {G}romov-{W}itten invariants
  and its applications to mirror symmetry}, Pacific J. Math. \textbf{254}
  (2011), no.~2, 275--293. \MR{2900016, Zbl 1246.53116}

\bibitem{CCLT12}
K.~Chan, C.-H. Cho, S.-C. Lau, and H.-H. Tseng, \emph{Lagrangian {F}loer
  superpotentials and crepant resolutions for toric orbifolds}, Comm. Math.
  Phys. \textbf{328} (2014), no.~1, 83--130. \MR{3196981, Zbl 1109.53079}

\bibitem{chan-lau}
K.~Chan and S.-C. Lau, \emph{Open {G}romov-{W}itten invariants and
  superpotentials for semi-{F}ano toric surfaces}, Int. Math. Res. Not. IMRN
  (2014), no.~14, 3759--3789. \MR{3239088, Zbl 06369457}

\bibitem{CLL}
K.~Chan, S.-C. Lau, and N.~C. Leung, \emph{S{YZ} mirror symmetry for toric
  {C}alabi-{Y}au manifolds}, J. Differential Geom. \textbf{90} (2012), no.~2,
  177--250. \MR{2899874, Zbl 1297.53061}

\bibitem{CLLT12}
K.~Chan, S.-C. Lau, N.~C. Leung, and H.-H. Tseng, \emph{Open {G}romov-{W}itten
  invariants, mirror maps, and {S}eidel representations for toric manifolds},
  preprint, \href{http://arxiv.org/abs/1209.6119}{arXiv: 1209.6119}.

\bibitem{CLT11}
K.~Chan, S.-C. Lau, and H.-H Tseng, \emph{Enumerative meaning of mirror maps
  for toric {C}alabi-{Y}au manifolds}, Adv. Math. \textbf{244} (2013),
  605--625. \MR{3077883, Zbl 1286.14056}

\bibitem{Chan-Leung}
K.~Chan and N.~C. Leung, \emph{Mirror symmetry for toric {F}ano manifolds via
  {SYZ} transformations}, Adv. Math. \textbf{223} (2010), no.~3, 797--839.
  \MR{2565550, Zbl 1201.14029}

\bibitem{Chan-Leung2}
\bysame, \emph{On {SYZ} mirror transformations}, New developments in algebraic
  geometry, integrable systems and mirror symmetry ({RIMS}, {K}yoto, 2008),
  Adv. Stud. Pure Math., vol.~59, Math. Soc. Japan, Tokyo, 2010, pp.~1--30.
  \MR{2683205, Zbl 1213.14073}

\bibitem{CR01}
W.~Chen and Y.~Ruan, \emph{Orbifold {G}romov-{W}itten theory}, Orbifolds in
  mathematics and physics ({M}adison, {WI}, 2001), Contemp. Math., vol. 310,
  Amer. Math. Soc., Providence, RI, 2002, pp.~25--85. \MR{1950941, Zbl
  1063.53091}

\bibitem{CR04}
\bysame, \emph{A new cohomology theory of orbifold}, Comm. Math. Phys.
  \textbf{248} (2004), no.~1, 1--31. \MR{2104605, Zbl 1063.53091}

\bibitem{cho06}
C.-H. Cho and Y.-G. Oh, \emph{Floer cohomology and disc instantons of
  {L}agrangian torus fibers in {F}ano toric manifolds}, Asian J. Math.
  \textbf{10} (2006), no.~4, 773--814. \MR{2282365, Zbl 1130.53055}

\bibitem{CP}
C.-H. Cho and M.~Poddar, \emph{Holomorphic orbidiscs and {L}agrangian {F}loer
  cohomology of symplectic toric orbifolds}, J. Differential Geom. \textbf{98}
  (2014), no.~1, 21--116. \MR{3263515, Zbl 1300.53077}

\bibitem{CS}
C.-H. Cho and H.-S. Shin, \emph{{C}hern-{W}eil {M}aslov index and its orbifold
  analogue}, preprint, \href{http://arxiv.org/abs/1202.0556}{arXiv: 1202.0556}.

\bibitem{Coates09}
T.~Coates, \emph{On the crepant resolution conjecture in the local case}, Comm.
  Math. Phys. \textbf{287} (2009), no.~3, 1071--1108. \MR{2486673, Zbl
  1200.53081}

\bibitem{CCIT_toricDM}
T.~Coates, A.~Corti, H.~Iritani, and H.-H. Tseng, \emph{A mirror theorem for
  toric stacks}, to appear in Compos. Math.,
  \href{http://arxiv.org/abs/1310.4163}{arXiv:1310.4163}.

\bibitem{CCIT09}
\bysame, \emph{Computing genus-zero twisted {G}romov-{W}itten invariants}, Duke
  Math. J. \textbf{147} (2009), no.~3, 377--438. \MR{2510741, Zbl 1176.14009}

\bibitem{CIT09}
T.~Coates, H.~Iritani, and H.-H. Tseng, \emph{Wall-crossings in toric
  {G}romov-{W}itten theory. {I}. {C}repant examples}, Geom. Topol. \textbf{13}
  (2009), no.~5, 2675--2744. \MR{2529944, Zbl 1184.53086}

\bibitem{Coates-Ruan}
T.~Coates and Y.~Ruan, \emph{Quantum cohomology and crepant resolutions: a
  conjecture}, Ann. Inst. Fourier (Grenoble) \textbf{63} (2013), no.~2,
  431--478. \MR{3112518, Zbl 1275.53083}

\bibitem{Cox-Katz}
D.~Cox and S.~Katz, \emph{Mirror symmetry and algebraic geometry}, Mathematical
  Surveys and Monographs, vol.~68, American Mathematical Society, Providence,
  RI, 1999. \MR{1677117, Zbl 0951.14026}

\bibitem{CLS_toricbook}
D.~Cox, J.~Little, and H.~Schenck, \emph{Toric varieties}, Graduate Studies in
  Mathematics, vol. 124, American Mathematical Society, Providence, RI, 2011.
  \MR{2810322, Zbl 1223.14001}

\bibitem{Doran-Kerr11}
C.~Doran and M.~Kerr, \emph{Algebraic {$K$}-theory of toric hypersurfaces},
  Commun. Number Theory Phys. \textbf{5} (2011), no.~2, 397--600. \MR{2851155,
  Zbl 1274.19003}

\bibitem{Doran-Kerr13}
\bysame, \emph{Algebraic cycles and local quantum cohomology}, Commun. Number
  Theory Phys. \textbf{8} (2014), no.~4, 703--727.

\bibitem{Efimov}
A.~Efimov, \emph{Homological mirror symmetry for curves of higher genus}, Adv.
  Math. \textbf{230} (2012), no.~2, 493--530. \MR{2914956, Zbl 1242.14039}

\bibitem{fang-liu}
B.~Fang and C.-C.~M. Liu, \emph{Open {G}romov-{W}itten invariants of toric
  {C}alabi-{Y}au 3-folds}, Comm. Math. Phys. \textbf{323} (2013), no.~1,
  285--328. \MR{3085667, Zbl 1284.14074}

\bibitem{fang-liu-tseng}
B.~Fang, C.-C.~M. Liu, and H.-H. Tseng, \emph{Open-closed {G}romov-{W}itten
  invariants of 3-dimensional {C}alabi-{Y}au smooth toric {D}{M} stacks},
  preprint, \href{http://arxiv.org/abs/1212.6073}{arXiv:1212.6073}.

\bibitem{FOOO10b}
K.~Fukaya, Y.-G. Oh, H.~Ohta, and K.~Ono, \emph{Lagrangian {F}loer theory and
  mirror symmetry on compact toric manifolds}, preprint,
  \href{http://arxiv.org/abs/1009.1648}{arXiv:1009.1648}.

\bibitem{FOOO12}
\bysame, \emph{Technical details on {K}uranishi structure and virtual
  fundamental chain}, preprint,
  \href{http://arxiv.org/abs/1209.4410}{arXiv:1209.4410}.

\bibitem{FOOO_book}
\bysame, \emph{Lagrangian intersection {F}loer theory: anomaly and
  obstruction}, AMS/IP Studies in Advanced Mathematics, vol.~46, American
  Mathematical Society, Providence, RI, 2009. \MR{2553465, Zbl 1181.53002}

\bibitem{FOOO1}
\bysame, \emph{Lagrangian {F}loer theory on compact toric manifolds. {I}}, Duke
  Math. J. \textbf{151} (2010), no.~1, 23--174. \MR{2573826, Zbl 1190.53078}

\bibitem{FOOO2}
\bysame, \emph{Lagrangian {F}loer theory on compact toric manifolds {II}: bulk
  deformations}, Selecta Math. (N.S.) \textbf{17} (2011), no.~3, 609--711.
  \MR{2827178, Zbl 1234.53023}

\bibitem{Fu}
W.~Fulton, \emph{Introduction to toric varieties}, Annals of Mathematics
  Studies, vol. 131, Princeton University Press, Princeton, NJ, 1993, The
  William H. Roever Lectures in Geometry. \MR{1234037, Zbl 0885.14025}

\bibitem{GKZ89}
I.~M. Gel'fand, M.~M. Kapranov, and A.~V. Zelevinsky, \emph{Hypergeometric
  functions and toric varieties}, Funktsional. Anal. i Prilozhen. \textbf{23}
  (1989), no.~2, 12--26. \MR{1011353, Zbl 0994.33501}

\bibitem{GKZ90}
\bysame, \emph{Generalized {E}uler integrals and {$A$}-hypergeometric
  functions}, Adv. Math. \textbf{84} (1990), no.~2, 255--271. \MR{1080980, Zbl
  0741.33011}

\bibitem{givental_imrn96}
A.~Givental, \emph{Equivariant {G}romov-{W}itten invariants}, Internat. Math.
  Res. Notices (1996), no.~13, 613--663. \MR{1408320, Zbl 0881.55006}

\bibitem{goldstein}
E.~Goldstein, \emph{Calibrated fibrations on noncompact manifolds via group
  actions}, Duke Math. J. \textbf{110} (2001), no.~2, 309--343. \MR{1865243,
  Zbl 1021.53030}

\bibitem{G-I11}
E.~Gonz{\'a}lez and H.~Iritani, \emph{Seidel elements and mirror
  transformations}, Selecta Math. (N.S.) \textbf{18} (2012), no.~3, 557--590.
  \MR{2960027, Zbl 1263.14054}

\bibitem{graber_pand}
T.~Graber and R.~Pandharipande, \emph{Localization of virtual classes}, Invent.
  Math. \textbf{135} (1999), no.~2, 487--518. \MR{1666787, Zbl 0953.14035}

\bibitem{graber-zaslow01}
T.~Graber and E.~Zaslow, \emph{Open-string {G}romov-{W}itten invariants:
  calculations and a mirror ``theorem''}, Orbifolds in mathematics and physics
  ({M}adison, {WI}, 2001), Contemp. Math., vol. 310, Amer. Math. Soc.,
  Providence, RI, 2002, pp.~107--121. \MR{1950943, Zbl 1085.14518}

\bibitem{gross_examples}
M.~Gross, \emph{Examples of special {L}agrangian fibrations}, Symplectic
  geometry and mirror symmetry ({S}eoul, 2000), World Sci. Publ., River Edge,
  NJ, 2001, pp.~81--109. \MR{1882328, Zbl 1034.53054}

\bibitem{Gross-Siebert03}
M.~Gross and B.~Siebert, \emph{Affine manifolds, log structures, and mirror
  symmetry}, Turkish J. Math. \textbf{27} (2003), no.~1, 33--60. \MR{1975331,
  Zbl 1063.14048}

\bibitem{Gross-Siebert06}
\bysame, \emph{Mirror symmetry via logarithmic degeneration data. {I}}, J.
  Differential Geom. \textbf{72} (2006), no.~2, 169--338. \MR{2213573, Zbl
  1107.14029}

\bibitem{Gross-Siebert10}
\bysame, \emph{Mirror symmetry via logarithmic degeneration data, {II}}, J.
  Algebraic Geom. \textbf{19} (2010), no.~4, 679--780. \MR{2669728, Zbl
  1209.14033}

\bibitem{gross07}
\bysame, \emph{From real affine geometry to complex geometry}, Ann. of Math.
  (2) \textbf{174} (2011), no.~3, 1301--1428. \MR{2846484, Zbl 1266.53074}

\bibitem{gross08}
\bysame, \emph{An invitation to toric degenerations}, Surveys in differential
  geometry. {V}olume {XVI}. {G}eometry of special holonomy and related topics,
  Surv. Differ. Geom., vol.~16, Int. Press, Somerville, MA, 2011, pp.~43--78.
  \MR{2893676, Zbl 1276.14057}

\bibitem{hosono06}
S.~Hosono, \emph{Central charges, symplectic forms, and hypergeometric series
  in local mirror symmetry}, Mirror symmetry. {V}, AMS/IP Stud. Adv. Math.,
  vol.~38, Amer. Math. Soc., Providence, RI, 2006, pp.~405--439. \MR{2282969,
  Zbl 1114.14025}

\bibitem{iritani09}
H.~Iritani, \emph{An integral structure in quantum cohomology and mirror
  symmetry for toric orbifolds}, Adv. Math. \textbf{222} (2009), no.~3,
  1016--1079. \MR{2553377, Zbl 1190.14054}

\bibitem{Jarvis-Kimura_BG}
T.~Jarvis and T.~Kimura, \emph{Orbifold quantum cohomology of the classifying
  space of a finite group}, Orbifolds in mathematics and physics ({M}adison,
  {WI}, 2001), Contemp. Math., vol. 310, Amer. Math. Soc., Providence, RI,
  2002, pp.~123--134. \MR{1950944, Zbl 1065.14069}

\bibitem{Jiang08}
Y.~Jiang, \emph{The orbifold cohomology ring of simplicial toric stack
  bundles}, Illinois J. Math. \textbf{52} (2008), no.~2, 493--514. \MR{2524648,
  Zbl 1231.14002}

\bibitem{KKOY09}
A.~Kapustin, L.~Katzarkov, D.~Orlov, and M.~Yotov, \emph{Homological mirror
  symmetry for manifolds of general type}, Cent. Eur. J. Math. \textbf{7}
  (2009), no.~4, 571--605. \MR{2563433, Zbl 1200.53079}

\bibitem{Ke-Zhou}
H.-Z. Ke and J.~Zhou, \emph{Quantum {M}c{K}ay correspondence for disc
  invariants of toric {C}alabi-{Y}au 3-orbifolds}, preprint,
  \href{http://arxiv.org/abs/1410.4376}{arXiv:1410.4376}.

\bibitem{konishi09}
Y.~Konishi and S.~Minabe, \emph{Local {B}-model and mixed {H}odge structure},
  Adv. Theor. Math. Phys. \textbf{14} (2010), no.~4, 1089--1145. \MR{2821394,
  Zbl 1229.81243}

\bibitem{kontsevich00}
M.~Kontsevich and Y.~Soibelman, \emph{Homological mirror symmetry and torus
  fibrations}, Symplectic geometry and mirror symmetry ({S}eoul, 2000), World
  Sci. Publ., River Edge, NJ, 2001, pp.~203--263. \MR{1882331, Zbl 1072.14046}

\bibitem{kontsevich-soibelman04}
\bysame, \emph{Affine structures and non-{A}rchimedean analytic spaces}, The
  unity of mathematics, Progr. Math., vol. 244, Birkh\"auser Boston, Boston,
  MA, 2006, pp.~321--385. \MR{2181810, Zbl 1114.14027}

\bibitem{Lau14}
S.-C. Lau, \emph{{G}ross-{S}iebert's slab functions and open {GW} invariants
  for toric {C}alabi-{Y}au manifolds}, to appear in Math. Res. Lett. (2014),
  \href{http://arxiv.org/abs/1405.3863}{arXiv:1405.3863}.

\bibitem{LLW10}
S.-C. Lau, N.~C. Leung, and B.~Wu, \emph{A relation for {G}romov-{W}itten
  invariants of local {C}alabi-{Y}au threefolds}, Math. Res. Lett. \textbf{18}
  (2011), no.~5, 943--956. \MR{2875867, Zbl 1239.14049}

\bibitem{LLW_surfaces}
\bysame, \emph{Mirror maps equal {SYZ} maps for toric {C}alabi-{Y}au surfaces},
  Bull. Lond. Math. Soc. \textbf{44} (2012), no.~2, 255--270. \MR{2914605, Zbl
  1239.14033}

\bibitem{leung01}
N.~C. Leung, \emph{Mirror symmetry without corrections}, Comm. Anal. Geom.
  \textbf{13} (2005), no.~2, 287--331. \MR{2154821 (2006c:32028), Zbl
  1086.32022}

\bibitem{LYZ}
N.~C. Leung, S.-T. Yau, and E.~Zaslow, \emph{From special {L}agrangian to
  {H}ermitian-{Y}ang-{M}ills via {F}ourier-{M}ukai transform}, Adv. Theor.
  Math. Phys. \textbf{4} (2000), no.~6, 1319--1341. \MR{1894858, Zbl
  1033.53044}

\bibitem{liu_HIM}
C.-C.~M. Liu, \emph{Localization in {G}romov-{W}itten {T}heory and {O}rbifold
  {G}romov-{W}itten {T}heory}, {H}andbook of {M}oduli, {V}olume {II}, Adv.
  Lect. Math. (ALM), vol.~25, International {P}ress and {H}igher {E}ducation
  {P}ress, 2013, pp.~353--425. \MR{3184181, Zbl 1260.14002}

\bibitem{McDuff-Wehrheim12}
D.~McDuff and K.~Wehrheim, \emph{Smooth {K}uranishi atlases with trivial
  isotropy}, preprint, \href{http://arxiv.org/abs/1208.1340}{arXiv:1208.1340}.

\bibitem{Ruan06}
Y.~Ruan, \emph{The cohomology ring of crepant resolutions of orbifolds},
  Gromov-{W}itten theory of spin curves and orbifolds, Contemp. Math., vol.
  403, Amer. Math. Soc., Providence, RI, 2006, pp.~117--126. \MR{2234886, Zbl
  1105.14078}

\bibitem{Ruddat-Siebert14}
H.~Ruddat and B.~Siebert, \emph{Canonical coordinates in toric degenerations},
  preprint, \href{http://arxiv.org/abs/1409.4750}{arXiv:1409.4750}.

\bibitem{syz96}
A.~Strominger, S.-T. Yau, and E.~Zaslow, \emph{Mirror symmetry is
  {$T$}-duality}, Nuclear Phys. B \textbf{479} (1996), no.~1-2, 243--259.
  \MR{1429831, Zbl 0896.14024}

\bibitem{orbQRR}
H.-H. Tseng, \emph{Orbifold quantum {R}iemann-{R}och, {L}efschetz and {S}erre},
  Geom. Topol. \textbf{14} (2010), no.~1, 1--81. \MR{2578300, Zbl 1178.14058}

\end{thebibliography}

\end{document}